\documentclass[11pt]{article}
\pdftrailerid{}

\usepackage[T1]{fontenc}
\usepackage{lmodern}
\usepackage{microtype}
\usepackage[margin=1in]{geometry}
\usepackage{amsmath,amssymb,amsthm,mathtools}
\usepackage{booktabs,array,tabularx,longtable}
\usepackage{algorithm,algpseudocode}
\usepackage{xcolor}
\usepackage{graphicx}
\usepackage{tikz}
\usetikzlibrary{arrows.meta,calc,positioning,fit,backgrounds}
\usepackage[numbers,sort&compress]{natbib}
\usepackage{url}
\usepackage[hidelinks,pdfauthor={Mohit Sinha},
pdftitle={Finite-Precision Symmetric Krylov Methods}]{hyperref}
\usepackage[font=small,labelfont=bf,labelsep=period]{caption}
\algrenewcommand\alglinenumber[1]{\scriptsize #1}
\algrenewcommand\algorithmicindent{1.25em}

\definecolor{navy}{RGB}{29,61,104}
\definecolor{teal}{RGB}{23,109,103}
\definecolor{softgray}{RGB}{245,246,248}

\newtheorem{theorem}{Theorem}
\newtheorem{proposition}{Proposition}
\newtheorem{lemma}{Lemma}
\newtheorem{corollary}{Corollary}
\newtheorem{assumption}{Assumption}
\theoremstyle{definition}

\newtheorem{remark}{Remark}

\newcommand{\fl}{\operatorname{fl}}
\newcommand{\diag}{\operatorname{diag}}
\newcommand{\dist}{\operatorname{dist}}
\newcommand{\col}{\operatorname{col}}
\newcommand{\R}{\mathbb{R}}
\newcommand{\eps}{\varepsilon}
\title{Finite-Precision Symmetric Krylov Methods:\\
Exact Rounding Examples, Block Paige Identities,\\
and a Variable-Block Lanczos Model}
\author{Mohit Sinha}
\date{}

\begin{document}
\maketitle

\begin{abstract}
Short-recurrence Krylov methods that are equivalent in exact arithmetic often
diverge in floating-point arithmetic. To illustrate this, we provide an exactly
representable two-cycle for steepest descent with a recursively updated
residual. A complementary convergence theorem gives a sufficient condition
under which the stored residual decreases geometrically. We then show that a
second positive definite family yields different outcomes for a Hestenes--Stiefel
Conjugate Gradient (CG) implementation and a direct Lanczos--Galerkin approach.
CG finds the exact solution after four updates, whereas the two-step projected
system becomes inconsistent. Under stated perturbation and transfer hypotheses,
a common spectral enclosure gives comparable convergence bounds. For a dense,
left-to-right evaluation order, we prove a sufficient precision bound for a
prescribed backward error within $n$ updates for an $n\times n$ matrix, together with a
computable stopping test and specified exponent-range assumptions. We use
polynomial estimates and numerical experiments to examine the trade-off
between working precision and the number of CG iterations needed to achieve
a prescribed backward error. We also compare results on inexact matrix-vector products and
preconditioning with the frameworks of Paige and Greenbaum.

For block Lanczos, we analyze Householder orthogonalization and singular-value
truncation as the block size changes. Under componentwise and normwise error
bounds, the computed coefficients satisfy a controlled local recurrence and an
exact block Lanczos relation for a nearby symmetric problem in a larger space.
A block form of Paige's identity bounds the overlap with a Ritz vector along
its residual coordinate direction. An inter-block recurrence describes the
evolution of overlap, and a Gram-matrix argument gives a count of additional
nearby Ritz values. We use physical residuals and numerical rank to distinguish
independent directions from repeated approximations to directions already
represented. Together, these results connect local rounding errors
and changing block rank to the spectral information carried by the computed
vectors.
\end{abstract}

\section{Introduction}
\label{sec:introduction}

Krylov subspace methods approximate the solution of large linear systems by
projecting the problem onto low-dimensional spaces generated by repeated
matrix-vector multiplication. Consider a real linear system $Ax=b$, where
$A\in\R^{n\times n}$ is nonsingular and $b\in\R^n$.
Given an initial approximation $x_0\in\R^n$, define the residual
$r_0=b-Ax_0$ and assume $r_0\ne0$. For a positive integer $m\le n$, the
$m$th Krylov subspace is
\begin{equation}
 \mathcal K_m(A,r_0)
 =\operatorname{span}\{r_0,Ar_0,\ldots,A^{m-1}r_0\}.
 \label{eq:krylov-space}
\end{equation}
Suppose that $\dim\mathcal K_m(A,r_0)=m$. The Arnoldi process constructs an
orthonormal basis $V_m=[v_1,\ldots,v_m]\in\R^{n\times m}$, with
$v_1=r_0/\|r_0\|_2$. Throughout, vectors are real columns,
$\langle x,y\rangle=x^Ty$ is the Euclidean inner product, and $\|\cdot\|_2$
denotes the Euclidean norm for vectors and the spectral norm for matrices.
We write $I_d$ for the identity matrix of order $d$. Thus $V_m^TV_m=I_m$.
For $1\le j\le m$, let $e_j\in\R^m$ denote the $j$th coordinate vector.

A Galerkin approximation is a vector $x_m=x_0+V_my_m$, with
$y_m\in\R^m$, whose residual is orthogonal to $\mathcal K_m(A,r_0)$.
Its coefficients satisfy the projected system
\begin{equation}
 V_m^T(b-Ax_m)=0,
 \qquad (V_m^TAV_m)y_m=\|r_0\|_2e_1.
 \label{eq:galerkin-intro}
\end{equation}
For symmetric $A$, the upper Hessenberg matrix from Arnoldi becomes a
symmetric tridiagonal matrix $T_m=V_m^TAV_m$. In exact arithmetic,
\begin{equation}
 AV_m=V_mT_m+\beta_{m+1}v_{m+1}e_m^T,
 \qquad T_m=V_m^TAV_m,
 \label{eq:lanczos-relation}
\end{equation}
where $\beta_{m+1}=\|(I_n-V_mV_m^T)Av_m\|_2$ and, when this scalar is
positive, $v_{m+1}=(I_n-V_mV_m^T)Av_m/\beta_{m+1}$.
At breakdown the final product is zero. The tridiagonal structure gives the
Lanczos three-term recurrence, allowing the next basis vector to be computed
from the two preceding vectors \cite{Paige1976}.
The same projection approximates eigenvalues. If $T_my=\theta y$ for a
nonzero $y\in\R^m$, then $\theta$ is a Ritz value and $V_my$ is its Ritz
vector in the original space.

When $A$ is symmetric positive definite, the Conjugate Gradient (CG) method
of Hestenes and Stiefel produces the same Galerkin approximations in exact
arithmetic \cite{HestenesStiefel1952,Saad2003}. Its search directions $p_j$
are conjugate with respect to $A$, meaning that $p_i^TAp_j=0$ for $i\ne j$.
The recurrence implicitly factors the Lanczos matrix as $LDL^T$, with $L$
unit lower triangular and $D$ diagonal \cite{PaigeSaunders1975}.

Floating-point rounding changes this equivalence. Although CG and direct
Lanczos--Galerkin express the same projection in exact arithmetic, one
implementation updates a residual and a search direction while the other
normalizes basis vectors and computes tridiagonal coefficients, so the
rounding errors enter different operations and propagate through different
recurrences. The stored residual can drift from $b-Ax_m$. The Lanczos
vectors can lose mutual orthogonality. With block vectors, a residual block
may lose numerical rank, so the next can have fewer columns.

Problems 2.15--2.19 in Section 2.6 of the Simons workshop report
\cite{OpenProblems2026}, recorded by Chen and Greenbaum, concern these
effects. They ask when a recursively updated residual becomes small and how
the computed CG and Lanczos iterations are related. The precision required
for a backward-error guarantee within $n$ steps is a separate question.
The final two problems concern inexact operator applications and the extension
of Paige's analysis to block Lanczos, including block orthogonalization and
deflation. The results below address these questions through rounding
constructions, convergence bounds, and a spectral interpretation of the
computed block recurrence.

The analysis of such effects has a long history. Paige's local rounding-error
relations connect Ritz convergence to loss of orthogonality
\cite{Paige1976,Paige1980}, showing how an accurate eigenvalue approximation
can emerge from a basis whose columns are losing orthogonality and providing
the basis for the selective and partial reorthogonalization strategies of
Parlett and Scott and of Simon \cite{ParlettScott1979,Simon1984Reorth}.
Cullum and Willoughby studied spurious Ritz values
\cite{CullumWilloughby1980,CullumWilloughby1985}, and W\"ulling analyzed
the stabilization of Ritz clusters \cite{Wulling2005,Wulling2005BIT}.
Parlett and Reid analyzed how to track Lanczos convergence
\cite{ParlettReid1981}, and Parlett presented the associated symmetric
eigenvalue theory \cite{Parlett1998}.
Augmented formulations subsequently made the relation between computed and
orthogonal bases more precise
\cite{Paige2010Augmented,Paige2019,Paige2024,ChangPaigeTitleyPeloquin2021}.

Greenbaum's backward-error analysis gives another interpretation
\cite{Greenbaum1989}. Under local perturbation hypotheses, the computed
recurrence can be compared with an exact recurrence for a larger symmetric
matrix whose eigenvalues form clusters near those of the original matrix,
thereby changing the polynomial approximation problem that governs convergence
\cite{GreenbaumStrakos1992,Strakos1991,Notay1993,GergelitsStrakos2014}.
Convergence can consequently be delayed.
The residual gap has its own error accumulation, studied by Greenbaum
\cite{Greenbaum1997Residual} and in analyses of residual replacement and
coupled recurrences \cite{SleijpenVanderVorst1996,VanderVorstYe2000,
SleijpenVanderVorstModersitzki2000,GutknechtStrakos2000}.
Meurant and Strako\v{s} \cite{MeurantStrakos2006} examine these developments
together. Local orthogonality also supports error estimation in CG
\cite{StrakosTichy2002,StrakosTichy2005,MeurantTichy2024}.

For a symmetric positive definite matrix $A$, rounding errors of size
$u\|A\|_2$ are comparable with its smallest eigenvalue when $u\kappa_2(A)$
is of order one, where $\kappa_2(A)$ is the condition number in the $2$-norm
and $u$ is the unit roundoff. A symmetric perturbation can destroy positive
definiteness once its norm reaches the smallest eigenvalue.
Convergence also depends on the accumulation of errors
over successive steps. An iteration may converge accurately after more
steps than its exact-arithmetic counterpart, even though its tridiagonal
coefficients differ substantially from the exact ones. The sensitivity
of these coefficients is related to the underlying moment problem
\cite{Knizhnerman1996,ChenTrogdon2024}.

The steepest-descent construction has an exactly representable two-cycle
at $u\kappa_2(A)=8$.
The stored and true residuals agree throughout that cycle. A complementary
theorem proves geometric decrease of the stored residual under a quantitative
smallness condition, with a sharper bound for diagonal matrices.
For CG and direct Lanczos--Galerkin, a second positive definite family leads
to an exact solution after four CG updates and an inconsistent two-step
Galerkin system, with Paige's local identity identifying the overlap that
allows the projected system to become singular even though the original
matrix is positive definite. A common spectral enclosure gives a
complementary comparison bound. Its hypotheses specify the enlarged exact
models and the transfer to computed quantities.

For the dense evaluation order studied here, we derive sufficient
significand precision for a prescribed
backward error within $n$ updates and give a computable stopping test.
An interpolation estimate describes the difficulty of controlling a degree-$n$
polynomial on small spectral intervals. Variable-precision experiments show
how allowing additional iterations changes the precision found for particular
matrices. These results complement finite-precision matrix-function bounds
\cite{DruskinKnizhnerman1991,DruskinGreenbaumKnizhnerman1998,
MuscoMuscoSidford2018,ChenGreenbaumMuscoMusco2022} and recent analyses of
Krylov solvers in bit-complexity and mixed-precision models.
Inexact products lead to a related perturbation analysis, reviewed in
Section~\ref{sec:inexact}.

Block Lanczos methods for multiple and clustered eigenvalues were developed
by Cullum and Donath,
Underwood, and Golub and Underwood
\cite{CullumDonath1974,Underwood1975,GolubUnderwood1977}.
Band formulations and rank-deflation methods allow dependent directions to
be removed as the iteration proceeds
\cite{Ruhe1979,BaiDayYe1999,AliagaBoleyFreundHernandez2000,
FreundMalhotra1997,Baglama2000}.
Grimes, Lewis, and Simon used inter-block overlap recurrences to guide partial
reorthogonalization \cite{GrimesLewisSimon1994}.
More recently, \v{S}imonov\'a and Tich\'y developed a block continuation of
Greenbaum's model under sufficient conditions and distinguished proper from
improper Ritz clusters \cite{SimonovaTichy2025}, providing a comparison
for the variable-block construction below and identifying the need to relate
block Ritz convergence to the evolving loss of orthogonality.

We analyze block Lanczos with Householder orthogonalization and a compact
singular value decomposition (SVD) to determine the rank of each new block.
The local error bound remains applicable when the next block has fewer
columns. After normalizing the computed blocks, we construct an exact block
Lanczos relation for a nearby symmetric problem in a larger space.
To interpret the computed Ritz values, we first examine the exact-arithmetic
projection. Decomposing the residual into its components in the existing
space and orthogonal to that space gives a Gram identity for the rank of
the next block. A separate exact-arithmetic construction shows how repeated
Ritz values can correspond to independent vectors, in agreement with the
matrix-measure and block quadrature viewpoint
\cite{GolubMeurant2010,FenuMartinReichelRodriguez2013}.

In finite precision, newly computed blocks can lose orthogonality to
earlier Ritz vectors. The block Paige identity bounds the component of
this overlap along the residual coordinate direction. An inter-block
recurrence describes how the overlap evolves. Under a stated residual
bound, singular vectors of the overlap matrix give a trial subspace from
which we obtain a lower bound on the number of nearby Ritz values.
We then use physical residuals and numerical rank to assess how many of
the associated Ritz vectors are independent in the original space.
In the matrix-free Hessian example, sixteen vectors pass a residual test
and form a matrix of numerical rank five. A rank-aware recurrence and
established eigensolvers return sixteen independent directions.
The analysis explains how a cluster of computed Ritz values can contain
repeated approximations to the same directions in the original space.

\section{Floating-point model and algorithms}
\label{sec:arithmetic-model}

We now specify the arithmetic and the order of operations. These choices
determine the computed recurrences. Lanczos and CG coefficients receive
different labels whenever both occur in the same calculation
\cite[p.~200]{Saad2003}.

Let $u$ denote the unit roundoff. For a nonsingular matrix $A$, let
$\kappa_2(A)=\|A\|_2\|A^{-1}\|_2$ denote its condition number in the $2$-norm.
For symmetric positive definite $A$, this is the ratio of its largest
eigenvalue to its smallest eigenvalue.

Let $p\ge2$ be the number of bits in a binary significand, including the
leading bit. Then $u=2^{-p}$. The exact constructions use rounding to
nearest with ties to even. Write $\mathbb F\subset\R$ for the representable
numbers and $\fl$ for a rounded operation. Every scalar assignment and vector
component is stored in the target format before the next operation.
Products in vector updates are rounded before addition or subtraction.
Consequently, algebraically equivalent rearrangements can produce different
iterates, as Proposition~\ref{prop:cg-curvature-loss} illustrates.
For operations in the normal range, the relative error model is
\begin{equation}
 \fl(x\circ y)=(x\circ y)(1+\delta),\qquad
 |\delta|\le u,\qquad \circ\in\{+,-,\times,/\}.
 \label{eq:relative-rounding-model}
\end{equation}
Exact zero results are allowed. Statements about finite exponent ranges
specify their additional assumptions.
For a positive integer $s$ with $su<1$, define $\gamma_s=su/(1-su)$.
If $x,y\in\R^s$, an inner product accumulated from left to right satisfies
\begin{equation}
 |\fl(\langle x,y\rangle)-\langle x,y\rangle|
 \le\gamma_s|x|^T|y|,
 \label{eq:dot-error-model}
\end{equation}
where absolute values are componentwise \cite[Chap.~3]{Higham2002}.

For an approximate solution $\widehat x\in\R^n$, the normwise relative
backward error is
\begin{equation}
 \eta(\widehat x)=\frac{\|b-A\widehat x\|_2}
 {\|A\|_2\|\widehat x\|_2+\|b\|_2}.
 \label{eq:backward-error}
\end{equation}
This is the smallest relative normwise perturbation of $A$ and $b$ that
makes $\widehat x$ an exact solution \cite{RigalGaches1967,Higham2002}.
When both numerator and denominator vanish, we set $\eta(\widehat x)=0$.
The true residual is $b- A\widehat x$, evaluated over the real numbers.
The symbol $r_j$ denotes the stored vector generated by the recurrence.
Their difference is the residual gap.

\begin{algorithm}[htbp]
\caption{Steepest descent with a recursively updated residual}
\label{alg:sd}
\begin{algorithmic}[1]
\Statex Given $A$, $b$, and $x_0\in\mathbb F^n$.
\State $r_0\gets\fl(b-Ax_0)$
\For{$j=0,1,\ldots$}
  \If{$r_j=0$}
    \State \Return $x_j$
  \EndIf
  \State $q_j\gets\fl(Ar_j)$
  \State $a_j\gets\displaystyle\fl\!\left(
    \frac{\fl(\langle r_j,r_j\rangle)}{\fl(\langle q_j,r_j\rangle)}\right)$
  \State $x_{j+1}\gets\fl(x_j+a_jr_j)$
  \State $r_{j+1}\gets\fl(r_j-a_jq_j)$
\EndFor
\Statex Each matrix-vector product and inner product is accumulated in increasing
component order. The step is defined when its denominator is nonzero and
the stated arithmetic assumptions hold.
\end{algorithmic}
\end{algorithm}

The Hestenes--Stiefel CG recurrence \cite{HestenesStiefel1952,Saad2003}
starts from $r_0=\fl(b-Ax_0)$ and $p_0=r_0$. At a step with $r_j\ne0$,
it evaluates
\begin{equation}
\begin{aligned}
 q_j&=\fl(Ap_j),\\
 a_j^{\rm cg}&=\fl\!\left(
       \frac{\fl(\langle r_j,r_j\rangle)}
            {\fl(\langle q_j,p_j\rangle)}\right),\\
 x_{j+1}&=\fl(x_j+a_j^{\rm cg}p_j),\qquad
 r_{j+1}=\fl(r_j-a_j^{\rm cg}q_j),\\
 b_{j+1}^{\rm cg}&=\fl\!\left(
       \frac{\fl(\langle r_{j+1},r_{j+1}\rangle)}
            {\fl(\langle r_j,r_j\rangle)}\right),\\
 p_{j+1}&=\fl(r_{j+1}+b_{j+1}^{\rm cg}p_j).
\end{aligned}
\label{eq:hscg}
\end{equation}
An exactly zero residual terminates the CG iteration. Algorithm~\ref{alg:tested-cg}
below adds a true-residual stopping test to this recurrence.

For direct Lanczos--Galerkin, set $v_0=0$ and compute
\[
 \rho_0=\fl(\langle r_0,r_0\rangle),\qquad
 \beta_1^{\rm lan}=\fl(\sqrt{\rho_0}),\qquad
 v_1=\fl(r_0/\beta_1^{\rm lan}).
\]
For $j=1,2,\ldots$, compute
\begin{equation}
\begin{aligned}
 q_j&=\fl(Av_j),\\
 w_j&=\fl(q_j-\beta_j^{\rm lan}v_{j-1}),\\
 \alpha_j^{\rm lan}&=\fl(\langle v_j,w_j\rangle),\\
 z_j&=\fl(w_j-\alpha_j^{\rm lan}v_j),\\
 \rho_j&=\fl(\langle z_j,z_j\rangle),\\
 \beta_{j+1}^{\rm lan}&=\fl(\sqrt{\rho_j}),\\
 v_{j+1}&=\fl(z_j/\beta_{j+1}^{\rm lan}).
\end{aligned}
\label{eq:scalar-lanczos-order}
\end{equation}
A zero $\beta_{j+1}^{\rm lan}$ ends the recurrence before the division.
The symmetric tridiagonal matrix $T_m$ has diagonal entries
$\alpha_1^{\rm lan},\ldots,\alpha_m^{\rm lan}$ and off-diagonal entries
$\beta_2^{\rm lan},\ldots,\beta_m^{\rm lan}$.
This is the conventional ordering analyzed by Paige
\cite{Paige1976,Paige1980}. The direct Galerkin approximation solves
$T_my=\beta_1^{\rm lan}e_1$ and forms $\fl(x_0+V_my)$
\cite[Secs.~6.6--6.7]{Saad2003}. Writing each rounded assignment separately
makes the evaluation order unambiguous.
\section{Steepest descent in finite precision}

Algorithm~\ref{alg:sd} generates a recursively updated residual that can
diverge from the true residual $b-Ax_j$.  We first examine an exactly
representable system for which the two residuals agree at every step.
The iteration enters a periodic orbit.

\begin{theorem}[A two-cycle for steepest descent with a recursive residual]
\label{thm:sd-cycle}
Fix a binary floating-point format with significand precision $p\ge2$,
rounding to nearest with ties to even, and unit roundoff $u=2^{-p}$.  Define
\[
 a=\frac{u}{8},\qquad X=\frac{8}{u},
\]
and assume that $a$ and $X$ are representable normal numbers.  Consider the
linear system
\[
 A=\diag(a,1),\qquad b=(2,1)^T,\qquad x_0=(X,0)^T.
\]
Let $x_j$ and $r_j$ denote the stored iterate and residual generated by
Algorithm~\ref{alg:sd}, with the operations evaluated in the displayed
order.  Then the iteration is well defined for every $j\ge0$, and
\[
\begin{aligned}
 x_{2j}&=(X,0)^T,   & r_{2j}&=(1,1)^T,\\
 x_{2j+1}&=(X,2)^T, & r_{2j+1}&=(1,-1)^T.
\end{aligned}
\]
Furthermore,
\[
 r_j=b-Ax_j,
 \qquad
 \|b-Ax_j\|_2=\sqrt{2}
 \quad (j\ge0),
\]
where $b-Ax_j$ is evaluated in exact real arithmetic.  Both residuals
therefore have constant nonzero norm.
\end{theorem}

\begin{proof}
The representable numbers immediately adjacent to one are $1-u$ and
$1+2u$.  Its rounding interval extends downward by $u/2$ and upward by $u$.
The choices $a=u/8$ and $2a=u/4$ place $1+a$ and $1-2a$ strictly
inside these midpoint boundaries, giving
\[
 \fl(1+a)=1,\qquad \fl(1-2a)=1,
\]
Rounding to nearest
commutes with sign reversal.  At $X=2^{p+3}$ the successor spacing is $16$,
so $\fl(X+2)=X$.

Since $aX=1$ exactly, the initial stored residual is $(1,1)^T$.  Suppose
$s\in\{1,-1\}$ and
\[
 x=(X,1-s)^T,\qquad r=(1,s)^T.
\]
The matrix-vector product $q=Ar=(a,s)^T$ is exact.  The stored inner products
evaluate to
\[
 \fl(r^Tr)=2,
 \qquad
 \fl(r^Tq)=\fl(a+1)=1.
\]
The stored step size is consequently $2$.  The componentwise updates give
\[
 \fl(x+2r)=(X,1+s)^T,
 \qquad
 \fl(r-2q)=(1,-s)^T.
\]
One iteration exchanges the two states, which proves the cycle by induction.
Furthermore,
\[
 b-A(X,1-s)^T=(2-aX,1-(1-s))^T=(1,s)^T,
\]
establishing equality of the stored and true residuals and the asserted norm.
\end{proof}

The alternating directions recall the classical two-limit behavior of
steepest descent \cite[Sec.~2.7]{OpenProblems2026}.  Here the recurrence
forms an exact floating-point cycle at $\kappa_2(A)u=8$, and the residual
gap is zero.  Bollen's convergence analyses impose small-roundoff
conditions \cite{Bollen1979,Bollen1984}, while Greenbaum's attainable-accuracy
analysis describes the accumulated difference between true and recursively
updated residuals \cite{Greenbaum1997Residual}.  Theorem~\ref{thm:sd-cycle}
gives the nonconvergent example requested in Problem~2.15.
We next give a sufficient condition for geometric decrease of the stored
residual.
\subsection{Geometric decrease of the recursively updated residual}
\label{sec:sd-positive}

In the two-cycle, the correction in the first component is smaller than
half the distance to the next representable iterate.  The relation
$\kappa_2(A)u=8$ places this example at a very different scale from the
small-roundoff assumptions used in convergence analyses
\cite{Bollen1984}.  We now give a quantitative condition under which the
recursively updated residual decreases geometrically.  The argument first
bounds the error in one step and then applies the exact steepest-descent
contraction.

Throughout this subsection, $A\in\mathbb F^{n\times n}$ is symmetric positive
definite. Let $\lambda_{\min}$ and $\lambda_{\max}$ denote its smallest and
largest eigenvalues.
We use the norms
\[
 \|x\|_{A^{-1}}=(x^TA^{-1}x)^{1/2},
 \qquad
 \|x\|_A=(x^TAx)^{1/2}.
\]
For a nonzero residual $r$, the exact steepest-descent coefficient is
\[
 \tilde a(r)=\frac{\langle r,r\rangle}{\langle Ar,r\rangle}.
\]
The exact residual after the step is $(I-\tilde a(r)A)r$.
We will compare the computed residual with this vector.

To state the rounding-error bound, define
\[
 c_A=\frac{\|\,|A|\,\|_2}{\|A\|_2}\in[1,\sqrt n\,],
\]
where $|A|$ is the matrix of absolute values. Inner products and the rows
of a matrix-vector product are accumulated in any fixed order. Under the
model in Section~\ref{sec:arithmetic-model}, their errors satisfy
\[
 |\fl(Ar)-Ar|\le\gamma_n|A||r|,
 \qquad
 |\fl(\langle x,y\rangle)-\langle x,y\rangle|
 \le\gamma_n|x|^T|y|,
\]
componentwise for the first inequality \cite[Sec.~3.5]{Higham2002}.

\begin{lemma}[Rounding error in one steepest-descent step]
\label{lem:sd-one-step}
Let $r\in\mathbb F^n$ be nonzero and suppose
$c_A\gamma_n\kappa_2(A)\le1/16$.
Compute the matrix-vector product, step coefficient, and next residual by
\[
\begin{aligned}
 q&=\fl(Ar),\\
 a&=\fl\!\left(
       \frac{\fl(\langle r,r\rangle)}{\fl(\langle q,r\rangle)}
       \right),\\
 r_i^+&=\fl\bigl(r_i-\fl(aq_i)\bigr),\qquad 1\le i\le n.
\end{aligned}
\]
Assume that these operations incur neither underflow nor overflow.
Then $\fl(\langle q,r\rangle)>0$ and $a>0$.
There is a vector $f\in\R^n$ such that
\begin{equation}
 r^+=(I-\tilde a(r)A)\,r+f,
 \qquad
 \|f\|_{A^{-1}}\le14\,c_A\gamma_n\kappa_2(A)\,\|r\|_{A^{-1}} .
 \label{eq:sd-perturbed-step}
\end{equation}
\end{lemma}

\begin{proof}
Put $\tilde a=\tilde a(r)$ and
$\eta_0=c_A\gamma_n\kappa_2(A)\le1/16$.
We first estimate the relative error in the computed coefficient $a$.
The matrix-vector product has the form $q=Ar+\delta q$, with
\[
 \|\delta q\|_2\le\gamma_n c_A\|A\|_2\|r\|_2.
\]
Since $\langle Ar,r\rangle\ge\lambda_{\min}\|r\|_2^2$,
\[
 \langle q,r\rangle=\langle Ar,r\rangle(1+\eta),
 \qquad
 |\eta|\le
 \frac{\gamma_n c_A\|A\|_2\|r\|_2^2}{\langle Ar,r\rangle}
 \le\eta_0.
\]

The first inner product is a sum of nonnegative terms, so
\[
 \fl(\langle r,r\rangle)=\langle r,r\rangle(1+\theta_1),
 \qquad |\theta_1|\le\gamma_n.
\]
For the second inner product, the error bound and the preceding estimate
give
\[
\begin{aligned}
 |\fl(\langle q,r\rangle)-\langle q,r\rangle|
   &\le\gamma_n(1+\eta_0)\|A\|_2\|r\|_2^2,\\
 \langle q,r\rangle
   &\ge(1-\eta_0)\lambda_{\min}\|r\|_2^2.
\end{aligned}
\]
Thus $\fl(\langle q,r\rangle)=\langle q,r\rangle(1+\theta_2)$, where
\[
 |\theta_2|
 \le\gamma_n\kappa_2(A)\frac{1+\eta_0}{1-\eta_0}
 \le2\eta_0\le\frac18.
\]
In particular, the computed denominator is positive.

Let $\delta_3$ be the relative error in the division, so that
$|\delta_3|\le u$. The computed coefficient satisfies
\[
 a=\tilde a\,(1+\psi),
 \qquad
 1+\psi=\frac{(1+\theta_1)(1+\delta_3)}{(1+\eta)(1+\theta_2)} .
\]
The numerator in this expression minus its denominator is
$\theta_1+\delta_3+\theta_1\delta_3-\eta-\theta_2-\eta\theta_2$.
Its absolute value is at most $5\eta_0+3\eta_0^2\le5.2\eta_0$,
using $u\le\gamma_n\le\eta_0$.
The denominator is at least $(15/16)(7/8)$. It follows that
\[
 |\psi|\le
 \frac{5.2\eta_0}{(15/16)(7/8)}
 <6.35\eta_0=:\psi_0\le0.4.
\]
Hence $a>0$.

There are two rounded operations in each component of the residual update.
Writing their relative errors as $\delta_i$ and $\delta_i'$ gives
\[
 r_i^+=\bigl(r_i-aq_i(1+\delta_i)\bigr)(1+\delta_i'),
 \qquad |\delta_i|,|\delta_i'|\le u.
\]
The update error $e=r^+-(r-aq)$ therefore satisfies
\[
 |e_i|\le u|r_i|+3u\,a|q_i|,
 \qquad
 \|e\|_2\le u\|r\|_2+3u\,a\|q\|_2.
\]
Substituting $q=Ar+\delta q$ and $a=\tilde a(1+\psi)$ into $r^+$
gives the residual identity in \eqref{eq:sd-perturbed-step}, with
\[
 f=-\psi\,\tilde aAr-(1+\psi)\,\tilde a\,\delta q+e .
\]

It remains to bound $f$ in the $A^{-1}$-norm. The Cauchy--Schwarz
inequality and the bounds on the extreme eigenvalues give
\begin{equation}
\begin{aligned}
 \langle r,r\rangle&\le\|r\|_A\|r\|_{A^{-1}},\\
 \|r\|_2&\le\lambda_{\max}^{1/2}\|r\|_{A^{-1}},
 \qquad
 \|r\|_2\le\lambda_{\min}^{-1/2}\|r\|_A.
\end{aligned}
 \label{eq:three-inequalities}
\end{equation}
Consequently,
\begin{equation}
\begin{aligned}
 \tilde a\,\|r\|_A
   &=\frac{\langle r,r\rangle}{\|r\|_A}
     \le\|r\|_{A^{-1}},\\
 \tilde a\,\|r\|_2
   &=\frac{\|r\|_2^3}{\|r\|_A^2}
     \le\frac{\|r\|_2\,\|r\|_{A^{-1}}}{\|r\|_A}
     \le\lambda_{\min}^{-1/2}\|r\|_{A^{-1}}.
\end{aligned}
 \label{eq:a-tilde-bounds}
\end{equation}
For the first two terms of $f$, these inequalities give
\[
\begin{aligned}
 \|\tilde aAr\|_{A^{-1}}
   &=\tilde a\|r\|_A\le\|r\|_{A^{-1}},\\
 \|\tilde a\,\delta q\|_{A^{-1}}
   &\le\lambda_{\min}^{-1/2}\tilde a\|\delta q\|_2
     \le\eta_0\|r\|_{A^{-1}}.
\end{aligned}
\]
For the update error, use
$\|e\|_{A^{-1}}\le\lambda_{\min}^{-1/2}\|e\|_2$.
Its two contributions satisfy
\[
\begin{aligned}
 \lambda_{\min}^{-1/2}u\|r\|_2
   &\le u\sqrt{\kappa_2(A)}\,\|r\|_{A^{-1}},\\
 \lambda_{\min}^{-1/2}3u\,a\|q\|_2
   &\le3u(1+\psi_0)(1+\eta_0)\kappa_2(A)\|r\|_{A^{-1}}\\
   &\le4.5u\kappa_2(A)\|r\|_{A^{-1}}.
\end{aligned}
\]
The second estimate uses
$a\|q\|_2\le(1+\psi_0)(1+\eta_0)\tilde a\|A\|_2\|r\|_2$.
Finally, $u\sqrt{\kappa_2(A)}\le u\kappa_2(A)\le\eta_0$, so adding
the four contributions yields
\[
 \|f\|_{A^{-1}}
 \le(6.35+1.4+1+4.5)\eta_0\|r\|_{A^{-1}}
 <14\eta_0\|r\|_{A^{-1}}.
\]
\end{proof}

\begin{theorem}[Convergence of the recursively updated residual of steepest descent]
\label{thm:sd-positive}
Let $A\in\mathbb F^{n\times n}$ be symmetric positive definite and suppose
\begin{equation}
 14\,c_A\,\gamma_n\,\kappa_2(A)\bigl(\kappa_2(A)+1\bigr)\le1 .
 \label{eq:sd-condition}
\end{equation}
Execute Algorithm~\ref{alg:sd} in the standard model with any fixed
accumulation order. Terminate if the stored residual is zero.
For each integer $k\ge0$ such that the computation through $r_{k+1}$
incurs neither underflow nor overflow, the stored residual satisfies
\begin{equation}
 \|r_{k+1}\|_{A^{-1}}\le\Bigl(1-\frac1{\kappa_2(A)+1}\Bigr)\|r_k\|_{A^{-1}},
 \qquad
 \|r_k\|_2\le\sqrt{\kappa_2(A)}\Bigl(1-\frac1{\kappa_2(A)+1}\Bigr)^k\|r_0\|_2 .
 \label{eq:sd-geometric}
\end{equation}
Thus the stored residual decreases geometrically for as long as the
relative-error model applies.  In an unbounded-exponent model it falls
below every positive threshold in finitely many steps.  After exact
termination, set subsequent residuals to zero.
\end{theorem}

\begin{proof}
Condition \eqref{eq:sd-condition} implies
$c_A\gamma_n\kappa_2(A)\le1/28$, so
Lemma~\ref{lem:sd-one-step} applies whenever $r_k\ne0$ and the arithmetic
assumptions hold. The exact steepest-descent contraction follows from the
Kantorovich inequality \cite[Sec.~5.3]{Saad2003},
\[
 \|(I-\tilde a(r_k)A)r_k\|_{A^{-1}}
 \le\frac{\kappa_2(A)-1}{\kappa_2(A)+1}\|r_k\|_{A^{-1}}.
\]
Adding the rounding-error bound of Lemma~\ref{lem:sd-one-step} gives
\[
\begin{aligned}
 \|r_{k+1}\|_{A^{-1}}
 &\le\left(\frac{\kappa_2(A)-1}{\kappa_2(A)+1}
           +14c_A\gamma_n\kappa_2(A)\right)\|r_k\|_{A^{-1}}\\
 &\le\left(\frac{\kappa_2(A)-1}{\kappa_2(A)+1}
           +\frac1{\kappa_2(A)+1}\right)\|r_k\|_{A^{-1}}.
\end{aligned}
\]
This is the first inequality in \eqref{eq:sd-geometric}.
Iterating it and using
$\|r\|_2\le\lambda_{\max}^{1/2}\|r\|_{A^{-1}}$ and
$\|r\|_{A^{-1}}\le\lambda_{\min}^{-1/2}\|r\|_2$
gives the second inequality.
\end{proof}

\begin{remark}[Diagonal matrices]
\label{rem:sd-diagonal}
For diagonal $A$, the sufficient condition improves to
\[
 7\,\gamma_{n+1}\bigl(\kappa_2(A)+1\bigr)\le1 .
\]
Here the dependence on $\kappa_2(A)$ is linear.  To verify the constant,
put $g=\gamma_{n+1}$.  The condition gives $g\le1/14$.  All terms in the
denominator are nonnegative, so its combined relative error is at most
$g$.  The numerator followed by the final division has the same bound.
Writing $a=\tilde a(1+\psi)$ therefore gives
\[
 |\psi|\le\frac{2g}{1-g},\qquad
 |1+\psi|\le\frac{1+g}{1-g}.
\]
For positive diagonal $A$,
$\|\delta q\|_{A^{-1}}\le u\|r\|_A$ and
$\|q\|_{A^{-1}}\le(1+u)\|r\|_A$.
The two rounded operations in each residual component also give
\[
 \|e\|_{A^{-1}}\le u\|r\|_{A^{-1}}
       +(2u+u^2)a\|q\|_{A^{-1}}.
\]
Using $\tilde a\|r\|_A\le\|r\|_{A^{-1}}$ and $u\le g$, the defect in
\eqref{eq:sd-perturbed-step}, divided by $g\|r\|_{A^{-1}}$, is at most
\[
 \frac{2}{1-g}+\frac{1+g}{1-g}+1+
 \frac{(2+g)(1+g)^2}{1-g}<7\qquad(0\le g\le1/14).
\]
The expression increases with $g$ and is below $6.869$ at $g=1/14$,
so the contraction proof applies with $7g$ in place of
$14c_A\gamma_n\kappa_2(A)$.  For a dense matrix, the product error is bounded
by $\gamma_n|A||r|$, which introduces the additional condition-number
factor in \eqref{eq:sd-condition}.
\end{remark}

\begin{figure}[t]
\centering
\includegraphics[width=\textwidth]{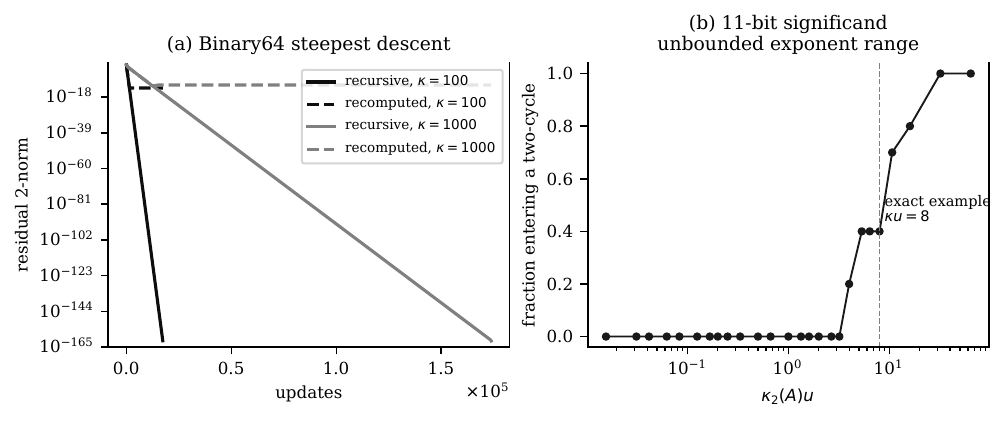}
\caption{(a) Recursively updated and true residuals for steepest descent
on dense SPD matrices with $n=100$ and logarithmically spaced eigenvalues,
computed in binary64.  The runs stop when the squared norm used by the
code underflows.  The residual components are still nonzero.
(b) The fraction of ten starting residuals entering a two-cycle for
$A=\operatorname{diag}(cu,1)$, computed with an eleven-bit significand and
unbounded exponents.  Here $c>0$ is the scanned parameter.
The construction of Theorem~\ref{thm:sd-cycle} has $\kappa_2(A)u=8$.}
\label{fig:sd}
\end{figure}

Figure~\ref{fig:sd} compares these effects.  For the dense binary64
examples with $\kappa_2(A)\in\{10^2,10^3\}$, condition
\eqref{eq:sd-condition} holds and the stored residual decreases long
after the true residual has reached its attainable-accuracy level.
The squared norm eventually underflows while the residual components
remain nonzero.  In the eleven-bit diagonal scan, the first observed
two-cycles occur at $\kappa_2(A)u=4$.  Four of ten starts enter a cycle
at $\kappa_2(A)u=8$, and all ten do so at the sampled values at least
$32$.  These observations illustrate the dependence on conditioning
and the initial direction alongside the sufficient convergence theorem.

\subsection{Scaling and recursively updated CG residuals}
\label{sec:cg-residual}

For CG, the recurrence has an exact scaling property in binary arithmetic.
It separates the scale of a stored residual from the evolution of its
direction.

\begin{proposition}[Scale invariance of the residual recurrence]
\label{prop:scale-invariance}
Consider two executions of \eqref{eq:hscg} with the same matrix and
evaluation order in binary round-to-nearest arithmetic.  At an index
$j$, let their stored residual and search-direction pairs be
$(r_j,p_j)$ and $(2^t r_j,2^t p_j)$ for an integer $t$.
Assume that the required scalings are representable, all subsequent
operations in both executions avoid underflow and overflow, and the
divisors are nonzero.  Then their subsequent residuals and search
directions differ by exactly the factor $2^t$, while their coefficients
$a_i^{\rm cg}$ and $b_{i+1}^{\rm cg}$ agree.  The same statement holds
for the residual and step size in Algorithm~\ref{alg:sd}.
\end{proposition}

\begin{proof}
Multiplication by a power of two commutes with rounding under the range
assumptions.  Thus $q_j$ scales by $2^t$ and both stored inner products
in $a_j^{\rm cg}$ scale by $2^{2t}$.  Their quotient is unchanged.
The residual update scales by $2^t$, including its separately rounded
products and subtractions.  Both squared norms in $b_{j+1}^{\rm cg}$
scale by $2^{2t}$, so this coefficient is also unchanged, and the search
direction update scales by $2^t$.  Induction proves the assertion.
Setting the search direction equal to the residual gives the same
argument for steepest descent.
\end{proof}

The stored iterate is updated from $(r_j,p_j)$, while their recurrence
depends only on the matrix, the current pair, and the arithmetic.
Dyadic scaling therefore preserves the directional evolution as long as
the range assumptions hold.  This explains why unit roundoff is a
relative scale, rather than an absolute threshold for the norm of the
updated residual.

Wo\'zniakowski analyzed rounding in a class of CG recurrences
\cite{Wozniakowski1980}.  Strako\v{s} and Tich\'y developed local
orthogonality and energy-error estimates for CG and preconditioned CG
\cite{StrakosTichy2002,StrakosTichy2005}, with a fuller account given by
Meurant and Tich\'y \cite{MeurantTichy2024}.  Their bounds
involve the norms of the current vectors and ratios of consecutive
residual norms.  Greenbaum's nearby-problem model then relates perturbed
Lanczos recurrences to exact recurrences on enlarged symmetric matrices
\cite{Greenbaum1989}.  Together, these analyses identify the quantities
that a long-run CG convergence argument must control.  In
Section~\ref{sec:envelope} we state a common spectral bound for the
enlarged models and make the transfer to the computed residuals explicit.

The residual gap also depends on the organization of the computation.
Carson and Demmel analyze attainable accuracy in $s$-step Krylov
methods \cite{CarsonDemmel2014}, while Cools and coauthors study
residual replacement in pipelined CG \cite{CoolsEtAl2018}.
Carson, Rozlo\v{z}n\'ik, Strako\v{s}, Tich\'y, and T\r{u}ma compare
the numerical behavior of communication-reducing CG variants
\cite{CarsonRozloznikStrakosTichyTuma2018}.
\section{CG and Lanczos computations in finite precision}

In exact arithmetic, CG produces the Galerkin iterates generated by
symmetric Lanczos.  Their floating-point recurrences store different
quantities.  CG updates $x_j$, $r_j$, and $p_j$, while direct
Lanczos--Galerkin forms and solves a projected tridiagonal system.
The following construction follows both computations through their
individual rounded assignments.

\begin{theorem}[Different outcomes for Hestenes--Stiefel CG and a direct Lanczos--Galerkin computation]
\label{thm:cg-lanczos-separation}
Fix a binary floating-point format with significand precision $p\ge2$,
rounding to nearest with ties to even, and unit roundoff $u=2^{-p}$.  Define
\[
 s=\left\lceil\frac{p+2}{2}\right\rceil,
 \quad H=2^s,
 \quad h=H^{-1},
 \quad \epsilon_A=\frac{h^2u}{4}.
\]
Assume that every nonzero exact intermediate result in the four CG iterations
defined by \eqref{eq:hscg} and the first two Lanczos steps defined by
\eqref{eq:scalar-lanczos-order} is normal and finite.  Consider the symmetric
positive definite system
\[
 A=\diag(\epsilon_A,1),\qquad b=(1,H)^T,\qquad x_0=0.
\]
Initialize \eqref{eq:hscg} with $r_0=b$ and $p_0=r_0$.  Execute four CG
iterations, taking all four updates before testing convergence, and denote the stored iterates by
$x_j^{\rm cg}$.  Then
\[
 x_4^{\rm cg}=(\epsilon_A^{-1},H)^T=A^{-1}b.
\]

Next, apply \eqref{eq:scalar-lanczos-order} for two steps, and form
$T_2^{\rm lan}$ from the stored recurrence coefficients.  The stored initial
coefficient and tridiagonal matrix are
\[
 \beta_1^{\rm lan}=H,
 \qquad
 T_2^{\rm lan}=\begin{bmatrix}1&h\\h&h^2\end{bmatrix}.
\]
The stored matrix is singular, and the projected Galerkin system
\[
 T_2^{\rm lan}y=\beta_1^{\rm lan}e_1=He_1
\]
is inconsistent.  Thus CG stores the exact solution after four updates,
while the direct projected equation at step two is singular and inconsistent.
\end{theorem}

The normality hypothesis holds for binary32, binary64, and binary128.
For $p=11$, the value $\epsilon_A=2^{-27}$ requires an exponent range
larger than that of binary16.

\begin{proof}
The relation $h^2<u/2$ gives the strict rounding identities
\[
 \fl(H^2+1)=H^2,
 \quad \fl(1+h^2)=1,
 \quad \fl(h+h^3)=h,
 \quad \fl(h^2+\epsilon_A)=h^2,
 \quad \fl(1+\epsilon_AH^2)=1.
\]
Every nonzero stored vector entry and coefficient is a signed power of
two.  Substitution into the specified assignments gives the following
trajectory.

For CG, the stored states are
\[
\begin{array}{c|c|c|c}
j&x_j&r_j&p_j\\ \hline
0&(0,0)&(1,H)&(1,H)\\
1&(1,H)&(1,0)&(1,h)\\
2&(H^2,2H)&(1,-H)&(H^2,0)\\
3&(1/\epsilon_A,2H)&(0,-H)&(H^2,-H)\\
4&(1/\epsilon_A,H)&(-u/4,0)&\text{unused}.
\end{array}
\]
At the four steps, the stored coefficients $a_j^{\rm cg}$ are respectively
$1,H^2,4/u,1$, while
$b_1^{\rm cg},b_2^{\rm cg},b_3^{\rm cg}$ are $h^2,H^2,1$.
Substitution in \eqref{eq:hscg}, using the rounding relations above,
verifies every row.  The final true residual is zero even though the recursive
residual is $(-u/4,0)^T$.

For Lanczos, the rounded squared norm is $H^2$, so the initial
normalization gives $H$ and $v_1=(h,1)^T$.  The first step yields
\[
 \alpha_1^{\rm lan}=1,
 \qquad \beta_2^{\rm lan}=h,
 \qquad v_2=(-1,0)^T.
\]
In the second step, the first component of $w_2$ is
$\fl(-\epsilon_A-h^2)=-h^2$.  Consequently the next diagonal coefficient is
\(
 \alpha_2^{\rm lan}=\fl(h^2+\epsilon_A)=h^2.
\)
The determinant of the stored matrix is zero.  Its range consists of
vectors $z$ satisfying $z_2=hz_1$.  Since $hH=1$, the right-hand side
$(H,0)^T$ lies outside that range.
\end{proof}

For $\nu\in\{1,2\}$, define a $\nu$-pass modified Gram--Schmidt variant as
follows.  After $z_1$ has been formed in
\eqref{eq:scalar-lanczos-order}, set $z_1^{(0)}=z_1$ and compute
\[
 c_\ell=\fl\!\left(v_1^Tz_1^{(\ell-1)}\right),
 \qquad
 z_1^{(\ell)}=
 \fl\!\left(z_1^{(\ell-1)}-c_\ell v_1\right),
 \qquad 1\le\ell\le\nu.
\]
Normalize $z_1^{(\nu)}$ in the target format to obtain $v_2^{(\nu)}$.
Recompute
\[
 t_{12}^{(\nu)}=\fl\!\left(v_1^T\fl(Av_2^{(\nu)})\right),
 \qquad
 t_{22}^{(\nu)}=\fl\!\left((v_2^{(\nu)})^T
                      \fl(Av_2^{(\nu)})\right),
\]
using the stated left-to-right accumulation order, and set
\[
 T_2^{(\nu)}=
 \begin{bmatrix}
  \alpha_1^{\rm lan}&t_{12}^{(\nu)}\\
  t_{12}^{(\nu)}&t_{22}^{(\nu)}
 \end{bmatrix}.
\]

\begin{corollary}[One or two reorthogonalization passes]
\label{cor:cg-lanczos-reorthogonalization}
Under the hypotheses of Theorem~\ref{thm:cg-lanczos-separation}, the
reorthogonalized computations defined above satisfy
\[
 v_2^{(\nu)}=(-1,h)^T,
 \qquad
 T_2^{(\nu)}=\begin{bmatrix}1&h\\h&h^2\end{bmatrix},
 \qquad \nu\in\{1,2\}.
\]
Consequently, $T_2^{(\nu)}y=He_1$ has no solution for either value of $\nu$.
\end{corollary}

\begin{proof}
The first projection coefficient is $-h^2$.  The corrected vector is
$(-h,h^2)^T$, whose stored norm is $h$.  Normalization therefore gives
$v_2^{(\nu)}=(-1,h)^T$.  The rounding identity $\fl(h-h^3)=h$ is used
in the corrected first component.  A second projection coefficient is
zero.  Recomputing the off-diagonal and diagonal entries gives $h$ and
$h^2$, respectively.
\end{proof}

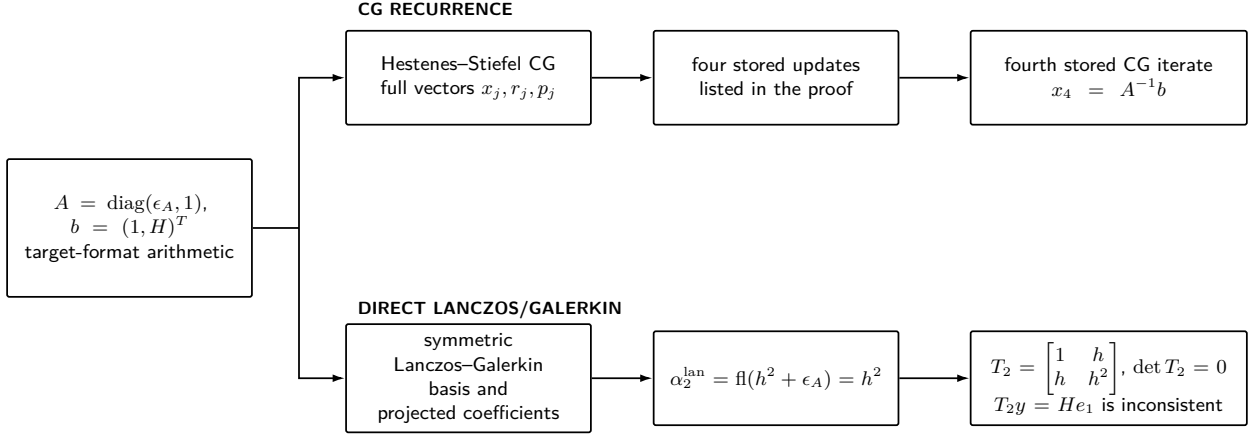
\begin{figure}[t]
  \centering
  \resizebox{\textwidth}{!}{\begin{tikzpicture}[
  x=1mm,y=1mm,
  font=\sffamily\footnotesize,
  >=Latex,
  input/.style={draw=black, line width=0.8pt, rounded corners=1pt,
                fill=white, text width=36mm, minimum height=22mm,
                align=center, inner sep=1.5mm},
  stage/.style={draw=black, line width=0.8pt, rounded corners=1pt,
                fill=white, text width=36mm, minimum height=15mm,
                align=center, inner sep=1.5mm},
  outcome/.style={draw=black, line width=0.8pt, rounded corners=1pt,
                  fill=white, text width=41mm, minimum height=15mm,
                  align=center, inner sep=1.5mm},
  branchlabel/.style={font=\sffamily\bfseries\scriptsize, text=black},
  line/.style={line width=0.8pt, draw=black},
  arr/.style={-{Latex[length=2mm,width=1.25mm]}, line width=0.8pt, draw=black}
]

\node[input] (data) at (0,0)
  {$A=\operatorname{diag}(\epsilon_A,1)$, $b=(1,H)^T$\\target-format arithmetic};

\node[branchlabel, anchor=west] at (35,35) {CG RECURRENCE};
\node[stage] (cgstate) at (54,24)
  {Hestenes--Stiefel CG\\full vectors $x_j,r_j,p_j$};
\node[stage] (retain) at (103,24)
  {four stored updates\\listed in the proof};
\node[outcome] (solution) at (156,24)
  {fourth stored CG iterate\\$x_4=A^{-1}b$};

\node[branchlabel, anchor=west] at (35,-13) {DIRECT LANCZOS/GALERKIN};
\node[stage] (lstate) at (54,-24)
  {symmetric Lanczos--Galerkin\\basis and\\projected coefficients};
\node[stage] (compress) at (103,-24)
  {$\alpha_2^{\rm lan}=\fl(h^2+\epsilon_A)=h^2$};
\node[outcome] (singular) at (156,-24)
  {$T_2=\begin{bmatrix}1&h\\h&h^2\end{bmatrix}$, $\det T_2=0$\\$T_2y=He_1$ is inconsistent};

\coordinate (fork) at (27,0);
\draw[line] (data.east) -- (fork);
\draw[arr] (fork) |- (cgstate.west);
\draw[arr] (fork) |- (lstate.west);
\draw[arr] (cgstate.east) -- (retain.west);
\draw[arr] (retain.east) -- (solution.west);
\draw[arr] (lstate.east) -- (compress.west);
\draw[arr] (compress.east) -- (singular.west);

\end{tikzpicture}
}
  \caption{The two computations in
  Theorem~\ref{thm:cg-lanczos-separation}.  The right-hand side $He_1$ lies
  outside the range of the stored matrix $T_2^{\rm lan}$, making the
  projected system inconsistent.  CG stores the exact solution at step four.}
  \label{fig:cg-lanczos}
\end{figure}

The theorem gives an exact two-dimensional comparison for the stated
evaluation orders, at $\kappa_2(A)\ge16/u^2$.  A loose stopping tolerance
can accept the common first iterate.  Here the four
stored CG updates are compared with the direct projected equation.
Saad discusses breakdown in direct Lanczos and the LQ-based SYMMLQ
alternative \cite[pp.~197--198]{Saad2003}.

Finite-precision differences between equivalent Krylov recurrences have
been studied by Gutknecht and Strako\v{s}, Meurant and Strako\v{s}, and
Greenbaum, Liu, and Chen
\cite{GutknechtStrakos2000,MeurantStrakos2006,GreenbaumLiuChen2021}.
\v{S}imonov\'a and Tich\'y also give a structured example comparing
standard Lanczos with Hestenes--Stiefel CG \cite{SimonovaTichy2022}.
We next identify the local orthogonality term in the present construction,
then compare the two methods through a common spectral enclosure.
\subsection{Paige's local-orthogonality relation}
\label{sec:mechanism-216}

The computed vectors in Theorem~\ref{thm:cg-lanczos-separation} make
the effect of local orthogonality particularly clear.  A small weighted
inner product changes the second diagonal coefficient by more than the
smallest eigenvalue of the original matrix.

\begin{proposition}[Loss of orthogonality in Theorem~\ref{thm:cg-lanczos-separation}]
\label{prop:paige-step-one}
In Theorem~\ref{thm:cg-lanczos-separation}, the stored Lanczos vectors
$v_1=(h,1)^T$ and $v_2=(-1,0)^T$ satisfy
\[
 v_1^Tv_2=-h,\qquad
 \beta_2^{\rm lan}v_1^Tv_2=-h^2\in\{-u/4,-u/8\}.
\]
The computed second diagonal coefficient is
\[
 \alpha_2^{\rm lan}
 =\fl\bigl(v_2^TAv_2-\beta_2^{\rm lan}v_2^Tv_1\bigr)
 =\fl(\epsilon_A+h^2)=h^2.
\]
The exact Galerkin matrix of this stored basis is
\[
 V_2^TAV_2=
 \begin{bmatrix}1+\epsilon_Ah^2&-\epsilon_Ah\\
                 -\epsilon_Ah&\epsilon_A\end{bmatrix}.
\]
It is positive definite with determinant $\epsilon_A$.  The local
orthogonality term $-\beta_2^{\rm lan}v_2^Tv_1=h^2$ dominates
$\epsilon_A=h^2u/4$ in the formation of the singular stored matrix
$T_2^{\rm lan}$.
\end{proposition}

\begin{proof}
The values of $v_1$, $v_2$, $\beta_2^{\rm lan}$, and
$\alpha_2^{\rm lan}$ follow from the preceding proof.  Since
$s=\lceil(p+2)/2\rceil$, the value $h^2=2^{-2s}$ is $u/4$ for even $p$
and $u/8$ for odd $p$.  Direct multiplication by
$A=\diag(\epsilon_A,1)$ gives the displayed Galerkin matrix and
\[
 (1+\epsilon_Ah^2)\epsilon_A-\epsilon_A^2h^2=\epsilon_A>0.
\]
Its leading diagonal entry is positive, so it is positive definite.
In the rounded formation of $w_2$, the quantity
$-\epsilon_A-h^2$ rounds to $-h^2$, producing the stated coefficient.
\end{proof}

The weighted adjacent-vector product is at most $u/4$, consistent with
Paige's local relation, while the inner product itself has magnitude
$h$, of order $\sqrt u$, because its multiplying coefficient is small.
The zero Ritz value is therefore associated with a substantial residual
even though its distance to the smallest eigenvalue is tiny.
The following estimate quantifies this distinction.

\begin{proposition}[A residual bound for the computed Ritz values]
\label{prop:singular-T}
Let $A\in\R^{n\times n}$ be symmetric positive definite and suppose
\[
 AV_k=V_kT_k+\beta_{k+1}v_{k+1}e_k^T+F_k,
 \qquad T_k=T_k^T,\qquad \beta_{k+1}\ge0.
\]
For an eigenpair $(\theta,s)$ of $T_k$, assume $\|s\|_2=1$ and
$y=V_ks\ne0$.  Then
\begin{equation}
 \operatorname{dist}(\theta,\sigma(A))
 \le\frac{\|F_k\|_2+\beta_{k+1}|e_k^Ts|\|v_{k+1}\|_2}{\|y\|_2}.
 \label{eq:ritz-a-posteriori}
\end{equation}
If an orthonormal eigenbasis $s_1,\ldots,s_k$ of $T_k$ satisfies
$V_ks_i\ne0$ and
\[
 \lambda_{\min}(A)>
 \max_{1\le i\le k}
 \frac{\|F_k\|_2+\beta_{k+1}|e_k^Ts_i|\|v_{k+1}\|_2}
      {\|V_ks_i\|_2},
\]
then $T_k$ is positive definite.  Its Galerkin system therefore has a
unique solution.
\end{proposition}

\begin{proof}
Applying $A-\theta I$ to the lifted vector gives
\[
 (A-\theta I)y=F_ks+\beta_{k+1}v_{k+1}e_k^Ts.
\]
Symmetry implies
$\|(A-\theta I)y\|_2\ge\operatorname{dist}(\theta,\sigma(A))\|y\|_2$.
The triangle inequality proves \eqref{eq:ritz-a-posteriori}.  Under the
last hypothesis, every eigenvalue of $T_k$ has distance less than
$\lambda_{\min}(A)$ from the positive spectrum of $A$, and is positive.
\end{proof}

In the two-dimensional construction, the eigenvector of $T_2^{\rm lan}$
for zero is proportional to $(h,-1)^T$.  Its lifted vector has norm of
order one, while its residual has a component of order $h$.
The bound is consistent with the much smaller value
$\lambda_{\min}(A)=\epsilon_A$.
Paige's separate containment analysis bounds computed Ritz values by the
extreme eigenvalues of $A$ and a local-error term depending on the step
count \cite{Paige1976,Paige1980}.  When that term is smaller
than $\lambda_{\min}(A)$, the computed tridiagonal remains positive
definite.  CG instead requires a positive stored curvature denominator
at each update.  Proposition~\ref{prop:cg-curvature-loss} below gives
an exact example in which that denominator vanishes.

\subsection{A common spectral convergence bound}
\label{sec:envelope}

Greenbaum's analysis associates perturbed Lanczos recurrences with exact
recurrences on enlarged symmetric matrices \cite{Greenbaum1989}.
Greenbaum and Strako\v{s} use such models to study convergence delay
\cite{GreenbaumStrakos1992}.  Two computed recurrences can have different
coefficients while their model matrices share a spectral enclosure.
The next proposition states the resulting common bound, with separate
residuals for the model and the physical problem.

\begin{proposition}[A common convergence bound for enlarged models]
\label{prop:envelope}
Let $A\in\R^{n\times n}$ be symmetric positive definite, let
$0\le\delta<\lambda_{\min}(A)$, and set
\[
 a_*=\lambda_{\min}(A)-\delta,\quad
 b_*=\lambda_{\max}(A)+\delta,\quad
 \widehat\kappa=b_*/a_*,\quad
 q_*=\frac{\sqrt{\widehat\kappa}-1}{\sqrt{\widehat\kappa}+1}.
\]
Suppose that the CG and direct Lanczos recurrences admit exact enlarged
Lanczos models $\widehat A_{\rm cg}$ and $\widehat A_{\rm lan}$ whose
spectra lie in $[a_*,b_*]$.  The models reproduce the respective
tridiagonal coefficients and the next coupling coefficient through $k$
steps.  For CG, use the Jacobi coefficients recovered in exact arithmetic
from the positive stored coefficients $a_j^{\rm cg}$ and $b_{j+1}^{\rm cg}$ by
\[
 \overline\alpha_1=1/a_0^{\rm cg},\qquad
 \overline\alpha_j=1/a_{j-1}^{\rm cg}
       +b_{j-1}^{\rm cg}/a_{j-2}^{\rm cg}\quad(2\le j\le k),
\]
\[
 \overline\beta_{j+1}=\sqrt{b_j^{\rm cg}}/a_{j-1}^{\rm cg}
 \quad(1\le j\le k).
\]
For $\ell\in\{{\rm cg},{\rm lan}\}$, let $\widehat r_j^\ell$ be the
exact CG residual on $\widehat A_\ell$, starting from zero with a
nonzero multiple of its first Lanczos vector as right-hand side.
Use the convention $q_*^0=1$.  Then, for $0\le j\le k$,
\[
 \frac{\|\widehat r_j^\ell\|_{\widehat A_\ell^{-1}}}
      {\|\widehat r_0^\ell\|_{\widehat A_\ell^{-1}}}
 \le2q_*^j,\qquad
 \frac{\|\widehat r_j^\ell\|_2}{\|\widehat r_0^\ell\|_2}
 \le2\sqrt{\widehat\kappa}\,q_*^j.
\]
For the stored CG residuals, assume that each squared norm is reused
or identically reevaluated in the same order in the coefficient ratios,
with relative error at most
$\gamma_n<1$, and that the divisions have relative error at most $u<1$.
Then, before a zero residual or a zero divisor,
\[
 \frac{\|r_j\|_2}{\|r_0\|_2}
 \le\sqrt{\frac{1+\gamma_n}{1-\gamma_n}}(1-u)^{-j/2}
 \frac{\|\widehat r_j^{\rm cg}\|_2}{\|\widehat r_0^{\rm cg}\|_2}.
\]
For direct Lanczos, write $r_0^{\rm true}=b-Ax_0$ and suppose
$AV_j=V_jT_j+\beta_{j+1}v_{j+1}e_j^T+F_j$.
For $\beta_1>0$, if $T_jy_j=\beta_1e_1$ is solved exactly and
$x_j=x_0+V_jy_j$, then
\[
 b-Ax_j=(r_0^{\rm true}-\beta_1v_1)
        -\beta_{j+1}(e_j^Ty_j)v_{j+1}-F_jy_j.
\]
With $\|\widehat r_0^{\rm lan}\|_2=\beta_1$, this gives
\[
 \|b-Ax_j\|_2\le\|r_0^{\rm true}-\beta_1v_1\|_2+
 \|v_{j+1}\|_2\|\widehat r_j^{\rm lan}\|_2+\|F_jy_j\|_2.
\]
\end{proposition}

\begin{proof}
Both model matrices are positive definite and have condition number at
most $\widehat\kappa$.  The classical Chebyshev estimate for exact CG
gives the energy-norm bounds \cite[Sec.~6.11.3]{Saad2003}.
Equivalence of Euclidean and inverse energy norms gives the factor
$\sqrt{\widehat\kappa}$ in the second bound.

For the CG transfer, denote the stored squared norm by
$d_i=\|r_i\|_2^2(1+\theta_i)$, where $|\theta_i|\le\gamma_n$.
Write the stored ratio as
$b_i^{\rm cg}=(d_i/d_{i-1})(1+\epsilon_i)$, with $|\epsilon_i|\le u$.
The exact CG--Lanczos correspondence for the recovered coefficients gives
\[
 \frac{\|\widehat r_j^{\rm cg}\|_2^2}
      {\|\widehat r_0^{\rm cg}\|_2^2}
 =\prod_{i=1}^j b_i^{\rm cg}
 =\frac{\|r_j\|_2^2}{\|r_0\|_2^2}
   \frac{1+\theta_j}{1+\theta_0}
   \prod_{i=1}^j(1+\epsilon_i).
\]
Solving for the stored residual ratio proves its bound.
Multiplication of the physical Lanczos relation by $y_j$ gives the
true-residual identity.  In the exact enlarged Lanczos model the residual
norm is $\beta_{j+1}|e_j^Ty_j|$, so the triangle inequality proves the
last estimate.
\end{proof}

The existence and width of these enlarged models follow from the local
perturbation and orthogonality hypotheses of the corresponding
backward-error analysis.  The proposition shows how a common enclosure
leads to a common upper bound.  The true CG residual additionally
contains the residual gap $b-A\widehat x_j-r_j$, studied by
Greenbaum \cite{Greenbaum1997Residual}.  A rounded projected solve or
rounded formation of the Lanczos iterate adds its residual term to the
displayed direct-Lanczos identity.

Figure~\ref{fig:cg-lanczos-strakos} illustrates the distinction between
coefficients and convergence on a dense Strako\v{s} matrix with $n=48$,
extreme eigenvalues $0.1$ and $100$, and spacing parameter $\rho=0.8$.
The binary64 residual histories are similar, although the recovered
tridiagonal coefficients separate as basis orthogonality deteriorates.
The reference trajectory is computed by CG in 300-bit arithmetic.
The computed methods need additional iterations to attain the same
residual level.  Such convergence delay is consistent with enlarged
spectral models \cite{GreenbaumStrakos1992,Strakos1991,Notay1993},
while the sensitivity of tridiagonal coefficients is related to the
conditioning of the moment-to-Jacobi map
\cite{Knizhnerman1996,ChenTrogdon2024}.
Gergelits and Strako\v{s} develop composite convergence bounds using
Chebyshev polynomials \cite{GergelitsStrakos2014}.
Meurant and Strako\v{s} discuss these effects in detail
\cite{Meurant2006,MeurantStrakos2006}, and Carson, Liesen, and
Strako\v{s} examine how concrete examples clarify convergence
\cite{CarsonLiesenStrakos2024}.

\begin{figure}[t]
\centering
\includegraphics[width=\textwidth]{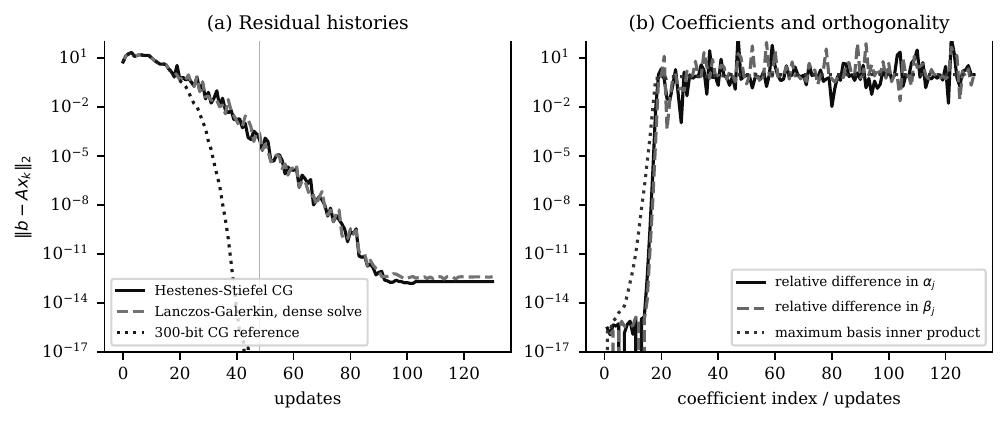}
\caption{Hestenes--Stiefel CG and direct Lanczos--Galerkin in binary64
on a dense Strako\v{s} matrix with $n=48$, $\kappa_2(A)=10^3$, and
$\rho=0.8$.  The projected system is evaluated with a general dense
linear solve.  Panel (a) compares true residual histories with a
300-bit CG reference.  The vertical line marks $k=n$.
Panel (b) shows relative differences between the tridiagonal
coefficients recovered from CG and those computed by Lanczos, together
with the loss of orthogonality of the Lanczos basis.}
\label{fig:cg-lanczos-strakos}
\end{figure}
\section{Backward error and precision bounds for CG}

The precision required for CG depends on matrix conditioning as well as
dimension.  When $u\kappa_2(A)$ is large, rounding in a matrix-vector
product can make a positive curvature denominator vanish.  We therefore
retain the condition number and fix the accumulation order.
The main theorem gives a sufficient significand precision together
with a true-residual stopping test.  A two-dimensional construction then
gives a condition-dependent lower bound for a worst-case guarantee over
representable inputs.

The proof compares the computed recurrence with exact CG on the same
data.  Exact CG reaches the solution in at most $n$ iterations.  Before
its first iterate with sufficiently small backward error, the residual
and both scalar divisors have positive lower bounds.  These bounds
permit an operation-by-operation comparison of the computed and exact
quantities.  An independently evaluated true-residual test identifies
a stored iterate meeting the target tolerance.

\subsection{Evaluation order, scaling, and error constants}
Every dense matrix row and inner product is accumulated from left to
right.  Products and partial sums are stored separately.  In an AXPY,
each product is rounded before the addition, and each quotient is formed
by one division.  A stored squared residual norm may be reused.
We count an addition, subtraction, multiplication, division, or fused
multiply-add as one scalar operation.  A fused operation may replace
one displayed product-add pair with its correctly rounded result.
The operation count uses the longer nonfused sequence with the stated
parentheses and accumulation order.

For a nonzero vector $x$, define the matrix-only backward error
\begin{equation}
 \eta_A(x)=\frac{\|b-Ax\|_2}{\|A\|_2\|x\|_2}.
 \label{eq:matrix-only-eta}
\end{equation}
The target condition has the cross-multiplied form
\begin{equation}
 \|b-Ax\|_2\le\varepsilon\|A\|_2\|x\|_2.
 \label{eq:matrix-only-cross}
\end{equation}
For $x\ne0$, this is equivalent to $\eta_A(x)\le\varepsilon$ and implies
the normwise backward-error condition \eqref{eq:backward-error}.
At $x=0$, its positive left side is $\|b\|_2$.

Let $A\in\mathbb F^{n\times n}$ be symmetric positive definite, let
$b\in\mathbb F^n\setminus\{0\}$, and take $x_0\in\mathbb F^n$.  Put
\[
 \chi=\max\left\{1,\frac{\|b-Ax_0\|_2}{\|b\|_2}\right\},
 \qquad 0<\varepsilon\le1.
\]
Choose the dyadic tolerance
\begin{equation}
 \varepsilon_0=2^{\lfloor\log_2\varepsilon\rfloor},
 \qquad \varepsilon/2<\varepsilon_0\le\varepsilon.
 \label{eq:dyadic-tolerance}
\end{equation}
Given positive powers of two $s_A$ and $s_b$, set
\begin{equation}
 A'=s_AA,\qquad b'=s_bb,\qquad x'=\frac{s_b}{s_A}x.
 \label{eq:cg-scaling-map}
\end{equation}
Assume they satisfy, in the primed variables and with
$M=\|A'\|_2$ and $B=\|b'\|_2$,
\begin{equation}
 \frac12\le M\le1,
 \qquad
 \frac12\le B\le1.
 \label{eq:cg-scaling}
\end{equation}
These powers can be determined from certified norm bounds before the
solve.  The operation count starts after these scalings have been determined.
The transformation preserves $\kappa_2(A)$, $\chi$, and the backward-error
ratios.  We drop the primes in the analysis and recover the original
variable by the exact dyadic scaling $x=(s_A/s_b)x'$.

The parameters $q$ and $d$ give lower bounds before the first accurate
exact iterate.  The constants $G$, $R$, $P$, and $U$ bound coefficients
and intermediate quantities along that trajectory, while $C$, $N$, and
$\delta$ control accumulated forward error.
Define
\begin{gather}
 q=\frac{\varepsilon_0}{10},
 \qquad
 d=\frac{\varepsilon_0^2}{200\kappa_2(A)},
 \qquad
 G=\kappa_2(A),
 \label{eq:cg-qdg}\\
 R=\sqrt{\kappa_2(A)}\chi,
 \qquad
 P=\kappa_2(A)\chi,
 \label{eq:cg-rp}\\
 U=2+d^{-1}+\kappa_2(A)\chi^2+2\kappa_2(A)(1+\chi)
      +P^2+(2\kappa_2(A)+G+2)P,                                    \label{eq:cg-U}\\
 C=\frac{8U^2}{d^2},
 \qquad
 N=2n^3+13n^2-n,
 \qquad
 \delta=\min\left\{1,\frac d2,
              \frac{\varepsilon_0}{20\sqrt n}\right\}.      \label{eq:cg-CNdelta}
\end{gather}
For the stopping calculation, $X$ and $V$ bound the residual test, while
$C_{\rm s}$, $N_{\rm s}$, and $E_{\rm s}$ bound its accumulated error.  Let
$\omega\ge0$ and set
\begin{gather}
 X=2\kappa_2(A)(1+\chi)+1,
 \qquad V=4(2+X)^2,
 \qquad C_{\rm s}=8V^2,                                      \label{eq:selector-scales}\\
 N_{\rm s}=2n^2+4n+2,
 \qquad
 \rho_{\rm s}=2u(V+1)^2+\omega,
 \qquad
 E_{\rm s}=2\rho_{\rm s}C_{\rm s}^{N_{\rm s}}.             \label{eq:selector-error}
\end{gather}

Assume scalar round-to-nearest arithmetic satisfies
\begin{equation}
 |\fl(z)-z|\le u|z|+\omega,                                   \label{eq:absolute-rounding}
\end{equation}
and that
\begin{equation}
 2(u+\omega)C^N\le\delta,
 \qquad
 E_{\rm s}\le\frac{\varepsilon_0^2}{256}.                   \label{eq:cg-smallness}
\end{equation}
For an application to the unscaled system, the powers in
\eqref{eq:cg-scaling-map}, their inverse output factor, the scaled input data,
and every component of the final inverse-scaled output must lie in the
available exponent range.  The following assumption concerns the scaled
calculation analyzed below.  Define
\begin{equation}
 H=\max\left\{2(U+1),(U+1)^2+(U+1),
              \frac{2(U+1)}d,2(V+1)^2\right\}.
 \label{eq:raw-range}
\end{equation}

The stopping test uses a separately evaluated true residual.
After forming each iterate $\widehat x_i$, beginning with $i=0$, recompute
the following quantity without modifying the CG state,
\begin{equation}
 \Phi(\widehat x_i)=\|b-A\widehat x_i\|_2^2
       -\frac{\varepsilon_0^2}{4}\|\widehat x_i\|_2^2.
 \label{eq:selector-phi}
\end{equation}
Since $\|A\|_2\ge1/2$, the inequality $\Phi(x)\le0$ implies the target condition
\eqref{eq:matrix-only-cross} with tolerance $\varepsilon_0$.  Let
$\widehat\Phi_i$ be the stored value.  Since $b\ne0$,
$\Phi(0)=\|b\|_2^2>0$, so every accepted exact test has a nonzero iterate.
Stop before computing the next CG denominator when
\begin{equation}
 \widehat\Phi_i\le-\frac{\varepsilon_0^2}{256}.
 \label{eq:selector-test}
\end{equation}

\begin{algorithm}[htbp]
\caption{Hestenes--Stiefel CG with a true-residual stopping test}
\label{alg:tested-cg}
\begin{algorithmic}[1]
\Statex Given the scaled data $A$, $b$, $x_0$, and the tolerance $\varepsilon_0$.
\State $r_0\gets\fl(b-Ax_0)$
\State $p_0\gets r_0$
\For{$i=0,1,\ldots$}
  \State Evaluate $\widehat\Phi_i$ from \eqref{eq:selector-phi} without modifying the CG state.
  \If{$\widehat\Phi_i\le-\varepsilon_0^2/256$}
    \State \Return $x_i$
  \EndIf
  \If{$r_i=0$}
    \State Terminate with a zero stored residual.
  \EndIf
  \State $q_i\gets\fl(Ap_i)$
  \State $a_i^{\rm cg}\gets\displaystyle\fl\!\left(
    \frac{\fl(\langle r_i,r_i\rangle)}{\fl(\langle q_i,p_i\rangle)}\right)$
  \State $x_{i+1}\gets\fl(x_i+a_i^{\rm cg}p_i)$
  \State $r_{i+1}\gets\fl(r_i-a_i^{\rm cg}q_i)$
  \State $b_{i+1}^{\rm cg}\gets\displaystyle\fl\!\left(
    \frac{\fl(\langle r_{i+1},r_{i+1}\rangle)}{\fl(\langle r_i,r_i\rangle)}\right)$
  \State $p_{i+1}\gets\fl(r_{i+1}+b_{i+1}^{\rm cg}p_i)$
\EndFor
\end{algorithmic}
\end{algorithm}

\begin{assumption}[Range of the scaled CG computation]
\label{ass:cg-range}
The quantities $\varepsilon_0$, $\varepsilon_0^2/4$, and
$\varepsilon_0^2/256$ are represented exactly.  The scaled data $A$, $b$,
and $x_0$ are stored floating-point values.  Every exact scalar result of
Algorithm~\ref{alg:tested-cg} and its stopping test that has magnitude at most
$H$ lies in the finite exponent range, so that
\eqref{eq:absolute-rounding} applies without overflow.
\end{assumption}

\subsection{Sufficient precision and condition-number bounds}

\begin{theorem}[Sufficient precision for a backward-error bound]
\label{thm:cg-bit-bound}
Let $A=A^T\succ0$ and $b\ne0$ be floating-point data satisfying
\[
 \frac12\le\|A\|_2\le1,
 \qquad
 \frac12\le\|b\|_2\le1.
\]
Let $x_0$ be a floating-point vector and let $0<\varepsilon\le1$.  Set
\[
 \chi=\max\left\{1,\frac{\|b-Ax_0\|_2}{\|b\|_2}\right\},
 \qquad
 \varepsilon_0=2^{\lfloor\log_2\varepsilon\rfloor}.
\]
Define
\[
 q,d,G,R,P,U,C,N,\delta
 \quad\hbox{by}\quad
 \eqref{eq:cg-qdg}\text{--}\eqref{eq:cg-CNdelta},
\]
and define
\[
 X,V,C_{\rm s},N_{\rm s},E_{\rm s}
 \quad\hbox{by}\quad
 \eqref{eq:selector-scales}\text{--}\eqref{eq:selector-error}.
\]

Suppose scalar rounding satisfies \eqref{eq:absolute-rounding}, the two
smallness inequalities \eqref{eq:cg-smallness} hold, and
Assumption~\ref{ass:cg-range} holds.  Then Algorithm~\ref{alg:tested-cg} can be evaluated
without a zero or nonfinite divisor until its stopping test succeeds.  The
test succeeds at an index $i_*\le n$, and the returned vector is nonzero and
satisfies
\begin{equation}
 \eta(\widehat x_{i_*})
 \le\eta_A(\widehat x_{i_*})
 \le\varepsilon_0\le\varepsilon.
 \label{eq:cg-eta-result}
\end{equation}
\end{theorem}

\begin{proof}
Let $M=\|A\|_2$, $B=\|b\|_2$, and $m=\lambda_{\min}(A)$.  Then
\begin{equation}
 \frac12\le M,B\le1,\qquad m=M/\kappa_2(A)\ge(2\kappa_2(A))^{-1}.
 \label{eq:scaled-basic}
\end{equation}
Consider exact CG on the same data, and let $j\le n$ be the first index
satisfying
\begin{equation}
 \|b-Ax_j\|_2\le\frac{\varepsilon_0}{4}M\|x_j\|_2.
 \label{eq:first-good-exact}
\end{equation}
Such an index exists because exact CG terminates within $n$ steps.
For $i<j$, the reverse inequality and
$B\le M\|x_i\|_2+\|r_i\|_2$ give
\begin{equation}
 \|r_i\|_2>\frac{\varepsilon_0B}{4+\varepsilon_0}
 \ge\frac{\varepsilon_0}{10}=q.
 \label{eq:residual-floor}
\end{equation}
Writing $e_i=A^{-1}b-x_i$, the minimizing property of exact CG yields
\begin{equation}
 \|r_i\|_2^2=e_i^TA^2e_i
 \le M\|e_i\|_A^2\le Mr_0^TA^{-1}r_0\le\kappa_2(A)\chi^2,
 \label{eq:exact-residual-ceiling}
\end{equation}
and consequently
\begin{equation}
 \|x_i\|_2\le2\kappa_2(A)(1+\chi).
 \label{eq:exact-x-ceiling}
\end{equation}

Residual orthogonality gives $\|p_i\|_2\ge\|r_i\|_2$.
For $i\ge1$, conjugacy and $p_i=r_i+b_i^{\rm cg}p_{i-1}$ give
$\|p_i\|_A^2=p_i^TAr_i\le\|p_i\|_A\|r_i\|_A$.
Thus $\|p_i\|_2\le\sqrt{\kappa_2(A)}\|r_i\|_2\le P$, also at $i=0$.
Monotonicity of $\|e_i\|_A$ gives
$\|r_{i+1}\|_2^2/\|r_i\|_2^2\le
M\|e_{i+1}\|_A^2/(m\|e_i\|_A^2)\le\kappa_2(A)$.
Before index $j$, these estimates imply
\begin{equation}
 0<a_i^{\rm cg}\le2\kappa_2(A),\qquad
 0\le b_{i+1}^{\rm cg}\le G,\qquad \|p_i\|_2\le P,
 \label{eq:coefficient-bounds}
\end{equation}
\begin{equation}
 p_i^TAp_i\ge m\|p_i\|_2^2\ge d,\qquad
 r_i^Tr_i\ge q^2\ge d.
 \label{eq:divisor-floor}
\end{equation}
Here the usual exact CG identities are applied only through the first
accurate iterate \cite[Secs.~6.7 and~6.11.3]{Saad2003}.

The intermediate quantities satisfy the bounds
\begin{center}
\begin{tabular}{c|c}
quantity & upper bound \\ \hline
$r_i$, $p_i$, $x_i$ & $R$, $P$, $2\kappa_2(A)(1+\chi)$ \\
partial sum in a row of $Ap_i$ & $P$ \\
partial sum in $p_i^TAp_i$ & $P^2$ \\
partial sum in $r_i^Tr_i$ & $\kappa_2(A)\chi^2$ \\
$a_i^{\rm cg}$, $b_{i+1}^{\rm cg}$ & $2\kappa_2(A)$, $G$ \\
componentwise products in vector updates & $(2\kappa_2(A)+G)P$
\end{tabular}
\end{center}
Cauchy--Schwarz bounds a partial row product by $P$, since every
partial row has norm at most $M\le1$, and gives the inner-product
bounds in the same way.  The definition of $U$ bounds these
quantities, the divisors, and the updated components.
Initialization uses at most $2n^2$ operations.  An iteration uses
$2n^2-n$ for the matrix-vector product, three inner products of
$2n-1$ operations, three AXPYs of $2n$ operations, and two divisions.
Therefore
\begin{equation}
 2n^2+n(2n^2+11n-1)=N
 \label{eq:operation-count}
\end{equation}
bounds the count through the required iterate.

We now compare stored scalars with their exact counterparts.
If previous errors are at most $E\le\min\{1,d/2\}$, the computed
operands have magnitude at most $U+1$ and the divisors are at least
$d/2$.  Addition and subtraction propagate at most $2E$,
multiplication at most $2UE+E^2$, and a fused multiply-add adds at
most one further $E$.  For division,
\[
 \left|\frac{\widehat a}{\widehat b}-\frac ab\right|
 \le\frac{2E}{d}+\frac{2UE}{d^2}.
\]
Including the final rounding and using $C=8U^2/d^2$ gives
\begin{equation}
 E_{\rm new}\le C(E+u+\omega).
 \label{eq:circuit-step}
\end{equation}
Assumption~\ref{ass:cg-range} and the definition of $H$ cover the
unrounded results of these operations.  Induction from stored inputs gives
\begin{equation}
 E_t\le(u+\omega)\sum_{s=1}^tC^s
 \le2(u+\omega)C^t\le\delta\qquad(t\le N).
 \label{eq:circuit-induction}
\end{equation}
The smallness assumptions thus keep the computed divisors positive
and yield
\begin{equation}
 \|\widehat x_j-x_j\|_2\le\sqrt n\,\delta\le\varepsilon_0/20.
 \label{eq:x-perturbation}
\end{equation}

Set $D=M\|x_j\|_2$.  Equation~\eqref{eq:first-good-exact} gives
\begin{equation}
 B\le(1+\varepsilon_0/4)D,\qquad D\ge2/5.
 \label{eq:D-floor}
\end{equation}
Since $M\|\widehat x_j-x_j\|_2\le\varepsilon_0D/8$,
\[
 \|b-A\widehat x_j\|_2\le3\varepsilon_0D/8,\qquad
 M\|\widehat x_j\|_2\ge7D/8>0.
\]
It remains to verify the stopping calculation.  Every stored iterate
through $j$ has norm less than $X$.  Partial row products in its
residual test are bounded by $X$, residual components by $1+X$,
and partial squared norms by $(1+X)^2$.
Thus $V=4(2+X)^2$ bounds the exact test intermediates and their
unit neighborhoods.  Forming the residual takes at most $2n^2$
operations, the two squared norms take $4n-2$, and the remaining
scaling and subtraction take four.  The same induction, now with
$C_{\rm s}$ and per-operation rounding bound $\rho_{\rm s}$, gives
\begin{equation}
 |\widehat\Phi_i-\Phi(\widehat x_i)|\le E_{\rm s}
 \le\varepsilon_0^2/256.
 \label{eq:selector-certified-error}
\end{equation}
Each test begins from the current stored iterate, independently of
earlier tests.

At $j$, the residual and iterate bounds and $M\le1$ give
\[
 \Phi(\widehat x_j)
 \le\left(\frac{3\varepsilon_0D}{8}\right)^2
 -\frac{\varepsilon_0^2}{4}\left(\frac{7D}{8}\right)^2
 =-\frac{13}{256}\varepsilon_0^2D^2.
\]
Since $D\ge2/5$,
\begin{equation}
 \Phi(\widehat x_j)\le-\frac{13}{1600}\varepsilon_0^2.
 \label{eq:selector-margin}
\end{equation}
Consequently
$\widehat\Phi_j\le-27\varepsilon_0^2/6400
<-\varepsilon_0^2/256$, so the test succeeds by step $j\le n$.
At any accepted index,
$\Phi(\widehat x_i)\le\widehat\Phi_i+E_{\rm s}\le0$, and hence
\[
 \|b-A\widehat x_i\|_2\le\frac{\varepsilon_0}{2}\|\widehat x_i\|_2
 \le\varepsilon_0M\|\widehat x_i\|_2.
\]
The assumption $b\ne0$ ensures that the accepted vector is nonzero. This proves
\eqref{eq:cg-eta-result}.
\end{proof}

\begin{corollary}[A sufficient significand precision]
\label{cor:cg-bit-count}
Assume all hypotheses of Theorem~\ref{thm:cg-bit-bound} except the two
smallness inequalities \eqref{eq:cg-smallness}.  In addition, let $\omega=0$
and suppose the exponent range is unbounded.  Define
\begin{equation}
 p_{\rm suff}=\left\lceil\max\left\{
  \log_2\frac{2}{\delta}+N\log_2C,
  \log_2\frac{1024(V+1)^2C_{\rm s}^{N_{\rm s}}}{\varepsilon_0^2}
 \right\}\right\rceil.
 \label{eq:cg-p-bound}
\end{equation}
If $u=2^{-p}$ and $p\ge p_{\rm suff}$, then
\eqref{eq:cg-smallness} holds.  Consequently,
Algorithm~\ref{alg:tested-cg} stops at an index $i_*\le n$ and returns a
nonzero vector satisfying
\[
 \eta(\widehat x_{i_*})
 \le\eta_A(\widehat x_{i_*})
 \le\varepsilon_0\le\varepsilon.
\]
When $x_0=0$,
\[
 p_{\rm suff}
 =O\!\left(n^3\log\left(2+\frac{n\kappa_2(A)}{\varepsilon}\right)\right).
\]
\end{corollary}

\begin{proof}
The first logarithm in \eqref{eq:cg-p-bound} implies the first inequality in
\eqref{eq:cg-smallness}.  Since
$E_{\rm s}=4u(V+1)^2C_{\rm s}^{N_{\rm s}}$, the second logarithm implies the
second inequality in \eqref{eq:cg-smallness}.

For $x_0=0$, one has $\chi=1$.  The definitions give
$P=G=\kappa_2(A)$ and hence
$\log C=O(\log(2+\kappa_2(A)/\varepsilon))$.  Since $N=O(n^3)$, the first term
in \eqref{eq:cg-p-bound} is
$O(n^3\log(2+n\kappa_2(A)/\varepsilon))$.  The second term is
$O(n^2\log(2+\kappa_2(A)/\varepsilon))$ and is no larger in order.
\end{proof}

In a finite-range format, the required exponent range has magnitude of order
$\log_2 H=O\bigl(\log(2+\kappa_2(A)/\varepsilon)\bigr)$ for $x_0=0$, in addition
to the exact scaling requirements stated before Assumption~\ref{ass:cg-range}.

The theorem controls the true residual through at most $n+1$ independent
stopping tests.  Each test begins with the current stored iterate and
uses its own $N_{\rm s}$-operation error bound.

The next two-dimensional example explains why some dependence on conditioning
is unavoidable for the chosen accumulation order.

\begin{proposition}[An SPD matrix with a zero first CG denominator]
\label{prop:cg-curvature-loss}
Fix a binary significand precision $p\ge2$, set $u=2^{-p}$, and define
\[
 A_u=\begin{bmatrix}1+2u&-1\\-1&1-u\end{bmatrix},
 \qquad b_u=(1-u,1)^T,\qquad x_0=0.
\]
Apply \eqref{eq:hscg}, with $r_0=p_0=b_u$, using the left-to-right
matrix-vector product in which every scalar product is rounded and stored
before the following addition.  No fused multiply-add is used in this
matrix-vector product.  Then $A_u$ is symmetric
positive definite and
\begin{equation}
 \frac1u<\kappa_2(A_u)\le\frac{81}{8u}.
 \label{eq:cg-curvature-kappa}
\end{equation}
For this evaluation order,
\[
 \fl(A_up_0)=0,
 \qquad
 \fl\!\left(p_0^T\fl(A_up_0)\right)=0.
\]
Thus the first CG step has a zero computed denominator and cannot
produce a finite iterate $x_1$.  The normwise backward error of $x_0$ is
$\eta(x_0)=1$.
\end{proposition}

\begin{proof}
The leading principal minor is positive and
$\det(A_u)=u(1-2u)>0$, proving positive definiteness.  The recurrence
starts from $r_0=p_0=b_u$.  The first row product satisfies
$1<(1+2u)(1-u)=1+u-2u^2<1+u$, so it rounds to one before the
subtraction.  Consequently,
\[
 \fl(A_up_0)_1
 =\fl\!\left(\fl((1+2u)(1-u))-1\right)=0,
 \qquad
 \fl(A_up_0)_2=\fl(-(1-u)+(1-u))=0.
\]
Thus the first stored CG denominator is zero, while its squared-norm
numerator is positive.  This causes breakdown at the first division.
At $x_0=0$, the backward error in
\eqref{eq:backward-error} equals one.

Let $\lambda_+$ be the larger eigenvalue.  Since the determinant is the product
of the eigenvalues,
\[
 \kappa_2(A_u)=\frac{\lambda_+^2}{u(1-2u)}.
\]
The Rayleigh quotient at the first coordinate gives $\lambda_+\ge1+2u>1$,
while $u(1-2u)<u$, which proves the lower bound in
\eqref{eq:cg-curvature-kappa}.  Since $u\le1/4$,
\[
 \lambda_+\le\operatorname{tr}(A_u)=2+u\le\frac94,
 \qquad 1-2u\ge\frac12.
\]
The upper bound follows.
\end{proof}

\begin{corollary}[A worst-case condition-number lower bound]
\label{cor:cg-curvature-bit-lower}
Fix $0<\varepsilon<1$ and $K\ge1$.  Suppose an integer precision $p\ge2$
guarantees that, for every representable two-dimensional SPD system
with $\kappa_2(A)\le K$, the nonfused left-to-right CG computation in
Proposition~\ref{prop:cg-curvature-loss} returns an iterate $x$
satisfying $\eta(x)\le\varepsilon$.  Then
\begin{equation}
 p>\log_2K-\log_2(81/8).
 \label{eq:cg-curvature-bit-lower}
\end{equation}
\end{corollary}

\begin{proof}
If $(81/8)2^p\le K$, the matrix $A_{2^{-p}}$ from
Proposition~\ref{prop:cg-curvature-loss} belongs to the stated class.
Its $p$-bit computation breaks down before producing an iterate with
backward error less than one, contradicting the guarantee.
Therefore $(81/8)2^p>K$, which gives the asserted inequality.
\end{proof}

At $\varepsilon=1$, the initial iterate meets the normwise criterion, so the
lower conclusion excludes that endpoint.  Evaluating the first row as the
single fused multiply-add $\fl((1+2u)(1-u)-1)$ preserves its positive
value $u(1-2u)$ and avoids this zero-denominator failure.
Thus the operation order is part of the mathematical statement.

Druskin, Greenbaum, and Knizhnerman studied convergence rates and
attainable accuracy for finite-precision Lanczos
\cite{DruskinGreenbaumKnizhnerman1998}.  Musco, Musco, and Sidford give
a bit-precision analysis for matrix-function approximation
\cite{MuscoMuscoSidford2018}.  Mixed-precision CG is studied by Bake,
Carson, and Ma \cite{BakeCarsonMa2026}, while iterative refinement and
other mixed-precision techniques provide additional ways to allocate
precision \cite{CarsonHigham2018,HighamMary2022}.
Chenakkod, Derezi\'nski, Dong, and Rudelson obtain logarithmic-bit
backward-error guarantees with constant success probability for CG
applied to normal equations after a randomized perturbation of the input
\cite{ChenakkodDerezinskiDongRudelson2026}.
The result here concerns the specified recurrence on the original SPD
input, with a recomputed-residual stopping test and an $n$-update budget.
The sufficient bound and the worst-case condition-number obstruction
leave a quantitative gap that depends on the chosen evaluation order.
\subsection{Precision and the iteration budget}
\label{sec:two-regimes}

The sufficient bound in Theorem~\ref{thm:cg-bit-bound} and the
condition-number obstruction in Proposition~\ref{prop:cg-curvature-loss}
concern a prescribed iteration budget.  Allowing more steps changes the
approximation problem.  This can be seen both in polynomial estimates
and in computed CG trajectories.

Musco, Musco, and Sidford show that Lanczos can approximate a matrix
function using a number of significand bits logarithmic in the relevant
dimension, accuracy, and conditioning parameters
\cite{MuscoMuscoSidford2018}.  For inversion on a positive spectral
interval, Chebyshev approximation gives an iteration bound of order
$\sqrt{\kappa_2(A)}\log(\kappa_2(A)/\varepsilon)$, with logarithmic
precision under their hypotheses.  Chen, Greenbaum, Musco, and Musco
develop further matrix-function error bounds
\cite{ChenGreenbaumMuscoMusco2022}.  These results express the tradeoff
between the approximation degree and the accuracy of the arithmetic.
Bit-complexity analyses also study this tradeoff for sparse linear
systems, matrix anti-concentration, and continuous optimization
\cite{PengVempala2021,Nie2022,GhadiriPengVempala2023},
with related linear-system algorithms developed by Derezi\'nski and
Yang \cite{DerezinskiYang2024}.

In exact arithmetic, CG terminates once a residual polynomial vanishes
at every relevant eigenvalue.  Nearby-problem models replace individual
eigenvalues by small spectral clusters.  A polynomial that vanishes at
the original eigenvalues can then vary rapidly across the clusters.
The following interpolation estimate quantifies that variation.

\begin{lemma}[Uniform smallness on disjoint intervals]
\label{lem:intervals}
Let $n\ge2$, let $0<\lambda_1<\cdots<\lambda_n$, and suppose
\[
 0<\delta<\min\left\{\lambda_1,\frac12
                   \min_{1\le j<n}(\lambda_{j+1}-\lambda_j)\right\}.
\]
Define $E=\bigcup_{j=1}^n[\lambda_j-\delta,\lambda_j+\delta]$.
Let $\ell_0,\ldots,\ell_n$ be the Lagrange basis polynomials for the
nodes $x_0=0$ and $x_j=\lambda_j$, and set
\[
 \Lambda(\delta)=1+\sum_{j=1}^n|\ell_j(\lambda_n+\delta)|.
\]
Every polynomial $q$ of degree at most $n$ satisfying $q(0)=1$ obeys
\begin{equation}
 \max_{x\in E}|q(x)|
 \ge\frac{|\ell_0(\lambda_n+\delta)|}{\Lambda(\delta)}
 =\frac{\delta}{\lambda_n\Lambda(\delta)}
   \prod_{j=1}^{n-1}\left(\frac{\lambda_n+\delta}{\lambda_j}-1\right).
 \label{eq:interval-lower-bound}
\end{equation}
For fixed nodes, $\Lambda(\delta)\to2$ as $\delta\to0$.
\end{lemma}

\begin{proof}
Lagrange interpolation at the $n+1$ nodes gives
\[
 q(x)=\ell_0(x)+\sum_{j=1}^nq(\lambda_j)\ell_j(x),
 \qquad \ell_0(x)=\prod_{j=1}^n(1-x/\lambda_j).
\]
Put $\varepsilon_E=\max_{x\in E}|q(x)|$ and evaluate at
$x=\lambda_n+\delta$.  The triangle inequality gives
\[
 |\ell_0(\lambda_n+\delta)|
 \le |q(\lambda_n+\delta)|
      +\sum_{j=1}^n|q(\lambda_j)||\ell_j(\lambda_n+\delta)|
 \le\varepsilon_E\Lambda(\delta).
\]
Division proves the lower bound, and the product formula for $\ell_0$
gives its displayed value.  As $\delta\to0$, the value
$\ell_n(\lambda_n+\delta)$ tends to one and the other terms in the sum
tend to zero.
\end{proof}

For a fixed spectrum, the first-order coefficient in this lower bound
is
\[
 \lim_{\delta\to0}
 \frac{1}{\delta}\,
 \frac{|\ell_0(\lambda_n+\delta)|}{\Lambda(\delta)}
 =\frac{1}{2\lambda_n}
   \prod_{j=1}^{n-1}\left(\frac{\lambda_n}{\lambda_j}-1\right).
\]
When this product is large, uniform accuracy on all the intervals
requires a small interval width.  This is a precise sensitivity
statement about polynomial approximation.  To apply it to CG, the
spectral weights of the enlarged problem and the actual cluster
locations must also be taken into account, since CG minimizes a weighted
error over those eigenvalues.  Greenbaum's spectral enclosure gives
an upper bound on the possible spread \cite{Greenbaum1989}.
Musco, Musco, and Sidford analyze the predictive limitations of
interval-based polynomial bounds \cite{MuscoMuscoSidford2018}.

\begin{figure}[t]
\centering
\includegraphics[width=0.75\textwidth]{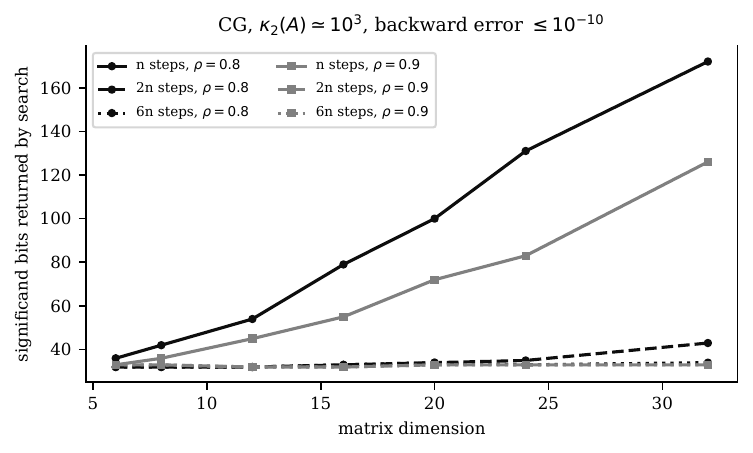}
\caption{Significand precisions returned by bisection for
Hestenes--Stiefel CG on the displayed dense Strako\v{s} test matrices,
with $\lambda_1=0.1$, $\lambda_n=100$, and two spectral-spacing
parameters.  The target is normwise backward error $10^{-10}$ within
$n$, $2n$, or $6n$ steps.  Each input matrix and right-hand side is held
fixed during its search. For the backward-error test, the residual and
vector norms are evaluated in 600-bit arithmetic from the stored iterate.
The matrix norm $\|A\|_2$ is estimated in binary64.}
\label{fig:bits}
\end{figure}

Figure~\ref{fig:bits} compares these budgets on fourteen test matrices
with $\kappa_2(A)=10^3$.  For $\rho=0.8$ and dimensions
$n=6,8,12,16,20,24,32$, the search returns
$36,42,54,79,100,131,172$ bits for the $n$-step budget.
For $\rho=0.9$, it returns $33,36,45,55,72,83,126$ bits.
The precisions found with $2n$ or $6n$ steps lie between $32$ and $43$
bits across these matrices.  Each reported precision meets the target
under this independently evaluated backward-error test.

The success criterion can vary nonmonotonically with precision.  For the
$n=12$, $\rho=0.8$ matrix and a twelve-step budget, the best backward
errors at $p=54,55,56,57$ are approximately
$7.73\times10^{-11}$, $1.90\times10^{-11}$,
$1.93\times10^{-10}$, and $4.16\times10^{-11}$.
The figure therefore reports the successful precisions found by the
specified search.

The numerical comparison shows a substantial reduction in arithmetic
precision when the iteration budget is increased on these examples.
The interpolation estimate describes sensitivity in spectral-cluster models,
and the sufficient theorem controls the specified floating-point computation
directly, giving complementary reasons to include both the stopping criterion
and the iteration budget in a precision analysis.
\section{Inexact matrix-vector products and preconditioning}
\label{sec:inexact}

Linear operators and preconditioners are often applied approximately.
Nested iterative solves are one example. Such computations introduce an
operator error in addition to rounding, and inexact Krylov methods relate
the allowable error to the progress of the outer iteration
\cite{BourasFraysse2005,SimonciniSzyld2003,VanDenEshofSleijpen2004}.
Problem 2.18 asks for extensions of Paige's analysis under these more
general perturbations, together with an accessible explanation of their
effect on Lanczos.

Chen's analysis \cite{Chen2026Simple} uses additive errors in symmetric
matrix-vector products. For $v\in\R^n$ and an accuracy parameter
$\varepsilon>0$, it assumes
\begin{equation}
 \fl(Av)=Av+e_v,\qquad
 \|e_v\|_2\le\varepsilon\|A\|_2\|v\|_2,
 \label{eq:inexact-product-model}
\end{equation}
together with normwise bounds for inner products, vector updates, and
normalization. Under smallness conditions involving $\varepsilon$ and the
step count, Chen obtains the perturbed three-term recurrence and Paige's
loss-of-orthogonality identity. The resulting estimates describe Ritz-value
containment and stabilization, followed by a Greenbaum-type interpretation
in terms of a nearby symmetric problem. They extend the scalar framework
of Paige and Greenbaum \cite{Paige1976,Paige1980,Greenbaum1989} to the
stated inexact-product model.

Derezi\'nski, Musco, and Yang \cite[Sec.~6]{DerezinskiMuscoYang2025}
analyze Lanczos with an inexact positive definite preconditioner.
Let $A,M\in\R^{n\times n}$ be symmetric positive definite and define
$\|v\|_M=(v^TMv)^{1/2}$. For each right-hand side $r\in\R^n$, suppose
the approximate solution $\widetilde z_r$ of $Mz=r$ satisfies
\begin{equation}
 \|\widetilde z_r-M^{-1}r\|_M
 \le\epsilon_0\|M^{-1}r\|_M.
 \label{eq:inexact-preconditioner}
\end{equation}
Set $\kappa_M=\kappa_2(M^{-1/2}AM^{-1/2})$ and let $0<\epsilon<1$ be
the requested relative error. Their Theorem 6.1 shows that, for a fixed
constant $c>0$, the condition
\begin{equation}
 \epsilon_0\le\left(\frac{\epsilon}{\kappa_M n}\right)^c
 \label{eq:dmy-accuracy}
\end{equation}
is sufficient for their algorithm to return $x_t$ with
\begin{equation}
 \|x_t-A^{-1}b\|_A\le\epsilon\|A^{-1}b\|_A,
 \qquad
 t=O\!\left(\sqrt{\kappa_M}\log\frac{\kappa_M}{\epsilon}\right).
 \label{eq:dmy-convergence}
\end{equation}
The reference matrix $M$ is fixed. The errors of successive solves may
differ, provided each obeys \eqref{eq:inexact-preconditioner}.

These results continue a substantial literature on inexact projections.
Bouras and Frayss\'e, Simoncini and Szyld, and van den Eshof and Sleijpen
developed conditions under which product accuracy can be relaxed as the
linear-system residual decreases
\cite{BourasFraysse2005,SimonciniSzyld2003,VanDenEshofSleijpen2004}.
Giraud, Gratton, and Langou analyzed the associated backward error
\cite{GiraudGrattonLangou2007}. Eigenvalue computations require a related
but distinct treatment, given for inexact Lanczos and Arnoldi by Simoncini
\cite{Simoncini2005}.
For CG, Golub and Ye \cite{GolubYe1999}, Notay \cite{Notay2000}, and
Knyazev and Lashuk \cite{KnyazevLashuk2007} studied inexact or variable
preconditioning. Strako\v{s} and Tich\'y incorporated errors in
preconditioner solves into error estimation \cite{StrakosTichy2005}.
Variable-precision products are analyzed by Gratton et al.\
\cite{GrattonSimonTitleyPeloquinToint2021}, while Zemke's perturbed-Krylov
formulation treats these recurrences in a common algebraic setting \cite{Zemke2007}.

\subsection{Block identities under inexact operations}

The block Paige identities in Theorem~\ref{thm:block-paige} depend on the
collected recurrence and symmetry. The same algebra therefore applies when
the local error contains an inexact product. We state this observation with
the total perturbation made explicit.

\begin{corollary}[Paige identities under inexact operations]
\label{cor:inexact-paige}
Let $A=A^T\in\R^{n\times n}$, let $k\ge1$, choose positive integers
$r_1,\ldots,r_k$, and let
$V_i\in\R^{n\times r_i}$ have orthonormal columns for $1\le i\le k$.
Put $m=\sum_{i=1}^k r_i$ and $V=[V_1\ \cdots\ V_k]$.
Let $E_k^{\rm inj}\in\R^{m\times r_k}$ be the coordinate injection into
the final block. Suppose $T=T^T\in\R^{m\times m}$ is block tridiagonal
and that $W\in\R^{n\times s}$ and $H\in\R^{s\times r_k}$ satisfy
$W^TW=I_s$, where $0\le s\le r_k$.

For each block, suppose the inexact product is $AV_i+E_i$, with
$\|E_i\|_2\le\varepsilon_i\|A\|_2$ and $\varepsilon_i\ge0$.
Let $F_{\rm round}\in\R^{n\times m}$ collect the other local errors, so
that the computation satisfies
\begin{equation}
 AV=VT+WH(E_k^{\rm inj})^T+F,\qquad
 F=F_{\rm round}-[E_1\ \cdots\ E_k].
 \label{eq:inexact-collected}
\end{equation}
Then the block Paige identities and bounds of
Theorem~\ref{thm:block-paige} hold with this $F$, and
\begin{equation}
 \|F\|_2\le\|F_{\rm round}\|_2+
 \|A\|_2\left(\sum_{i=1}^k\varepsilon_i^2\right)^{1/2}
 \le\|F_{\rm round}\|_2+\sqrt{k}\max_i\varepsilon_i\|A\|_2.
 \label{eq:inexact-total-error}
\end{equation}
\end{corollary}

\begin{proof}
The collected relation has precisely the dimensions, symmetry, and
within-block orthogonality required by Theorem~\ref{thm:block-paige}.
Its algebraic identities therefore apply with the stated total error.
For $x=\operatorname{col}(x_1,\ldots,x_k)\in\R^m$, where
$x_i\in\R^{r_i}$, the Cauchy--Schwarz inequality gives
\[
 \left\|\sum_i E_i x_i\right\|_2
 \le\left(\sum_i\|E_i\|_2^2\right)^{1/2}\|x\|_2.
\]
Taking the supremum over unit $x$ and using the triangle inequality proves
\eqref{eq:inexact-total-error}.
\end{proof}

For a fixed symmetric positive definite preconditioner $M$, the same
argument applies in the $M$-inner product. To see this, define
\[
 \overline A=M^{-1/2}AM^{-1/2},\qquad
 \overline V_i=M^{1/2}V_i.
\]
The matrix $M^{-1}A$ is self-adjoint in that inner product, and
$V_i^TMV_i=I_{r_i}$ becomes $\overline V_i^T\overline V_i=I_{r_i}$.
All estimates can thus be applied to the symmetric matrix $\overline A$.
If $v^TMv=1$ and the solve for $Av$ obeys
\eqref{eq:inexact-preconditioner}, its error $d_v$ satisfies
\[
 \|d_v\|_M\le\epsilon_0\|M^{-1}Av\|_M
 \le\epsilon_0\|\overline A\|_2.
\]
For a block of $r_i$ such columns, these individual bounds give a block
error of at most $\sqrt{r_i}\epsilon_0\|\overline A\|_2$ in the
transformed coordinates. This accounts for the width of the block when
columnwise solve tolerances are used in \eqref{eq:inexact-total-error}.
\section{Block Lanczos, rank deflation, and Ritz values}

Throughout this section, $A=A^T\in\R^{n\times n}$ is a real symmetric
operator.  Block Lanczos replaces single vectors by orthonormal matrix
blocks.  When a residual block is rank deficient, an orthonormal basis of its
column space forms the next block, whose size is reduced accordingly.  Let $r_i$ denote the
number of columns in block $i$.  In exact arithmetic, one step satisfies
\begin{equation}
 AV_i=V_{i-1}B_i^T+V_iA_i+V_{i+1}B_{i+1},
 \qquad V_i^TV_i=I_{r_i},\quad V_j^TV_i=0\ (j\ne i),\quad A_i=A_i^T,
 \label{eq:ideal-ragged-block}
\end{equation}
where $B_{i+1}\in\R^{r_{i+1}\times r_i}$.  Numerical deflation may give
$r_{i+1}<r_i$.  Our convention is that $B_{i+1}$ maps coordinates in block
$i$ to coordinates in block $i+1$.  The main quantities and their dimensions are
\begin{center}
\begin{tabular}{c|c|l}
symbol & size & meaning \\ \hline
$V_i$ & $n\times r_i$ & current basis block \\
$A_i$ & $r_i\times r_i$ & symmetric diagonal block \\
$B_{i+1}$ & $r_{i+1}\times r_i$ & coupling to the next block \\
$F_i$ & $n\times r_i$ & local perturbation
\end{tabular}
\end{center}
We first establish a finite-precision recurrence and bound its local error
$F_i$.  The estimate remains applicable as singular directions are
discarded.  Local loss of orthogonality is described by the two adjacent-block products
\[
 B_{i+1}V_i^TV_{i+1}
 \quad\hbox{and}\quad
 V_i^TV_{i+1}B_{i+1}.
\]
Their different multiplication orders require separate estimates.  The Gram
matrix then connects rank deflation with the distinction between independent
block Ritz vectors and repeated numerical approximations to the same
physical directions.

Related questions arise in block CG, from O'Leary's original formulation
and the variable-block schemes of Nikishin and Yeremin to Dubrulle's
reformulation \cite{OLeary1980,NikishinYeremin1995,Dubrulle2001}.
Later developments treat deflation with multiple right-hand sides and
recurrences that continue through rank deficiency
\cite{BirkFrommer2014,JiLi2017}.  In each setting, the width of a stored
block and the dimension it adds to the computed space are quantities that
the algorithm must distinguish.

\subsection{Finite-precision arithmetic in a block step}

We analyze Householder orthogonalization and a compact SVD under the
rounding-error assumptions stated below.
At block step $i$, the stored current block is
$\widetilde V_i\in\R^{n\times r_i}$.  Put
$Y_i=\fl(A\widetilde V_i)$ and
$\widehat A_i=\fl(\widetilde V_i^TY_i)$.  Form the symmetric diagonal block by
averaging corresponding entries,
\[
 (A_i)_{ab}=(A_i)_{ba}
 =\fl\!\left(\frac{(\widehat A_i)_{ab}+(\widehat A_i)_{ba}}{2}\right),
 \qquad 1\le a\le b\le r_i.
\]
With componentwise absolute values, define the computable bounds
\begin{align}
 \eps_i^A&=\|\gamma_n|A||\widetilde V_i|\|_2,
 \label{eq:matvec-error-bound}\\
 \eps_i^\alpha&=
 \|\widetilde V_i\|_2\eps_i^A
 +\|\gamma_n|\widetilde V_i|^T|Y_i|\|_2
 +\left\|\frac{u}{2}
   (|\widehat A_i|+|\widehat A_i|^T)\right\|_2.
 \label{eq:alpha-error-bound}
\end{align}
With $\widetilde V_0B_1^T=0$, form the residual block before orthogonalization,
\begin{equation}
 R_i=\fl\!\left(
  Y_i-\widetilde V_{i-1}B_i^T-\widetilde V_iA_i\right).
 \label{eq:raw-block-residual}
\end{equation}
For later error bounds, define the two stored matrix products
\[
 \widehat P_i^-=\fl(\widetilde V_{i-1}B_i^T),
 \qquad
 \widehat P_i^0=\fl(\widetilde V_iA_i).
\]
The standard matrix-product bound and the error bound for two successive
subtractions give
\begin{equation}
\begin{split}
 \eps_i^R={}&\eps_i^A+
 \bigl\|\gamma_{r_{i-1}}|\widetilde V_{i-1}||B_i|^T
       +\gamma_{r_i}|\widetilde V_i||A_i|\\
 &\hspace{17mm}+\gamma_2(
       |Y_i|+|\widehat P_i^-|+|\widehat P_i^0|)\bigr\|_2.
\end{split}
\label{eq:raw-error-bound}
\end{equation}

\begin{assumption}[Arithmetic in one block step]
\label{ass:block-arithmetic}
At block step $i$, no underflow or overflow occurs, and
\[
 n u<1,\qquad r_i u<1,\qquad r_{i-1}u<1,\qquad 2u<1.
\]
For $i=1$, set $r_0=0$ and $\gamma_0=0$.  If a stored matrix product
$\fl(XY)$ at this step has inner dimension $s$, then
\[
 |\fl(XY)-XY|\le\gamma_s|X||Y|
\]
componentwise.  For stored matrices $Y$, $P$, and $Q$, the two successive
subtractions satisfy
\[
 |\fl(Y-P-Q)-(Y-P-Q)|
 \le\gamma_2(|Y|+|P|+|Q|).
\]
Finally, each averaging operation used to form $A_i$ satisfies
\[
 \left|\fl\!\left(\frac{a+b}{2}\right)-\frac{a+b}{2}\right|
 \le\frac{u}{2}(|a|+|b|).
\]
\end{assumption}

Blocked matrix products and fused multiply-add instructions can change the
constants in these bounds.  Their standard componentwise form remains the
same when the corresponding implementation-dependent constants are used
\cite[Sec.~3.5]{Higham2002}.

\begin{lemma}[Errors in the diagonal and residual blocks]
\label{lem:block-arithmetic}
Fix $i\ge1$.  Set $Y_i=\fl(A\widetilde V_i)$ and
$\widehat A_i=\fl(\widetilde V_i^TY_i)$, form $A_i$ by the averaging rule
defined above, and form $R_i$ by \eqref{eq:raw-block-residual}.  Suppose
Assumption~\ref{ass:block-arithmetic} holds.

Then $A_i=A_i^T$.  Moreover, there are matrices
$E_i^\alpha\in\R^{r_i\times r_i}$ and
$E_i^R\in\R^{n\times r_i}$ such that
\begin{align}
 A_i&=\widetilde V_i^TA\widetilde V_i+E_i^\alpha,
 &\|E_i^\alpha\|_2&\le\eps_i^\alpha,
 \label{eq:alpha-error-identity}\\
 R_i&=A\widetilde V_i-\widetilde V_{i-1}B_i^T-
       \widetilde V_iA_i+E_i^R,
 &\|E_i^R\|_2&\le\eps_i^R.
 \label{eq:raw-error-identity}
\end{align}
Here $\widetilde V_0B_1^T$ denotes the zero matrix.
\end{lemma}

\begin{proof}
The standard product model bounds the error in $Y_i$ by $\eps_i^A$ and
the error in $\widehat A_i$ by
$\|\gamma_n|\widetilde V_i|^T|Y_i|\|_2$.  Symmetrization is a contraction since
\[
 \left\|\frac{X+X^T}{2}\right\|_2
 \le\frac{\|X\|_2+\|X^T\|_2}{2}=\|X\|_2.
\]
The final averaging operation contributes the last term in
\eqref{eq:alpha-error-bound}.  Consequently, the projection error satisfies
\eqref{eq:alpha-error-identity}.  Applying the product bounds with inner
dimensions $r_{i-1}$ and $r_i$, followed by the bound for two successive
subtractions, gives
\eqref{eq:raw-error-bound} and \eqref{eq:raw-error-identity}.
\end{proof}

Truncated economy-size SVDs are used in block Lanczos continuations
\cite[Sec.~6.2]{SimonovaTichy2025}.  For fixed-width blocks,
\v{S}imonov\'a and Tich\'y state bounds for the local recurrence error,
the right-weighted adjacent product, and the coupling norm.  They credit
these estimates to an unpublished 2022 report of Carson and Chen
\cite[Sec.~4, eqs.~(9), (11), (12)]{SimonovaTichy2025}.
Tich\'y, Meurant, and \v{S}imonov\'a also use Householder QR in block-CG
variants and an SVD for numerical rank selection
\cite{TichyMeurantSimonova2025}.

In the calculation studied here, Householder QR is applied first to
$\widetilde V_i$.  Its reflectors transform $R_i$ to coordinates along
the current block and its orthogonal complement.  A compact SVD of the
latter coordinates determines the next block.  Singular values exceeding
an absolute threshold are retained.  Choose a positive multiplier $\vartheta$
before the computation, independently of the retained block sizes, and set
\begin{equation}
 \tau_i=\vartheta u\max\{\|A\|_2,\|\fl(A\widetilde V_i)\|_2,
                  \|B_i\|_2,\|R_i\|_2\}.
 \label{eq:threshold}
\end{equation}
A certified upper bound for $\|A\|_2$ may be used in this definition, with a
corresponding adjustment of $\vartheta$.
The retained left singular vectors are then transformed back to the original
coordinates by the Householder reflectors.
The SVD calculation also returns a stored coupling block $B_{i+1}$.  Its
formation error and its rank are included in the assumptions below.

Orthogonalizing before rank selection controls both adjacent products
\[
 B_{i+1}\widetilde V_i^T\widetilde V_{i+1}
 \quad\hbox{and}\quad
 \widetilde V_i^T\widetilde V_{i+1}B_{i+1}.
\]
The bounds below follow directly from the residual factorization and depend
on the norm of the coupling block.  They therefore remain finite when a
retained singular value approaches the truncation threshold.

\subsection{Rounding-error models for Householder QR and the compact SVD}
\label{sec:block-algorithm}

Standard Householder analyses motivate the first two relations below
\cite[Lemma~19.3 and Theorem~19.4]{Higham2002}.  The LAPACK Users' Guide gives
a backward-error interpretation for the SVD 
\cite{LAPACKUsersGuide1999}.  The constants depend on the dimensions and on
the particular implementations.  The analysis below uses the following
properties.
The distinction between orthogonalization within a block and between
successive blocks is also central to the stability analyses of block
Gram--Schmidt \cite{CarsonLundRozloznik2021,CarsonLundRozloznikThomas2022}.

\begin{assumption}[Orthogonalization and rank truncation at block step $i$]
\label{ass:block-orthogonalization}
Let $1\le r_i\le n$.  There are an orthogonal matrix
$[Q_i,Q_{i,\perp}]\in\R^{n\times n}$, a matrix
$R_{i,11}\in\R^{r_i\times r_i}$ that is upper triangular, and perturbations
$E_i^V,E_i^H\in\R^{n\times r_i}$ such that
\begin{align}
 \widetilde V_i+E_i^V&=Q_iR_{i,11},
 &\|E_i^V\|_2&\le a_i^V\|\widetilde V_i\|_2,
 \label{eq:qr-relation}\\
 \begin{bmatrix}C_i\\G_i\end{bmatrix}
 &=[Q_i,Q_{i,\perp}]^T(R_i+E_i^H),
 &\|E_i^H\|_2&\le a_i^H\|R_i\|_2,
 \label{eq:reflector-relation}
\end{align}
where $C_i\in\R^{r_i\times r_i}$ and
$G_i\in\R^{(n-r_i)\times r_i}$.

The singular values used for the rank decision are the exact singular values
of $G_i+E_i^{\rm sv}$, where
\begin{equation}
 \|E_i^{\rm sv}\|_2\le a_i^{\rm sv}\|G_i\|_2.
 \label{eq:svd-value-backward-error}
\end{equation}
This assumption concerns the computed singular values.  The errors in the
computed singular vectors and their products are specified separately below.

Let $r_{i+1}\ge0$ be the number of retained left singular vectors, and let
$d_i\ge0$ be the number of discarded directions.  Suppose
\[
 U_{i,J}\in\R^{(n-r_i)\times r_{i+1}},\quad
 U_{i,D}\in\R^{(n-r_i)\times d_i},\quad
 \Sigma_{i,D}\in\R^{d_i\times d_i},\quad
 Z_{i,D}\in\R^{r_i\times d_i},
\]
where $\Sigma_{i,D}$ is diagonal with nonnegative entries, and suppose
$E_i^S\in\R^{(n-r_i)\times r_i}$ and
$E_i^M\in\R^{n\times r_{i+1}}$ satisfy
\begin{align}
 G_i&=U_{i,J}B_{i+1}
      +U_{i,D}\Sigma_{i,D}Z_{i,D}^T+E_i^S,
 &\|E_i^S\|_2&\le a_i^S\|G_i\|_2,
 \label{eq:svd-product-relation}\\
 \widetilde V_{i+1}&=Q_{i,\perp}U_{i,J}+E_i^M,
 &\|E_i^M\|_2&\le a_i^M.
 \label{eq:mapping-relation}
\end{align}
Here $B_{i+1}\in\R^{r_{i+1}\times r_i}$.  If $r_{i+1}>0$, assume
\begin{equation}
 \operatorname{rank}(B_{i+1})=r_{i+1},
 \qquad
 \|B_{i+1}\|_2\le(1+a_i^S)\|G_i\|_2.
 \label{eq:coupling-bound}
\end{equation}
Finally, suppose
\begin{align*}
 \|U_{i,J}^TU_{i,J}-I\|_2&\le a_i^U,\\
 \|U_{i,D}^TU_{i,D}-I\|_2&\le a_{i,D}^U,\\
 \|Z_{i,D}^TZ_{i,D}-I\|_2&\le a_{i,D}^Z,
\end{align*}
and every diagonal entry of $\Sigma_{i,D}$ is at most $\tau_i$.  All
constants in these bounds are nonnegative and independent of
$\sigma_{\min}(B_{i+1})$.
\end{assumption}

At the first step, the initial stored block is prepared by an
orthogonalization procedure satisfying the Gram bound used below.
Define
\begin{equation}
 \eta_i=a_i^U+2\sqrt{1+a_i^U}\,a_i^M+(a_i^M)^2.
 \label{eq:next-gram-bound}
\end{equation}
Also define
\begin{equation}
 \kappa_i^{\rm tail}
   =\sqrt{1+a_{i,D}^U}\sqrt{1+a_{i,D}^Z}.
 \label{eq:tail-factor}
\end{equation}

\begin{lemma}[Residual factorization after rank truncation]
\label{lem:thresholded-factorization}
Suppose Assumption~\ref{ass:block-orthogonalization} holds at block step $i$.
Define
\[
 h_i=\left[a_i^H+a_i^S(1+a_i^H)
       +a_i^M(1+a_i^S)(1+a_i^H)\right]\|R_i\|_2
\]
and
\begin{equation}
 E_i^{\rm fac}=Q_{i,\perp}
  \left(U_{i,D}\Sigma_{i,D}Z_{i,D}^T+E_i^S\right)
  -E_i^M B_{i+1}-E_i^H.
 \label{eq:factorization-error-definition}
\end{equation}
Then
\begin{equation}
 R_i=Q_iC_i+\widetilde V_{i+1}B_{i+1}+E_i^{\rm fac},
 \qquad
 \|E_i^{\rm fac}\|_2
 \le h_i+\kappa_i^{\rm tail}\tau_i.
 \label{eq:factorization-residual}
\end{equation}
The next block satisfies
\[
 \|\widetilde V_{i+1}^T\widetilde V_{i+1}-I\|_2\le\eta_i.
\]
\end{lemma}

\begin{proof}
Every discarded computed singular value is at most $\tau_i$, so
\[
 \|U_{i,D}\Sigma_{i,D}Z_{i,D}^T\|_2
 \le\kappa_i^{\rm tail}\tau_i.
\]
Substitution of \eqref{eq:svd-product-relation} and
\eqref{eq:mapping-relation} into \eqref{eq:reflector-relation} gives the exact
factorization identity.  Moreover,
$\|G_i\|_2\le(1+a_i^H)\|R_i\|_2$.  The product, mapping, and Householder bounds
then give the displayed value of $h_i$.
Expanding \eqref{eq:mapping-relation} and using
$\|U_{i,J}^TU_{i,J}-I\|_2\le a_i^U$ gives
the stated Gram-matrix bound with $\eta_i$ defined by
\eqref{eq:next-gram-bound}.
\end{proof}

To bound the component $C_i$, put
\[
 \delta_i=\|\widetilde V_i^T\widetilde V_i-I\|_2,
 \qquad
 \rho_i^-=
 \|B_i\widetilde V_{i-1}^T\widetilde V_i\|_2,
 \qquad
 \ell_i^{\rm qr}=\sqrt{1-\delta_i}
                  -a_i^V\sqrt{1+\delta_i}.
\]
For $i=1$, set $\rho_1^-=0$.  If $\delta_i<1$ and
$\ell_i^{\rm qr}>0$, define
\[
 \kappa_i=(\ell_i^{\rm qr})^{-1}
\]
and
\begin{equation}
\begin{split}
 c_i={}&\kappa_i\bigl(
 \rho_i^-+\eps_i^\alpha+\delta_i\|A_i\|_2
 +\sqrt{1+\delta_i}\,\eps_i^R\\
 &\hspace{29mm}
 +a_i^V\sqrt{1+\delta_i}\,\|R_i\|_2\bigr)
 +a_i^H\|R_i\|_2.
\end{split}
\label{eq:ci-definition}
\end{equation}

\begin{lemma}[Bound for the component along the current block]
\label{lem:current-block-coefficient}
Suppose the hypotheses of Lemma~\ref{lem:block-arithmetic} and
Assumption~\ref{ass:block-orthogonalization} hold at block step $i$.  If
$\delta_i<1$ and $\ell_i^{\rm qr}>0$, then
\begin{equation}
 \|C_i\|_2\le c_i.
 \label{eq:ci-bound}
\end{equation}
\end{lemma}

\begin{proof}
The QR relation and the Gram bound imply
\[
 \sigma_{\min}(R_{i,11})
 \ge \sqrt{1-\delta_i}-a_i^V\sqrt{1+\delta_i}.
\]
Moreover,
\[
 Q_i^TR_i
 =R_{i,11}^{-T}(\widetilde V_i+E_i^V)^TR_i.
\]
Insert \eqref{eq:raw-error-identity} in the last expression.  The terms in
$\widetilde V_i^TR_i$ are bounded, in order, by $\rho_i^-$,
$\eps_i^\alpha$, $\delta_i\|A_i\|_2$, and
$\sqrt{1+\delta_i}\,\eps_i^R$.  The term containing $E_i^V$ is at most
$a_i^V\sqrt{1+\delta_i}\|R_i\|_2$.  Finally,
$C_i=Q_i^T(R_i+E_i^H)$ and
$\|Q_i^TE_i^H\|_2\le a_i^H\|R_i\|_2$.  These inequalities give
\eqref{eq:ci-bound}.
\end{proof}

The next theorem combines these estimates to obtain the computed block
Lanczos relation.  It also controls the loss of orthogonality between
consecutive blocks.  Since $B_{i+1}$ multiplies
$\widetilde V_i^T\widetilde V_{i+1}$ from the left in one relation and from
the right in another, define
\begin{equation}
 q_i=\left(
 a_i^V\|\widetilde V_i\|_2\sqrt{1+a_i^U}
 +a_i^M\|\widetilde V_i\|_2\right)\|B_{i+1}\|_2.
 \label{eq:qi-definition}
\end{equation}
The two terms in $q_i$ account, respectively, for the QR error and for the
error made when the next block is formed.

\begin{theorem}[A computed variable-block Lanczos recurrence]
\label{thm:local-block}
Fix a block step $i$.  Set $Y_i=\fl(A\widetilde V_i)$ and
$\widehat A_i=\fl(\widetilde V_i^TY_i)$, form $A_i$ by the averaging rule above, and
form $R_i$ by \eqref{eq:raw-block-residual}.  Suppose
Assumptions~\ref{ass:block-arithmetic} and
\ref{ass:block-orthogonalization} hold and
\[
 \delta_i<\frac12,
 \qquad \eta_i<\frac12,
 \qquad \ell_i^{\rm qr}>0.
\]
Then $A_i=A_i^T$, and the following statements hold.

There is a matrix $\widetilde F_i\in\R^{n\times r_i}$ such that
\begin{equation}
 A\widetilde V_i=\widetilde V_{i-1}B_i^T+\widetilde V_iA_i
   +\widetilde V_{i+1}B_{i+1}+\widetilde F_i,
 \label{eq:stored-block-recurrence}
\end{equation}
with
\begin{equation}
 \|\widetilde F_i\|_2
 \le c_i+h_i+\kappa_i^{\rm tail}\tau_i+\eps_i^R.
 \label{eq:local-F-bound}
\end{equation}
The two adjacent-block products satisfy
\begin{equation}
 \max\!\left\{
 \|B_{i+1}\widetilde V_i^T\widetilde V_{i+1}\|_2,
 \|\widetilde V_i^T\widetilde V_{i+1}B_{i+1}\|_2
 \right\}\le q_i.
\label{eq:two-orders}
\end{equation}
The next block satisfies
\[
 \|\widetilde V_{i+1}^T\widetilde V_{i+1}-I\|_2
 \le\eta_i<\frac12.
\]
\end{theorem}

\begin{proof}
Equation \eqref{eq:qr-relation} gives the exact identity
\[
 \widetilde V_i^TQ_{i,\perp}=-(E_i^V)^TQ_{i,\perp}.
\]
Taking norms bounds the matrix of inner products between the current stored
block and the orthogonal complement.  The bound is used in both multiplication orders
below.  Combining
\eqref{eq:raw-error-identity} and \eqref{eq:factorization-residual} gives
\[
 \widetilde F_i=Q_iC_i+E_i^{\rm fac}-E_i^R.
\]
Use \eqref{eq:ci-bound} for the first term and
\eqref{eq:factorization-residual} for the second.  The two inequalities yield
\eqref{eq:stored-block-recurrence} and \eqref{eq:local-F-bound}.

Using \eqref{eq:qr-relation} and \eqref{eq:mapping-relation}, we obtain
\begin{align*}
 B_{i+1}\widetilde V_i^T\widetilde V_{i+1}
 &=-B_{i+1}(E_i^V)^TQ_{i,\perp}U_{i,J}
   +B_{i+1}\widetilde V_i^TE_i^M,\\
 \widetilde V_i^T\widetilde V_{i+1}B_{i+1}
 &=-(E_i^V)^TQ_{i,\perp}U_{i,J}B_{i+1}
   +\widetilde V_i^TE_i^MB_{i+1}.
\end{align*}
The bounds in Assumption~\ref{ass:block-orthogonalization} and
submultiplicativity give \eqref{eq:two-orders}, with $q_i$ defined by
\eqref{eq:qi-definition}.
\end{proof}

The recurrence and adjacent-product bounds depend directly on
$\|B_{i+1}\|_2$.  The order of the projection operations has a separate
effect.  In particular, symmetrizing the diagonal block after subtracting
the previous block introduces a term involving the other adjacent product.
We make this dependence precise by retaining the local operation errors.

\begin{proposition}[Symmetrization without re-projection couples the two adjacent products]
\label{prop:symmetrisation}
Consider a block Lanczos step in which
$Y_i=\fl(A\widetilde V_i-\widetilde V_{i-1}B_i^T)$,
$\widehat A_i=\fl(\widetilde V_i^TY_i)$, $A_i=\fl\bigl((\widehat A_i+\widehat A_i^T)/2\bigr)$,
$R_i=\fl(Y_i-\widetilde V_iA_i)$, and a computed factorization forms
$\widetilde V_{i+1}B_{i+1}$ directly from $R_i$.
Write $X_i=\widetilde V_{i-1}^T\widetilde V_i$ and suppose
$\|\widetilde V_i^T\widetilde V_i-I\|_2\le\delta_i$.
Define the local errors by
\begin{align*}
Y_i&=A\widetilde V_i-\widetilde V_{i-1}B_i^T+E_i^Y,\\
\widehat A_i&=\widetilde V_i^TY_i+E_i^{\rm proj},\\
A_i&=(\widehat A_i+\widehat A_i^T)/2+E_i^{\rm sym},\\
R_i&=Y_i-\widetilde V_iA_i+E_i^{\rm upd}
    =\widetilde V_{i+1}B_{i+1}+E_i^{\rm qr},
\end{align*}
and set
\[
\varepsilon_i^{\rm loc}
=\|\widetilde V_i\|_2
 (\|E_i^Y\|_2+\|E_i^{\rm upd}\|_2+\|E_i^{\rm qr}\|_2)
 +\|E_i^{\rm proj}\|_2+\|E_i^{\rm sym}\|_2.
\]
Here $E_i^Y,E_i^{\rm upd},E_i^{\rm qr}\in\R^{n\times r_i}$ and
$E_i^{\rm proj},E_i^{\rm sym}\in\R^{r_i\times r_i}$.
Then
\begin{equation}
 \|\widetilde V_i^T\widetilde V_{i+1}B_{i+1}\|_2
 \le\tfrac12\|B_iX_i\|_2+\tfrac12\|X_i^TB_i^T\|_2
       +\delta_i\|A_i\|_2+\varepsilon_i^{\rm loc}
 =\|B_iX_i\|_2+\delta_i\|A_i\|_2+\varepsilon_i^{\rm loc}.
 \label{eq:symmetrisation-recursion}
\end{equation}
For $i\ge2$, suppose additionally that $B_i$ is a nonempty matrix with full row rank.
Let $B_i^\dagger$ be its Moore--Penrose inverse and define
$\kappa_2(B_i)=\|B_i\|_2\|B_i^\dagger\|_2$.  With
$L_i=\|\widetilde V_i^T\widetilde V_{i+1}B_{i+1}\|_2$, one obtains
\[
 L_i\le\kappa_2(B_i)L_{i-1}
           +\delta_i\|A_i\|_2+\varepsilon_i^{\rm loc}.
\]
For an unsymmetrized diagonal block $A_i=\widehat A_i$, the direct bound is
\[
 L_i\le\delta_i\|\widehat A_i\|_2+\|E_i^{\rm proj}\|_2
       +\|\widetilde V_i\|_2
          (\|E_i^{\rm upd}\|_2+\|E_i^{\rm qr}\|_2).
\]
Alternatively, let $P_i^{\rm cur}$ be the orthogonal projector onto
$\operatorname{range}(\widetilde V_i)$ and suppose the next block is formed
from $(I-P_i^{\rm cur})R_i+E_i^{\rm rep}$ with factorization error
$E_i^{\rm qr}$, where $E_i^{\rm rep}\in\R^{n\times r_i}$.  Then
\[
 L_i\le\|\widetilde V_i\|_2
          (\|E_i^{\rm rep}\|_2+\|E_i^{\rm qr}\|_2).
\]
\end{proposition}

\begin{proof}
Set $D_i=\widetilde V_i^T\widetilde V_i-I$ and write
$\operatorname{sym}(X)=(X+X^T)/2$ and
$\operatorname{skew}(X)=(X-X^T)/2$ for square matrices $X$.
Substitution of the four local error identities gives
\begin{align*}
\widetilde V_i^T\widetilde V_{i+1}B_{i+1}
={}&\tfrac12(B_iX_i-X_i^TB_i^T)-D_iA_i
 +\operatorname{skew}(\widetilde V_i^TE_i^Y)\\
&-\operatorname{sym}(E_i^{\rm proj})-E_i^{\rm sym}
 +\widetilde V_i^T(E_i^{\rm upd}-E_i^{\rm qr}).
\end{align*}
Both symmetrization and skew-symmetrization have spectral norm at most
that of their argument.  The triangle inequality proves
\eqref{eq:symmetrisation-recursion}.  When $B_i$ has full row rank,
$B_iB_i^\dagger=I$, so
\[
B_iX_i=B_i(X_iB_i)B_i^\dagger,
\qquad
\|B_iX_i\|_2\le\kappa_2(B_i)\|X_iB_i\|_2.
\]
This proves the recurrence for $L_i$.  Taking $A_i=\widehat A_i$ in the
local error identities gives the unsymmetrized bound directly.
For the final assertion, multiply the projected residual factorization by
$\widetilde V_i^T$ and use
$\widetilde V_i^T(I-P_i^{\rm cur})=0$.
\end{proof}

\begin{figure}[t]
\centering
\includegraphics[width=0.94\textwidth]{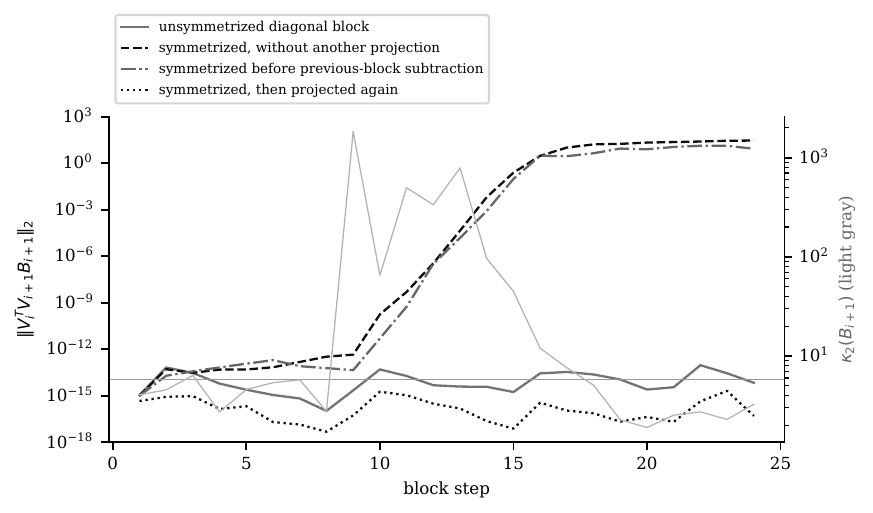}
\caption{Right-weighted adjacent products for four block-update orders
in binary64 on a dense Strako\v{s} matrix ($n=200$, $\lambda_1=0.1$,
$\lambda_n=100$, spacing parameter $\rho=0.7$), with block width $3$.
The unsymmetrized and re-projected variants retain small adjacent products.
The two symmetrized variants accumulate larger products.  In variant (d),
a thin QR of the current block determines the projector, and a QR of the
projected residual gives the next block.  The gray curve shows
$\kappa_2(B_{i+1})$ on the right axis, and the horizontal line marks
$u\|A\|_2$.}
\label{fig:localorth}
\end{figure}

Figure~\ref{fig:localorth} illustrates the dependence on update order.
The adjacent products grow substantially in the two symmetrized variants,
whereas the other two variants keep them near the local rounding-error
scale.  The recurrence in Proposition~\ref{prop:symmetrisation} explains
how an ill-conditioned coupling block can amplify an existing adjacent
product.  Re-projection removes that term.  In the scalar case the skew
part vanishes, which explains why the distinction first appears for blocks.
These products enter the diagonal blocks of $M$ below and hence the
constant in the block Paige identities.

\subsection{Numerical rank and components of the residual}

The residual used for rank selection contains both new directions and
components in earlier blocks.  Removing the current block separates these
two orthogonal parts.  Their Gram matrices describe the resulting singular
values.
The dimension of the exact block Krylov space is characterized by its block
grade \cite{GutknechtSchmelzer2009}, while attainable block Arnoldi and
GMRES behavior is studied in \cite{KubinovaSoodhalter2020}.  Here the
orthogonal decomposition identifies which residual directions enlarge the
space at the current step.

Let
\[
 \mathcal S_i=\operatorname{range}
 [\,\widetilde V_1\ \cdots\ \widetilde V_i\,],
\]
and let $P_i$ and $P_i^{\rm cur}$ be the orthogonal projectors onto
${\cal S}_i$ and $\operatorname{range}(\widetilde V_i)$, respectively.  To
separate the geometry from the errors in applying the Householder reflectors
and computing the SVD, define the residual after exact removal of the current
block by
\[
 \overline R_i=(I-P_i^{\rm cur})R_i.
\]
Its new and previously computed components are
\begin{equation}
 N_i=(I-P_i)R_i,
 \qquad
 C_i^{\rm old}=(P_i-P_i^{\rm cur})R_i.
 \label{eq:new-old-components}
\end{equation}

\begin{proposition}[Rank selected from the residual block]
\label{prop:rank-selection}
The matrices in \eqref{eq:new-old-components} satisfy
\begin{equation}
 \overline R_i=N_i+C_i^{\rm old},
 \qquad N_i^TC_i^{\rm old}=0.
 \label{eq:new-old-orthogonal}
\end{equation}
Consequently,
\begin{equation}
 \overline R_i^T\overline R_i
 =N_i^TN_i+(C_i^{\rm old})^TC_i^{\rm old}.
 \label{eq:new-old-gram}
\end{equation}
Suppose that the matrix whose singular values are used for the rank decision
has the backward representation
\[
 \widehat R_i=\overline R_i+E_i,
 \qquad \|E_i\|_2\le\epsilon_i,
\]
where $E_i$ includes the errors in applying the reflectors and computing the
compact SVD.  Put
\[
 d_i=\|C_i^{\rm old}\|_2+\epsilon_i,
 \qquad
 \widehat r_i=
 \#\{j:\sigma_j(\widehat R_i)>\tau_i\}.
\]
Then
\begin{equation}
 \sigma_j(N_i)-\epsilon_i
 \le\sigma_j(\widehat R_i)
 \le\sqrt{\sigma_j(N_i)^2+\|C_i^{\rm old}\|_2^2}+\epsilon_i.
 \label{eq:rank-selection-singular-bounds}
\end{equation}
In particular,
\begin{equation}
 \#\{j:\sigma_j(N_i)>\tau_i+\epsilon_i\}
 \le \widehat r_i
 \le
 \#\{j:\sigma_j(N_i)>\tau_i-d_i\}.
 \label{eq:rank-selection-bracket}
\end{equation}
The indices in these sets range from $1$ to $\min(n,r_i)$.  If
$\tau_i-d_i<0$, the upper bound is understood as the number of singular
values in that range.  In particular, if
$r=\operatorname{rank}(N_i)<r_i$, retaining more than $r$ directions
requires
\begin{equation}
 \|C_i^{\rm old}\|_2+\epsilon_i>\tau_i.
 \label{eq:old-component-crosses-threshold}
\end{equation}
\end{proposition}

\begin{proof}
The orthogonal decomposition
\[
 I-P_i^{\rm cur}=(I-P_i)+(P_i-P_i^{\rm cur}),
\]
shows that the two terms in \eqref{eq:new-old-orthogonal} have orthogonal ranges.
Equation~\eqref{eq:new-old-gram} follows by expansion.  Its second term is
positive semidefinite and bounded above by $\|C_i^{\rm old}\|_2^2I$.
The eigenvalue min--max principle and Weyl's singular-value inequality
therefore give \eqref{eq:rank-selection-singular-bounds}.  Since
$\sqrt{a^2+b^2}\le a+b$ for $a,b\ge0$, these bounds imply
\eqref{eq:rank-selection-bracket}.  If
$r=\operatorname{rank}(N_i)<r_i$, then
$\sigma_{r+1}(N_i)=0$ and
$\sigma_{r+1}(\widehat R_i)\le d_i$.  Hence a retained $(r+1)$st direction
implies \eqref{eq:old-component-crosses-threshold}.
\end{proof}

For the Householder QR and SVD above, the perturbation in this proposition
can be bounded directly.  Suppose $\delta_i<1$ and
$\ell_i^{\rm qr}>0$.  Set
\begin{equation}
 \epsilon_i=
 \left[a_i^V\sqrt{\frac{1+\delta_i}{1-\delta_i}}
       +a_i^H+a_i^{\rm sv}(1+a_i^H)\right]\|R_i\|_2.
 \label{eq:rank-decision-error}
\end{equation}
Indeed, take $\widehat R_i=Q_{i,\perp}(G_i+E_i^{\rm sv})$.
Its nonzero singular values are those used for the rank decision, with zeros
added if needed.  The QR backward error gives
\[
 \|P_i^{\rm cur}-Q_iQ_i^T\|_2
 \le\frac{\|E_i^V\|_2}{\sigma_{\min}(\widetilde V_i)}
 \le a_i^V\sqrt{\frac{1+\delta_i}{1-\delta_i}}.
\]
Here both projectors have rank $r_i$, and
$(I-Q_iQ_i^T)\widetilde V_i=-(I-Q_iQ_i^T)E_i^V$.
Subtracting $\overline R_i$ from $\widehat R_i$ leaves
\[
 (P_i^{\rm cur}-Q_iQ_i^T)R_i
 +(I-Q_iQ_i^T)E_i^H+Q_{i,\perp}E_i^{\rm sv}.
\]
The reflector bound and $\|G_i\|_2\le(1+a_i^H)\|R_i\|_2$ prove
\eqref{eq:rank-decision-error}.

The singular values of $N_i$ describe the residual directions orthogonal
to the computed space. Adding the positive semidefinite matrix
$(C_i^{\rm old})^TC_i^{\rm old}$ to $N_i^TN_i$ can increase the selected
rank. The component in the existing space can therefore prevent a reduction
in block size even when $N_i$ has lost numerical rank.
Every singular value of $N_i$ larger than $\tau_i+\epsilon_i$ is retained,
with the local error determining the uncertainty at the threshold.

\begin{lemma}[Growth bound for the coupling matrices]
\label{lem:coefficient-envelope}
Fix $k\ge1$.  For $1\le i\le k+1$, set
\[
 b_i=\|B_i\|_2,
 \qquad b_1=0.
\]
For $1\le i\le k$, set
\[
 \delta_i=\|\widetilde V_i^T\widetilde V_i-I\|_2,
 \qquad s_i=\sqrt{1+\delta_i},
\]
and set $s_0=0$.  Let $\sigma_i=(1+a_i^S)(1+a_i^H)$.  Suppose that, for
$1\le i\le k$,
\begin{align*}
 \|A_i\|_2&\le(1+\delta_i)\|A\|_2+\eps_i^\alpha,\\
 \|R_i\|_2&\le s_i\|A\|_2+s_{i-1}b_i
                   +s_i\|A_i\|_2+\eps_i^R,\\
 b_{i+1}&\le\sigma_i\|R_i\|_2.
\end{align*}
Assume also that $a_i^S$, $a_i^H$, $a_i^\alpha$, $r_i^A$, and $r_i^B$ are
nonnegative and satisfy
\[
 \eps_i^\alpha\le a_i^\alpha\|A\|_2,
 \qquad
 \eps_i^R\le r_i^A\|A\|_2+r_i^B b_i.
\]
Define
\begin{equation}
 \alpha_i^{\rm env}=\sigma_i
 \left[s_i(2+\delta_i)+s_i a_i^\alpha+r_i^A\right],
 \qquad
 \beta_i^{\rm env}=\sigma_i(s_{i-1}+r_i^B).
 \label{eq:alpha-beta-envelope}
\end{equation}
Then, for $1\le i\le k$,
\begin{equation}
 b_{i+1}\le\alpha_i^{\rm env}\|A\|_2+\beta_i^{\rm env}b_i,
 \qquad
 b_{i+1}\le\|A\|_2\sum_{j=1}^{i}
   \alpha_j^{\rm env}\prod_{\ell=j+1}^{i}\beta_\ell^{\rm env},
 \label{eq:affine-envelope}
\end{equation}
where an empty product equals one.
\end{lemma}

\begin{proof}
Substituting the residual estimate and
$\|A_i\|_2\le(1+\delta_i)\|A\|_2+\eps_i^\alpha$ into
$\|B_{i+1}\|_2\le\sigma_i\|R_i\|_2$ gives
$b_{i+1}\le\alpha_i^{\rm env}\|A\|_2+\beta_i^{\rm env}b_i$, the first
inequality in \eqref{eq:affine-envelope}.  Iterating this one-step inequality
from $b_1=0$ gives the second inequality in \eqref{eq:affine-envelope}, which
is the stated product formula.
\end{proof}

\begin{remark}[Accumulation over several block steps]
Define
\[
 \Pi_k=\max_{1\le j\le i\le k}
 \prod_{\ell=j}^{i}\max\{1,\beta_\ell^{\rm env}\}.
\]
Lemma~\ref{lem:coefficient-envelope} gives polynomial growth in $k$ when
$\Pi_k$ and $\max_{i\le k}\alpha_i^{\rm env}$ have polynomial bounds.
Bounded values of both quantities give linear growth.  For a constant
$C_0\ge0$, a sufficient condition for $\Pi_k\le e^{C_0}$ is
\[
 \sum_{\ell=1}^k\log\max\{1,\beta_\ell^{\rm env}\}\le C_0.
\]
For constants $c,\varepsilon_*\ge0$, if $\beta_i^{\rm env}\le1+c\varepsilon_*$ and
$ck\varepsilon_*\le C_0$, then $\Pi_k\le e^{C_0}$.  If a concrete rounding-error
analysis gives $\beta_i^{\rm env}\le1+cu$ uniformly and $ku$ is bounded, then the
cumulative product is bounded.  Obtaining a local perturbation of order $u$
also requires a relative threshold of order $u$ and bounded coefficient
norms.  The cumulative product records how these local estimates combine.
\end{remark}

\subsection{An orthonormal basis for an enlarged symmetric problem}
\label{sec:orthonormalized}

Concatenating the local recurrences gives a matrix relation for all computed
blocks.  Fix $k\ge1$ with positive retained ranks $r_1,\ldots,r_k$, and put
\[
 m=\sum_{i=1}^k r_i,\qquad
 \widetilde V=[\,\widetilde V_1\ \cdots\ \widetilde V_k\,]
    \in\R^{n\times m},\qquad
 \widetilde F=[\,\widetilde F_1\ \cdots\ \widetilde F_k\,]
    \in\R^{n\times m}.
\]
For $1\le j\le k$, let $E_j\in\R^{m\times r_j}$ be the canonical injection
into block $j$.  Let $T=T^T\in\R^{m\times m}$ be the original variable-block
tridiagonal matrix.  Write its final term as
$\widetilde W B_{k+1}E_k^T$, where
\[
 \widetilde W\in\R^{n\times s},\qquad
 B_{k+1}\in\R^{s\times r_k},\qquad 0\le s\le r_k.
\]
The value $s=0$ represents terminal rank loss.  The computed block recurrence,
written in matrix form, is
\[
 A\widetilde V=\widetilde VT+\widetilde W B_{k+1}E_k^T+\widetilde F.
\]
For $s>0$, Lemma~\ref{lem:thresholded-factorization} gives
\begin{equation}
 \|\widetilde W^T\widetilde W-I_s\|_2<\frac12.
 \label{eq:boundary-gram}
\end{equation}
Thus $\widetilde W$ has full column rank.  For $s=0$, use the empty-matrix
convention.
For $1\le i\le k$, define
\[
 D_i=(\widetilde V_i^T\widetilde V_i)^{1/2},\qquad
 V_i=\widetilde V_iD_i^{-1},\qquad
 D=\diag(D_1,\ldots,D_k).
\]
Orthonormalize the boundary block by defining
\[
 D_{k+1}=(\widetilde W^T\widetilde W)^{1/2},\qquad
 W=\widetilde W D_{k+1}^{-1}\in\R^{n\times s},\qquad
 H=D_{k+1}B_{k+1}D_k^{-1}\in\R^{s\times r_k}.
\]
For $s=0$, these are the corresponding empty matrices.  Set
$V=[\,V_1\ \cdots\ V_k\,]\in\R^{n\times m}$, leave $T$ unchanged, and define
\begin{equation}
 F=V(DTD^{-1}-T)+\widetilde FD^{-1}\in\R^{n\times m}.
 \label{eq:polar-transport}
\end{equation}
Substitution gives
\begin{equation}
 AV=VT+WHE_k^T+F,
 \qquad V_i^TV_i=I_{r_i},\quad W^TW=I_s.
 \label{eq:collected-block}
\end{equation}
If
\[
 \delta=\max_{1\le i\le k+1}
 \|\widetilde V_i^T\widetilde V_i-I\|_2<1/2,
 \qquad \widetilde V_{k+1}:=\widetilde W,
\]
where the last term is zero for $s=0$, and
\[
 \chi=\frac{\delta}
 {\sqrt{1-\delta}(1+\sqrt{1-\delta})},
\]
then
\[
 1+\chi=\frac{1}{\sqrt{1-\delta}}.
\]
Spectral calculus gives
\[
 \|D_i^{-1}-I\|_2\le\chi,
 \qquad \|D_i^{-1}\|_2\le1+\chi
 \quad(1\le i\le k+1),
\]
and
\begin{equation}
 \|F\|_2\le 2\sqrt{k}\,\chi\|T\|_2
       +\frac{\|\widetilde F\|_2}{\sqrt{1-\delta}}.
 \label{eq:normalized-F}
\end{equation}
Indeed,
\[
 DTD^{-1}-T=(D-I)TD^{-1}+T(D^{-1}-I),
\]
where $\|D-I\|_2\le\chi\sqrt{1-\delta}$ and
$\|D^{-1}\|_2\le1/\sqrt{1-\delta}$.  Hence this difference has norm at most
$2\chi\|T\|_2$.  Since $\|V\|_2\le\sqrt{k}$, the first term in
\eqref{eq:polar-transport} gives the first term in
\eqref{eq:normalized-F}.  In addition,
$\|\widetilde FD^{-1}\|_2\le\|\widetilde F\|_2/\sqrt{1-\delta}$.
Set $B_1=0$ and $q_0=0$, and define
\begin{equation}
 \overline q_i=(1+\chi)q_{i-1}
 +(1+\chi)(1+\delta)\chi\|B_i\|_2.
 \label{eq:transported-adjacent-envelope}
\end{equation}
Expansion of $D_{i-1}^{-1}$ and $D_i^{-1}$ around the identity gives
\begin{equation}
 \max\left\{
 \|B_iV_{i-1}^TV_i\|_2,
 \|V_{i-1}^TV_iB_i\|_2
 \right\}\le\overline q_i.
 \label{eq:transported-adjacent-bound}
\end{equation}

Each normalized block has orthonormal columns.  The cross-products between
different blocks are collected in the strictly block upper triangular
matrix $U$, with $U_{ij}=V_i^TV_j$ for $i<j$.  Thus
\begin{equation}
 V^TV=I+U+U^T.
 \label{eq:U-def}
\end{equation}
Put $a=V^TW$ and
\begin{equation}
 M=TU-UT-aHE_k^T.
 \label{eq:M-def}
\end{equation}
For interior diagonal blocks,
\begin{equation}
 M_{ii}=B_iV_{i-1}^TV_i-V_i^TV_{i+1}B_{i+1},
 \qquad i<k.
 \label{eq:M-diag}
\end{equation}
For the terminal block, with $B_1V_0^TV_1$ interpreted as zero,
\begin{equation}
 M_{kk}=B_kV_{k-1}^TV_k-V_k^TWH.
 \label{eq:M-terminal-diag}
\end{equation}
If $q_k$ bounds the final adjacent product from
\eqref{eq:two-orders}, then the factors involving $D_{k+1}$ cancel.  The
definition of $H$ places $D_{k+1}$ on the left, giving
\begin{equation}
 V_k^TWH
 =D_k^{-1}(\widetilde V_k^T\widetilde W B_{k+1})D_k^{-1},
 \qquad
 \|V_k^TWH\|_2\le(1+\chi)^2q_k.
 \label{eq:terminal-bound}
\end{equation}
Define
\begin{equation}
 \zeta=\max_{1\le i\le k}\|M_{ii}\|_2.
 \label{eq:zeta-definition}
\end{equation}
The final diagonal block is therefore estimated with
\eqref{eq:M-terminal-diag}.  Together with
\eqref{eq:transported-adjacent-envelope}, the preceding identities give a
bound for the diagonal blocks of $M$,
\begin{equation}
 \zeta\le
 \max\left\{
  \max_{1\le i<k}(\overline q_i+\overline q_{i+1}),
  \overline q_k+(1+\chi)^2q_k
 \right\},
 \label{eq:zeta-bound}
\end{equation}
where the interior maximum is interpreted as zero when $k=1$.

Greenbaum introduced an enlarged exact problem for a perturbed Lanczos
recurrence \cite{Greenbaum1989}.  Paige subsequently gave an augmented
Hermitian model for finite-precision Lanczos and used it to study eigenvalue
and linear-system computations \cite{Paige2010Augmented,Paige2019,Paige2024}.  The
definitions of $\widehat V$ and $\widehat W$ below are adapted from Chen's
closed-form scalar construction and extend it to blocks whose sizes and
terminal rank may change \cite[Sec.~4]{Chen2026Simple}.

We use the usual empty-matrix convention when $s=0$.  The notation
$\col(X_1,\ldots,X_j)$ means vertical concatenation, and $I_k\otimes A$ is
the block diagonal matrix with $k$ copies of $A$.

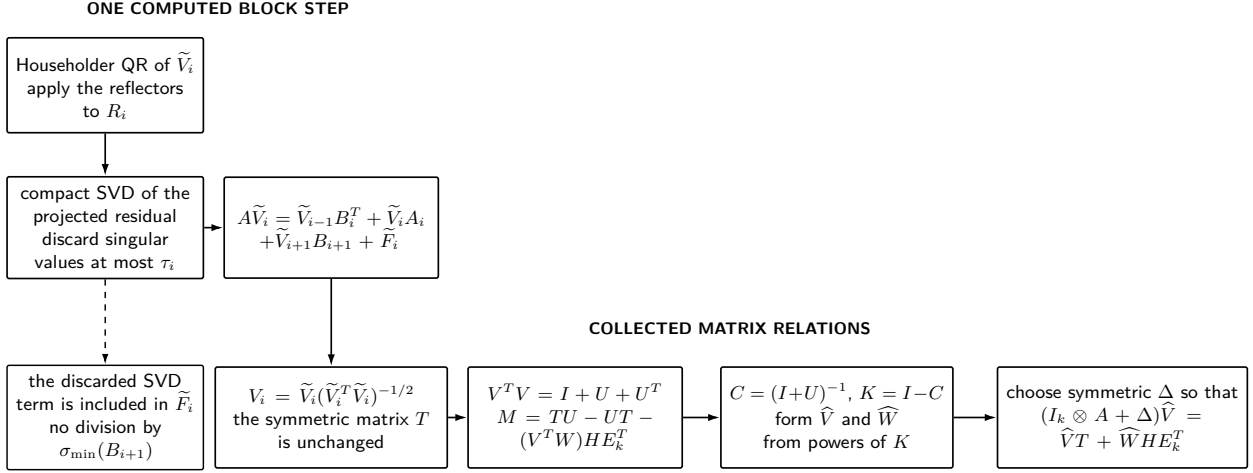
\begin{figure}[t]
  \centering
  \resizebox{\textwidth}{!}{\begin{tikzpicture}[
  x=1mm,y=1mm,
  font=\sffamily\footnotesize,
  >=Latex,
  proc/.style={draw=black, line width=0.8pt, rounded corners=1pt,
               fill=white, text width=30mm, minimum height=16mm,
               align=center, inner sep=1.5mm},
  relation/.style={draw=black, line width=0.8pt, rounded corners=1pt,
                   fill=white, text width=33mm, minimum height=17mm,
                   align=center, inner sep=1.5mm},
  result/.style={draw=black, line width=0.8pt, rounded corners=1pt,
                 fill=white, text width=36mm, minimum height=17mm,
                 align=center, inner sep=1.5mm},
  final/.style={draw=black, line width=0.8pt, rounded corners=1pt,
                fill=white, text width=39mm, minimum height=17mm,
                align=center, inner sep=1.5mm},
  note/.style={draw=black, line width=0.7pt, rounded corners=1pt,
               fill=white, text width=30mm, minimum height=17mm,
               align=center, inner sep=1.5mm},
  heading/.style={font=\sffamily\bfseries\scriptsize, text=black},
  arr/.style={-{Latex[length=2mm,width=1.25mm]}, line width=0.8pt, draw=black},
  note-arr/.style={-{Latex[length=1.8mm,width=1.1mm]}, dashed,
                   line width=0.7pt, draw=black}
]

\node[heading] at (19,39) {ONE COMPUTED BLOCK STEP};
\node[proc] (qr) at (0,26)
  {Householder QR of $\widetilde V_i$\\apply the reflectors\\to $R_i$};
\node[proc] (svd) at (0,2)
  {compact SVD of the projected residual\\discard singular values at most $\tau_i$};
\node[relation] (local) at (38,2)
  {$A\widetilde V_i=\widetilde V_{i-1}B_i^T+\widetilde V_iA_i$\\
   $+\widetilde V_{i+1}B_{i+1}+\widetilde F_i$};

\node[heading] at (105,-15) {COLLECTED MATRIX RELATIONS};
\node[note] (tail) at (0,-30)
  {the discarded SVD term is included in $\widetilde F_i$\\
   no division by $\sigma_{\min}(B_{i+1})$};
\node[result] (normalize) at (38,-30)
  {$V_i=\widetilde V_i(\widetilde V_i^T\widetilde V_i)^{-1/2}$\\
   the symmetric matrix $T$\\is unchanged};
\node[relation] (gram) at (79,-30)
  {$V^TV=I+U+U^T$\\$M=TU-UT-(V^TW)HE_k^T$};
\node[result] (stack) at (123,-30)
  {$C=(I+U)^{-1}$, $K=I-C$\\
   form $\widehat V$ and $\widehat W$ from powers of $K$};
\node[final] (complete) at (171,-30)
  {choose symmetric $\Delta$ so that\\
   $(I_k\otimes A+\Delta)\widehat V=\widehat VT+\widehat WHE_k^T$};

\draw[arr] (qr.south) -- (svd.north);
\draw[arr] (svd.east) -- (local.west);
\draw[arr] (local.south) -- (normalize.north);
\draw[arr] (normalize.east) -- (gram.west);
\draw[arr] (gram.east) -- (stack.west);
\draw[arr] (stack.east) -- (complete.west);
\draw[note-arr] (svd.south) -- (tail.north);

\end{tikzpicture}
}
  \caption{From the local QR and SVD error bounds to the enlarged symmetric
  problem.  The discarded singular directions contribute to
  $\widetilde F_i$.  Normalizing the blocks and using their Gram matrix
  then gives the exact relation in Theorem~\ref{thm:block-global}.}
  \label{fig:block-architecture}
\end{figure}

\begin{theorem}[An exact block Lanczos relation for $I_k\otimes A+\Delta$]
\label{thm:block-global}
Let $A=A^T\in\R^{n\times n}$ and let $k\ge1$.  Choose positive block sizes
$r_1,\ldots,r_k$ and put $m=\sum_{i=1}^k r_i$.  For $1\le i\le k$, let
$E_i\in\R^{m\times r_i}$ be the canonical injection into block $i$.
Let $T=T^T\in\R^{m\times m}$ be block tridiagonal with respect to this
partition.  Let $0\le s\le r_k$ and suppose
$V=[V_1\ \cdots\ V_k]\in\R^{n\times m}$,
$W\in\R^{n\times s}$, $H\in\R^{s\times r_k}$, and
$F\in\R^{n\times m}$ satisfy $V_i^TV_i=I_{r_i}$, $W^TW=I_s$, and
\begin{equation}
 AV=VT+WHE_k^T+F.
 \label{eq:block-global-hypothesis}
\end{equation}
Use $U$, $a$, and $M$ from \eqref{eq:U-def}--\eqref{eq:M-def}.  Set
$e=\|F\|_2$ and
\[
 \zeta=\max_{1\le i\le k}\|M_{ii}\|_2,
 \qquad \mu=k\max\{2e,\zeta\}.
\]

Then there are matrices
\[
 \widehat V\in\R^{nk\times m},
 \qquad \widehat W\in\R^{nk\times s},
 \qquad \Delta=\Delta^T\in\R^{nk\times nk}
\]
such that
\begin{equation}
 \widehat V^T\widehat V=I_m,
 \qquad \widehat V^T\widehat W=0,
 \qquad \|\widehat W\|_2\le1,
 \label{eq:dilated-gram}
\end{equation}
and
\begin{equation}
 (I_k\otimes A+\Delta)\widehat V
 =\widehat VT+\widehat WHE_k^T.
 \label{eq:nearby-model}
\end{equation}
The first block satisfies
\begin{equation}
 \widehat V E_1=\col(V_1,0,\ldots,0),
 \label{eq:first-block-preserved}
\end{equation}
and
\begin{equation}
 \|\Delta\|_2
 \le2\left(2\sqrt{k}\,e+4\sqrt{k}(\sqrt{k}+k-1)\mu\right).
 \label{eq:Delta-bound}
\end{equation}
\end{theorem}

\begin{proof}
Premultiplying \eqref{eq:block-global-hypothesis} by $V^T$ and subtracting the
transpose of the resulting identity, then using $A=A^T$, $T=T^T$,
$V^TV=I+U+U^T$, and the definition of $M$, gives
\begin{equation}
 M-M^T=V^TF-F^TV.
 \label{eq:block-commutator}
\end{equation}
The block-tridiagonal structure of $T$ and the strictly block upper-triangular
structure of $U$ imply that $TU-UT$ is block upper triangular.
The boundary term has nonzero blocks only in its last block column, so
$M$ is also block upper triangular. Each
strict upper block is therefore bounded by $2e$.  The block row-sum estimate
gives $\|M\|_2\le\mu$.

Define $C=(I+U)^{-1}$, $K=I-C$, $S=Ca$, and $Y=W-VS$.
Here $C,K\in\R^{m\times m}$, $S\in\R^{m\times s}$, and $Y\in\R^{n\times s}$.
Nilpotence gives $U^k=0$, so $C$ is a finite Neumann polynomial and
$K=UC=CU$.  Direct multiplication yields
\begin{equation}
\begin{aligned}
 (VC)^TVC&=I-K^TK,\\
 (VC)^TY&=-K^TS,\\
 S^TS+Y^TY&=I.
\end{aligned}
\label{eq:gram-triple}
\end{equation}
Positive semidefiniteness of $V^TV$ implies $\|K\|_2\le1$ and
$\|C\|_2\le2$.  Set $L=CMC$.  Substitution into the definitions of $C$, $K$,
$S$, and $M$ gives
\begin{equation}
 TK=KT+SHE_k^T+L,
 \qquad \|L\|_2\le4\mu.
 \label{eq:exchange}
\end{equation}

Form the block-column matrices
\begin{equation}
\begin{split}
 \widehat V&=\col(VC,VCK,\ldots,VCK^{k-1}),\\
 \widehat W&=\col(Y,VCS,VCKS,\ldots,VCK^{k-2}S).
\end{split}
\label{eq:dilation}
\end{equation}
For $k=1$, the second definition reads $\widehat W=Y$.  Since $U$ is
strictly block upper triangular, $UE_1=0$.  Hence $CE_1=E_1$ and $KE_1=0$.
It follows that $\widehat V E_1=\col(V_1,0,\ldots,0)$, which proves
\eqref{eq:first-block-preserved}.
Thus the enlarged relation begins with the computed starting block, so
spectral information derived from $T$ refers to that same initial block.
The identities in \eqref{eq:gram-triple} telescope and prove
\eqref{eq:dilated-gram}.  Repeated use of \eqref{eq:exchange} in
\eqref{eq:block-global-hypothesis} gives
\begin{equation}
 (I_k\otimes A)\widehat V
 =\widehat VT+\widehat WHE_k^T+\widehat F.
 \label{eq:dilated-recurrence}
\end{equation}
To bound $\widehat F$, write $X=VC$ and define
\[
 P_0=FC-VL,
 \qquad
 P_{j+1}=XK^jL+P_jK\quad(0\le j<k-1).
\]
The last block row of $K$ is zero, so $E_k^TK=0$.  Similarly,
$E_k^TC=E_k^T$.  Equation~\eqref{eq:exchange} now shows that
$\widehat F=\col(P_0,\ldots,P_{k-1})$.  The bounds
$\|K\|_2\le1$, $\|C\|_2\le2$, $\|V\|_2\le\sqrt{k}$, and
$\|L\|_2\le4\mu$, together with
$\|X\|_2\le1$ from \eqref{eq:gram-triple}, yield
\[
 \|P_j\|_2\le2e+4(\sqrt{k}+j)\mu.
\]
Consequently,
\begin{equation}
 \|\widehat F\|_2
 \le2\sqrt{k}\,e+4\sqrt{k}(\sqrt{k}+k-1)\mu.
 \label{eq:Fhat-bound}
\end{equation}

Put $R=-\widehat F$.  Premultiplying \eqref{eq:dilated-recurrence} by
$\widehat V^T$, and using \eqref{eq:dilated-gram} together with the symmetry
of $I_k\otimes A$ and $T$, shows that $\widehat V^T\widehat F$, and hence
$\widehat V^TR$, is symmetric.  Define
\begin{equation}
 \Delta=R\widehat V^T+\widehat VR^T
 -\widehat V(\widehat V^TR)\widehat V^T.
 \label{eq:symmetric-completion}
\end{equation}
The matrix $\Delta$ is symmetric and $\Delta\widehat V=R$.  Since
$\widehat V^TR$ is symmetric, it can also be written as
\[
 \Delta=R\widehat V^T+
 \widehat VR^T(I-\widehat V\widehat V^T).
\]
Both factors involving $\widehat V$ have norm at most one, so
$\|\Delta\|_2\le2\|R\|_2$.  Equations
\eqref{eq:nearby-model}--\eqref{eq:Delta-bound} follow.
\end{proof}

\begin{corollary}[Factorization of the final boundary block]
\label{cor:block-boundary-factorization}
Under the hypotheses of Theorem~\ref{thm:block-global}, choose matrices
$\widehat V$, $\widehat W$, and $\Delta$ that satisfy its conclusion.  Let
$t=\operatorname{rank}(\widehat WH)$ and choose a compact factorization
\[
 \widehat WH=\overline W\,\overline H,
 \qquad
 \overline W\in\R^{nk\times t},
 \qquad
 \overline H\in\R^{t\times r_k},
 \qquad
 \overline W^T\overline W=I_t.
\]
Then
\[
 \widehat V^T\overline W=0,
 \qquad
 \operatorname{rank}(\overline H)=t,
\]
and
\[
 (I_k\otimes A+\Delta)\widehat V
 =\widehat VT+\overline W\,\overline H E_k^T.
\]
If $\operatorname{rank}(B_i)=r_i$ for $2\le i\le k$, this relation is an
exact variable-block Lanczos recurrence through block $k$, up to an
orthogonal change of basis within each block.
\end{corollary}

\begin{proof}
Choose the columns of $\overline W$ as an orthonormal basis for
$\operatorname{range}(\widehat WH)$ and set
$\overline H=\overline W^T\widehat WH$.  Then
$\widehat WH=\overline W\overline H$ and
$\operatorname{rank}(\overline H)=t$.  Since
$\operatorname{range}(\overline W)\subseteq
\operatorname{range}(\widehat W)$ and
$\widehat V^T\widehat W=0$, one has
$\widehat V^T\overline W=0$.  Substitution in
\eqref{eq:nearby-model} proves the recurrence.
\end{proof}

\subsection{Exact-arithmetic block Ritz clusters and deflation}

Repeated block Ritz values can occur in exact arithmetic while their
physical Ritz vectors remain independent.  The current projection and the
rank of the next block can both be expressed through a Gram matrix.
This description follows the established connection between block Krylov
spaces, spectral measures, and orthogonal polynomials
\cite{FenuMartinReichelRodriguez2013,RinelliVandebril2026}, in the setting
of the exact block Ritz clusters studied by \v{S}imonov\'a and Tich\'y
\cite{SimonovaTichy2025}.
Exact-arithmetic convergence estimates for Lanczos and block Lanczos are
given in \cite{Saad1980,LiZhang2015}.  The corresponding moment and
quadrature formulation is developed in \cite{GolubMeurant2010}.

Let $Q\in\R^{n\times d}$ satisfy $Q^TQ=I_d$, and let $E_A(\Omega)$ denote
the spectral projector of $A$ associated with a Borel set $\Omega\subset\R$.
Define the scalar probability measure
\[
 \nu_Q(\Omega)=\frac{1}{d}\operatorname{tr}
 \bigl(Q^TE_A(\Omega)Q\bigr).
\]
Suppose the orthonormal polynomials $\phi_0,\ldots,\phi_k$ for $\nu_Q$
exist, with positive leading coefficients, and let
$J_k$ be the $k\times k$ Jacobi matrix in their three-term recurrence.  Put
\[
 X=[\,\phi_0(A)Q\ \cdots\ \phi_{k-1}(A)Q\,],
 \qquad Y=\phi_k(A)Q,
\]
and define
\[
 S=X^TX,
 \qquad C=X^TY,
 \qquad D=Y^TY.
\]

\begin{proposition}[Block Ritz values and the next block]
\label{prop:exact-clusters-deflation}
Assume that $S$ is nonsingular.  Let $e_k$ be the $k$th coordinate vector in
$\R^k$.  If $\beta_k$ is the final recurrence coefficient for
$\phi_0,\ldots,\phi_k$ and
\[
 {\cal B}_k=\beta_k(e_k^T\otimes I_d),
\]
then
\begin{align}
 AX&=X(J_k\otimes I_d)+Y{\cal B}_k,
 \label{eq:polynomial-block-recurrence}\\
 S^{-1}X^TAX
 &=J_k\otimes I_d+S^{-1}C{\cal B}_k.
 \label{eq:block-ritz-boundary-correction}
\end{align}
Let $P_X=XS^{-1}X^T$ and $Y_\perp=(I-P_X)Y$.  Then
\begin{equation}
 Y_\perp^TY_\perp=D-C^TS^{-1}C.
 \label{eq:next-block-schur-complement}
\end{equation}
Consequently, if $C=0$, every eigenvalue of $J_k$ occurs $d$ times in the
block Ritz problem.  The rank of the next new block is
$\operatorname{rank}(D-C^TS^{-1}C)$.
\end{proposition}

\begin{proof}
Applying the scalar three-term recurrence to each column of $Q$ gives
\eqref{eq:polynomial-block-recurrence}.  Premultiplication by $X^T$ and
multiplication by $S^{-1}$ give
\eqref{eq:block-ritz-boundary-correction}.  Since $P_X$ is the orthogonal
projector onto $\operatorname{range}(X)$,
\[
 Y_\perp^TY_\perp
 =Y^T(I-P_X)Y=D-C^TS^{-1}C.
\]
The remaining statements follow from
\eqref{eq:block-ritz-boundary-correction} and from the rank of the Gram
matrix $Y_\perp^TY_\perp$.
\end{proof}

The matrix $C$ gives the correction to the current block Ritz problem.
The Schur complement in \eqref{eq:next-block-schur-complement} measures the
new directions available at the next step.  When $C=0$, the repeated
values have independent lifted Ritz vectors because $X$ has full column
rank.  They are an exact-arithmetic feature of this projection.

Weyl's eigenvalue perturbation theorem applied to the enlarged symmetric
matrix shows that every Ritz value of $T$ lies in
\[
 [\lambda_{\min}(A)-\|\Delta\|_2,
   \lambda_{\max}(A)+\|\Delta\|_2].
\]
For $Ty=\theta y$ and $\|y\|_2=1$,
\begin{equation}
 \dist(\theta,\sigma(A))
 \le\|HE_k^Ty\|_2+\|\Delta\|_2.
 \label{eq:ritz-posteriori}
\end{equation}
Block Lanczos may have improper clusters away from every eigenvalue of $A$
\cite{SimonovaTichy2025}.  Localization of stabilized Ritz values therefore
requires the computable boundary term in \eqref{eq:ritz-posteriori}.

\subsection{Block Paige identities}
\label{sec:block-paige}

Paige's scalar identity relates loss of orthogonality to a Ritz residual
\cite{Paige1976}.  Its block analogue follows from
\eqref{eq:collected-block} and the symmetry of $A$ and $T$.
Use the normalized blocks of Section~\ref{sec:orthonormalized}, so that
$V_i^TV_i=I_{r_i}$, $W^TW=I_s$, and $m=\sum_{i=1}^k r_i$.
Here $s$ is the width of the final block.  With $M$ defined in
\eqref{eq:M-def}, the preceding proof gives
$\|M\|_2\le\mu=k\max\{2e,\zeta\}$, where $e=\|F\|_2$ and
$\zeta$ is defined in \eqref{eq:zeta-definition}.

\begin{theorem}[Block Paige identities]
\label{thm:block-paige}
Assume \eqref{eq:collected-block}.  Let $(\theta_i,s_i)$, $i=1,\ldots,m$, be
the eigenpairs of $T$ with orthonormal eigenvectors, and define the lifted
Ritz vectors, the block residual coordinates and the lost orthogonality by
\[
 y_i=Vs_i\in\R^n,\qquad
 w_i=HE_k^Ts_i\in\R^s,\qquad
 g_i=W^Ty_i\in\R^s .
\]
The full physical residual is $Ay_i-\theta_i y_i=Ww_i+Fs_i$.
For every $i$,
\begin{equation}
 g_i^Tw_i=-s_i^TMs_i,
 \qquad\text{so that}\qquad
 |g_i^Tw_i|\le\|M\|_2\le\mu,
 \label{eq:block-paige-diag}
\end{equation}
and, for every $i\ne j$,
\begin{equation}
 (\theta_j-\theta_i)\,y_j^Ty_i
 =g_j^Tw_i-g_i^Tw_j+s_j^T\bigl(V^TF-F^TV\bigr)s_i .
 \label{eq:block-paige-offdiag}
\end{equation}
Write $\rho_i=\|w_i\|_2$ for the boundary-residual norm.  Then
\begin{equation}
 \begin{aligned}
 \rho_i\,\bigl|\hat w_i^Tg_i\bigr|&\le\mu
       &&(\hat w_i=w_i/\rho_i,\ \rho_i>0),\\
 |y_j^Ty_i|&\le
 \frac{\|g_j\|_2\rho_i+\|g_i\|_2\rho_j
             +e(\|y_i\|_2+\|y_j\|_2)}{|\theta_j-\theta_i|}
       &&(\theta_i\ne\theta_j).
 \end{aligned}
 \label{eq:block-paige-consequences}
\end{equation}
\end{theorem}

\begin{proof}
By the definition \eqref{eq:M-def}, $TU-UT=M+aHE_k^T$ with $a=V^TW$.  For
any two eigenvectors, $s_j^T(TU-UT)s_i=(\theta_j-\theta_i)\,s_j^TUs_i$, and
$s_j^TaHE_k^Ts_i=(Vs_j)^TW(HE_k^Ts_i)=g_j^Tw_i$.  Hence
\begin{equation}
 (\theta_j-\theta_i)\,s_j^TUs_i=s_j^TMs_i+g_j^Tw_i .
 \label{eq:block-paige-core}
\end{equation}
For $j=i$ the left side vanishes, which is \eqref{eq:block-paige-diag}.
For $i\ne j$, subtracting \eqref{eq:block-paige-core} with the roles of $i$ and
$j$ exchanged gives
$(\theta_j-\theta_i)\,s_j^T(U+U^T)s_i=s_j^T(M-M^T)s_i+g_j^Tw_i-g_i^Tw_j$.
By \eqref{eq:U-def}, $s_j^T(U+U^T)s_i=s_j^T(V^TV-I)s_i=y_j^Ty_i$ for
$i\ne j$, and by \eqref{eq:block-commutator} $M-M^T=V^TF-F^TV$, which gives
\eqref{eq:block-paige-offdiag}.  The Cauchy--Schwarz inequality gives
\[
|s_j^T(V^TF-F^TV)s_i|
=|y_j^TFs_i-(Fs_j)^Ty_i|
\le e(\|y_j\|_2+\|y_i\|_2).
\]
Together with the diagonal identity, this proves
\eqref{eq:block-paige-consequences}.  The physical vector norms are bounded
by $\sqrt{k}$ because the individual blocks are orthonormal.
\end{proof}

For scalar Lanczos before breakdown, $s=1$ and both $w_i$ and $g_i$ are
scalars.  The formulas become
$w_i=\beta_{k+1}e_k^Ts_i$ and $g_i=v_{k+1}^Ty_i$, recovering Paige's
one-vector and two-vector identities.  With a wider final block,
$g_i^Tw_i$ controls the component of $g_i$ along its residual coordinate.
The following proposition counts the remaining freedom in these equations.

\begin{proposition}[What the identities determine]
\label{prop:dof}
Suppose $s\ge1$, every $w_i$ is nonzero, and the Ritz values are distinct.
Put $Y=[\,y_1\ \cdots\ y_m\,]$.  Regard
\eqref{eq:block-paige-diag}--\eqref{eq:block-paige-offdiag} as
equations for the unknown matrix $\Gamma=W^TY=[\,g_1\ \cdots\ g_m\,]\in\R^{s\times m}$
and the unknown off-diagonal entries of the Ritz Gram matrix $Y^TY$, the
quantities $s_j^TMs_i$, $s_j^T(V^TF-F^TV)s_i$ and $w_i$ being given.
There are $m+\binom m2$ equations for
$sm+\binom m2$ unknowns.  For $s=1$ they determine $\Gamma$ and the
off-diagonal entries of $Y^TY$
uniquely.  First, $g_i=-s_i^TMs_i/w_i$, and then $y_j^Ty_i$ follows from
\eqref{eq:block-paige-offdiag}.  For $s\ge2$ the component of each $g_i$
along $w_i$ is determined, the off-diagonal entries of $Y^TY$ are determined
by $\Gamma$, and the
$m(s-1)$ components of the $g_i$ orthogonal to the corresponding $w_i$ are
unconstrained by the identities.
\end{proposition}

\begin{proof}
Equation \eqref{eq:block-paige-diag} fixes the one scalar component
$\hat w_i^Tg_i$.  Given $\Gamma$, \eqref{eq:block-paige-offdiag} fixes
each $y_j^Ty_i$ because $\theta_j\ne\theta_i$.  Conversely, adding to
$g_i$ any vector orthogonal to $w_i$ preserves $g_i^Tw_i$.
The resulting change in the right side of \eqref{eq:block-paige-offdiag}
is absorbed by the corresponding change of $y_j^Ty_i$, so the affine solution
set has dimension $m(s-1)$.  The count of unknowns and equations is as
stated.
\end{proof}

The count concerns the displayed algebraic equations.  Actual Lanczos
overlaps also satisfy the Gram-matrix positivity conditions and the
inter-block recurrence.  For each nonzero residual coordinate, the diagonal
identity controls one direction in $\R^s$.  The other $s-1$ components
are governed by their evolution across block steps.  Thus the final retained
width determines the dimension of this complement after deflation.
If $w_i=0$, the diagonal equation reduces to $s_i^TMs_i=0$, and all $s$
components of $g_i$ remain free in that equation.

This distinction explains why a block Ritz vector can have appreciable
overlap with the next block while its own boundary residual remains large.
The dynamics of that overlap are described by
\eqref{eq:block-overlap-recurrence}.  The identity and the recurrence
together connect local rounding errors with later spectral behavior.

\begin{figure}[t]
\centering
\includegraphics[width=\textwidth]{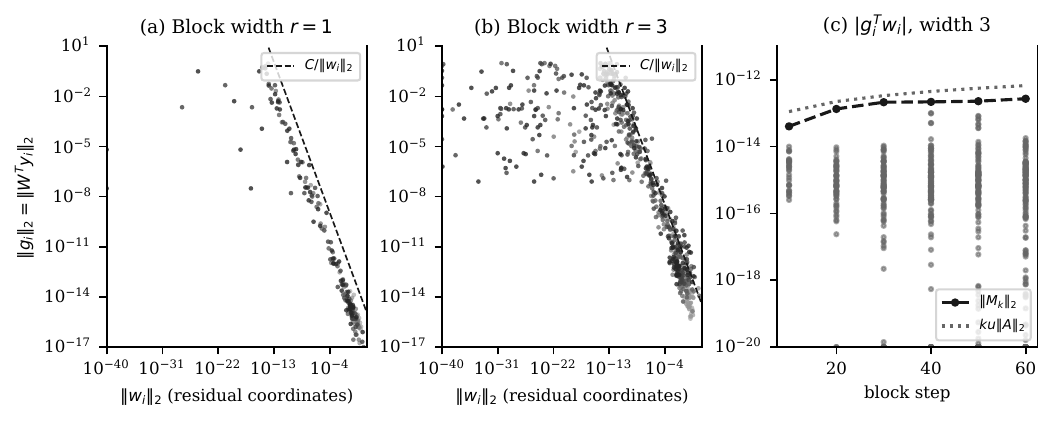}
\caption{Residual coordinates $w_i=HE_k^Ts_i$ and overlap
$g_i=W^TVs_i$ in binary64 block Lanczos with current-block re-projection
and a compact SVD, on a diagonal Strako\v{s} matrix
($n=300$, $\lambda_1=0.1$, $\lambda_n=100$, spacing parameter $\rho=0.9$).
Points are collected at $k=10,20,\ldots,60$ block steps.
(a) With width one, the full overlap satisfies the scalar bound.
(b) With width three, the full overlap can exceed that scalar bound.
Here $M_k$ denotes $M$ after $k$ blocks.  The curve uses a uniform
constant $C$, the maximum over sampled steps of $\|M_k\|_2$ plus the
largest computed Ritz-pair correction at that step.
(c) The directional inner products remain small for width three.}
\label{fig:blockpaige}
\end{figure}

Figure~\ref{fig:blockpaige} shows the difference between a directional
bound and a bound on the entire overlap.  In the scalar run,
$\|g_i\|_2\|w_i\|_2$ remains below the reference constant.  With width
three, this product can be much larger even though $|g_i^Tw_i|$ remains
small.  The components perpendicular to $w_i$ account for the difference.
For a computed eigenpair with coefficient residual
$t_i=Ts_i-\theta_i s_i$, the algebraic identity includes the correction
$t_i^TUs_i-s_i^TUt_i$.  The plotted reference constant includes the
magnitude of this correction.

\begin{corollary}[Ritz residuals and loss of orthogonality]
\label{cor:block-paige-consequences}
Under the hypotheses of Theorem~\ref{thm:block-paige}, suppose
$\rho_i>0$ and $|\hat w_i^Tg_i|\ge\tau>0$.  Then
$\rho_i\le\mu/\tau$.

For two Ritz pairs, suppose $\rho_i,\rho_j\le\rho$,
$\|g_i\|_2,\|g_j\|_2\le1$, and
$|\theta_i-\theta_j|\ge\gamma>0$.  Then
\[
 |y_j^Ty_i|
 \le\frac{2\rho+e(\|y_i\|_2+\|y_j\|_2)}{\gamma}
 \le\frac{2\rho+2\sqrt{k}\,e}{\gamma}.
\]
Every Ritz value also satisfies
\[
\lambda_{\min}(A)-\|\Delta\|_2\le\theta_i
 \le\lambda_{\max}(A)+\|\Delta\|_2,
\qquad
\operatorname{dist}(\theta_i,\sigma(A))\le\rho_i+\|\Delta\|_2,
\]
where $\Delta$ is the perturbation in Theorem~\ref{thm:block-global}.
\end{corollary}

\begin{proof}
The first two assertions follow from
\eqref{eq:block-paige-consequences} and $\|V\|_2\le\sqrt{k}$.
The last two bounds follow from the enlarged symmetric relation and
\eqref{eq:ritz-posteriori}.
\end{proof}

The corollary shows that a large overlap
along the unit residual coordinate requires a small boundary residual.
The physical residual is then bounded by $\rho_i+e$ before normalizing
$y_i$.  Components perpendicular to the residual coordinate enter through
the inter-block dynamics considered next.  Since the identities depend
on the collected recurrence, they also apply to the inexact products of
Corollary~\ref{cor:inexact-paige}.

\subsection{Subspace overlap and physical Ritz vectors}

To determine when a computed Ritz cluster represents redundant information,
we track the Gram matrix of the computed basis.  In the scalar case,
overlap with converged Ritz vectors underlies selective and partial
reorthogonalization \cite{Paige1976,ParlettScott1979,Simon1984Reorth}.
Grimes, Lewis, and Simon derived the corresponding block overlap recurrences
\cite{GrimesLewisSimon1994}.  Paige and Chang, Paige, and Titley-Peloquin
relate the structure of the lost orthogonality to independent and repeated
physical Ritz vectors \cite{Paige2019,ChangPaigeTitleyPeloquin2021}.
We now combine the block recurrence with a
basis-independent subspace test and a separate test in the original vector
space.

Let
\begin{equation}
 {\cal R}=AV-VT,
 \qquad G=V^TV,
 \qquad {\cal E}=TG-GT.
 \label{eq:ritz-gram-commutator}
\end{equation}
Since $A$ and $T$ are symmetric,
\begin{equation}
 {\cal E}=V^T{\cal R}-{\cal R}^TV.
 \label{eq:ritz-commutator-residual}
\end{equation}
The columns of ${\cal R}$ contain the local errors and final boundary term.
Fix an earlier physical Ritz subspace with orthonormal basis
$Y_0\in\R^{n\times d_0}$, where $1\le d_0\le n$.  Let
$\Theta_0=\Theta_0^T\in\R^{d_0\times d_0}$ and
$R_{Y_0}\in\R^{n\times d_0}$ satisfy
\[
 AY_0=Y_0\Theta_0+R_{Y_0}.
\]
Write $C_i=Y_0^TV_i\in\R^{d_0\times r_i}$ and let $F_i$ be the $i$th
block column of $F$.  Multiplying the normalized recurrence by $Y_0^T$
gives, for $1\le i<k$,
\begin{equation}
 C_{i+1}B_{i+1}
 =\Theta_0 C_i-C_iA_i-C_{i-1}B_i^T
   +R_{Y_0}^TV_i-Y_0^TF_i.
 \label{eq:block-overlap-recurrence}
\end{equation}
with $C_0B_1^T=0$ at the initial step.  The residual of the earlier
subspace and the local errors enter as the inhomogeneous terms.  They can
create overlap that subsequently propagates through the recurrence.
The following theorem turns a measured overlap into a count of nearby
Ritz values.

\begin{theorem}[Additional Ritz values from subspace overlap]
\label{thm:additional-ritz-values}
Let $1\le d\le m$, let $Z\in\R^{m\times d}$ satisfy $Z^TZ=I_d$, and let
$c\in\R$.  Use the symmetric matrix $T$ and Gram matrix $G$ defined above.
The
columns of $Z$ may be Ritz vectors from an earlier projected problem, padded
with zeros in the current coefficient space.  Define
\begin{equation}
 R_Z=(T-cI)Z,
 \qquad K_Z=Z^TGZ,
 \qquad W_Z=(I-ZZ^T)GZ.
 \label{eq:additional-ritz-definitions}
\end{equation}
Then
\begin{equation}
 (T-cI)W_Z={\cal E}Z+GR_Z-R_ZK_Z.
 \label{eq:additional-ritz-identity}
\end{equation}
Choose an integer $1\le t\le\operatorname{rank}(W_Z)$ and take $t$ leading
singular triplets of $W_Z$.  Write them as
\[
 W_ZP_W=Q_W\Sigma_W,
 \qquad Q_W^TQ_W=P_W^TP_W=I_t,
\]
where $Q_W\in\R^{m\times t}$, $P_W\in\R^{d\times t}$, and
$\Sigma_W\in\R^{t\times t}$ is diagonal with positive entries.
Set
\begin{equation}
 \varepsilon_Z=
 \left\|
 \left[\,
 R_Z\quad
 ({\cal E}Z+GR_Z-R_ZK_Z)P_W\Sigma_W^{-1}
 \right]
 \right\|_2.
 \label{eq:additional-ritz-radius}
\end{equation}
Then $T$ has at least $d+t$ eigenvalues, counted with multiplicity, in
\begin{equation}
 [c-\varepsilon_Z,c+\varepsilon_Z].
 \label{eq:additional-ritz-interval}
\end{equation}
Moreover,
\begin{equation}
 (VZ)^T(VQ_W)=P_W\Sigma_W.
 \label{eq:additional-ritz-physical-overlap}
\end{equation}
\end{theorem}

\begin{proof}
From $W_Z=GZ-ZK_Z$ and \eqref{eq:ritz-gram-commutator},
\[
 (T-cI)W_Z
 ={\cal E}Z+G(T-cI)Z-R_ZK_Z,
\]
which is \eqref{eq:additional-ritz-identity}.  The range of $W_Z$ is
orthogonal to the range of $Z$, so $Z^TQ_W=0$.  Since
$W_ZP_W\Sigma_W^{-1}=Q_W$, we have
\[
 (T-cI)Q_W
 =({\cal E}Z+GR_Z-R_ZK_Z)P_W\Sigma_W^{-1}.
\]
The matrix $[\,Z\ Q_W\,]$ has $d+t$ orthonormal columns, and its residual
for $T-cI$ is the matrix in \eqref{eq:additional-ritz-radius}.
Courant--Fischer applied to $(T-cI)^2$ proves
\eqref{eq:additional-ritz-interval}.  Finally,
$GZ=ZK_Z+W_Z$ gives
\[
 Z^TGQ_W=(GZ)^TQ_W=W_Z^TQ_W=P_W\Sigma_W,
\]
which proves \eqref{eq:additional-ritz-physical-overlap}.
\end{proof}

Taking $t=\operatorname{rank}(W_Z)$ uses every overlap direction.  A smaller
$t$ allows weak directions to be omitted.  For each choice the radius can
also be evaluated as
\[
 \varepsilon_Z=\|(T-cI)[\,Z\ Q_W\,]\|_2.
\]
This expression evaluates the bound directly from the selected orthonormal
trial directions.  For a fixed ordering of the singular vectors, the radii
are nondecreasing with $t$.
At a prescribed radius, one can therefore use the largest $t$ for which
the bound fits in that interval.

For an orthogonal matrix $O\in\R^{d\times d}$, changing the basis to $ZO$
changes $W_Z$ to $W_ZO$.  Its range,
rank, singular values, and norm are therefore independent of the chosen
basis.  The same is true of the selected left singular subspace when all
singular values above a fixed threshold are retained.  The matrix also has a
chronological interpretation.  Suppose $Z$
belongs to the projected problem after block
$j$, put $X_0=[\,V_1\ \cdots\ V_j\,]Z$, and pad $Z$ with zeros at later
steps.  If
\[
 G_\ell=[\,V_1\ \cdots\ V_\ell\,]^T
          [\,V_1\ \cdots\ V_\ell\,],
 \qquad
 W_{Z,\ell}=(I-ZZ^T)G_\ell Z,
\]
then, for $\ell\ge j$,
\begin{equation}
 W_{Z,\ell}^TW_{Z,\ell}
 =W_{Z,j}^TW_{Z,j}
  +\sum_{h=j+1}^{\ell}X_0^TV_hV_h^TX_0.
 \label{eq:accumulated-ritz-overlap}
\end{equation}
Indeed, the later block rows of $G_\ell Z$ are $V_h^TX_0$ and occupy
disjoint coefficient rows.  Hence
\eqref{eq:accumulated-ritz-overlap} records the accumulated overlap of later
Lanczos blocks with the earlier physical Ritz subspace.  Each added term is
positive semidefinite, so every singular value of $W_{Z,\ell}$ is
nondecreasing while this earlier subspace is held fixed.

Physical residuals and numerical rank determine how many independent
directions the additional Ritz values represent.  For a matrix $Y$ and
an absolute threshold $\eta\ge0$, define
\[
 \operatorname{rank}_\eta(Y)
 =\#\{i:\sigma_i(Y)>\eta\}.
\]
Singular values are padded with zeros when an index exceeds the smaller
matrix dimension.  Here $\eta$ is an absolute singular-value threshold.  A relative threshold
$\eta_{\rm rel}$ corresponds to $\eta=\eta_{\rm rel}\sigma_1(Y)$.
For $q$ normalized Ritz vectors satisfying a prescribed physical-residual
tolerance, define their ghost multiplicity as
$q-\operatorname{rank}_\eta(Y)$.  This quantity measures numerical
redundancy in the original vector space.  An exact multiple eigenvalue can
have several independent Ritz vectors, for which this redundancy is zero
at a threshold below their smallest singular value.

\begin{corollary}[A physical-rank test for Ritz ghosts]
\label{cor:physical-rank-ghosts}
Let
$J=[c-\varepsilon_Z,c+\varepsilon_Z]$ be the interval in
Theorem~\ref{thm:additional-ritz-values}, and suppose $J$ lies in an isolated
spectral interval $I$ of $A$.  Let $P_I$ be the spectral projector of $A$
onto $I$, and set $r=\operatorname{rank}(P_I)$.  Choose
$q=d+t$ orthonormal eigenvectors $z_1,\ldots,z_q$ of $T$ whose eigenvalues
$\theta_1,\ldots,\theta_q$ lie in $J$.  Suppose $Vz_i\ne0$, and put
\[
 y_i=\frac{Vz_i}{\|Vz_i\|_2},
 \qquad {\cal Y}=[\,y_1\ \cdots\ y_q\,].
\]
Let $\rho\ge0$ and $\gamma>0$.  If
\begin{equation}
 \|Ay_i-\theta_i y_i\|_2\le\rho,
 \qquad
 \dist(\theta_i,\sigma(A)\setminus I)\ge\gamma>0
 \label{eq:physical-rank-hypotheses}
\end{equation}
for $1\le i\le q$ and $q>r$, then
\begin{equation}
 \sigma_{r+1}({\cal Y})
 \le\sqrt q\,\frac{\rho}{\gamma}.
 \label{eq:physical-rank-bound}
\end{equation}
Consequently, every
$\eta\ge\sqrt q\,\rho/\gamma$ satisfies
$\operatorname{rank}_\eta({\cal Y})\le r$.  At least $q-r$ of the reported
directions are then redundant in the original vector space.
\end{corollary}

\begin{proof}
The projector $P_I$ commutes with $A$.  On
$\operatorname{range}(I-P_I)$, the smallest singular value of
$A-\theta_iI$ is at least $\gamma$.  Applying $I-P_I$ to the physical Ritz
residual gives
\[
 \|(I-P_I)y_i\|_2\le\rho/\gamma.
\]
The matrix $P_I{\cal Y}$ has rank at most $r$.  With $\|\cdot\|_F$ denoting
the Frobenius norm, the Eckart--Young theorem and the columnwise residual
bounds give
\[
 \sigma_{r+1}({\cal Y})
 \le\|(I-P_I){\cal Y}\|_2
 \le\|(I-P_I){\cal Y}\|_F
 \le\sqrt q\,\rho/\gamma.
\]
This proves \eqref{eq:physical-rank-bound} and the numerical-rank statement.
\end{proof}

Equations \eqref{eq:block-overlap-recurrence},
\eqref{eq:accumulated-ritz-overlap}, and
\eqref{eq:additional-ritz-interval} connect the successive parts of the
mechanism.  Ritz residuals and local errors can create overlap with an
earlier Ritz subspace.  The overlap accumulates in $W_Z$.  When the selected
overlap directions have small residuals for $T-cI$,
Theorem~\ref{thm:additional-ritz-values} guarantees additional nearby Ritz
values.  Corollary~\ref{cor:physical-rank-ghosts} then bounds the number of
independent physical directions they can represent.
Proposition~\ref{prop:rank-selection} shows that
components in the existing space can keep singular values above the SVD
threshold even when the residual has small components orthogonal to that space.

\subsection{A matrix-free softmax-Hessian calculation}
\label{sec:hessian-experiment}

Krylov approximations also occur in randomized low-rank matrix algorithms,
where block and single-vector constructions offer different ways to obtain
a useful spectral subspace
\cite{MuscoMusco2015,MeyerMuscoMusco2024,TroppWebber2023}.
We examined residual accuracy and physical independence using the damped
empirical softmax-loss Hessian of a trained digit classifier with $4{,}096$
fixed random features.  The
ten-class output layer had $40{,}970$ parameters, so storing its
dense binary64 Hessian would require $13.43$ GB.  The matrix-free operator
has the form $A=H_{\rm loss}+\lambda_{\rm damp}I$, where $H_{\rm loss}$ is
the empirical Hessian and $\lambda_{\rm damp}>0$ is the damping shift.
We requested the sixteen
largest eigenpairs using three fixed random seeds.  Every method was
subjected to the same post-run audit with absolute physical-residual tolerance
$2.5\times10^{-9}$ and relative independence tolerance
$\sqrt{u}=1.05\times10^{-8}$.  ARPACK, PRIMME, and LOBPCG are mature
matrix-free eigensolvers with different algorithms and safeguards
\cite{ARPACKUsersGuide1998,Knyazev2001,StathopoulosMcCombs2010}.  The
rank-aware block calculation used an SVD-truncated generalized Rayleigh--Ritz
projection with a basis capped at $64$ columns and accepted only independent,
residual-qualified candidates.  The unprotected recurrence used block width
two, an initial Householder QR, two projections against the current block,
and a compact SVD of the residual using LAPACK's DGESVD.  Stored basis blocks
and their operator images were used to compute the diagnostics.  These measurements
compare individual residual accuracy with the dimension of the returned
subspace.

\begin{center}
\small
\begin{tabular}{lcc}
\toprule
method & largest audited residual & independent directions among 16\\
\midrule
ARPACK & $(1.99\text{--}3.20)\times10^{-17}$ & 16\\
PRIMME & $(1.55\text{--}2.01)\times10^{-11}$ & 16\\
LOBPCG & $(1.18\text{--}1.88)\times10^{-9}$ & 16\\
rank-aware block calculation & $(1.42\text{--}2.11)\times10^{-9}$ & 16\\
unprotected residual-only recurrence & $(0.90\text{--}2.47)\times10^{-9}$ & 5\\
\bottomrule
\end{tabular}
\end{center}

Thus the three comparison solvers and the rank-aware calculation returned
sixteen residual-qualified, independent directions.  The unprotected block
recurrence also returned sixteen vectors satisfying the residual tolerance,
but those vectors spanned only five numerical dimensions.  The numerical-rank
test therefore separated residual accuracy from independence of the returned
directions.

At the first subspace accepted by the physical-residual test, we fixed an
orthonormal basis $Y_0$ for that subspace.  After each subsequent block step, we measured
$\|Y_0^TV_{i+1}\|_2$, recording its first crossing of $\sqrt u$.
This is the physical overlap in \eqref{eq:block-overlap-recurrence}.
The table records this crossing, the first numerical dependence in the
leading Ritz-vector set, and acceptance of sixteen vectors by their residuals.

\begin{center}
\small
\begin{tabular}{cccc}
\toprule
start & first measurable overlap & first dependent leading set
      & residual-only acceptance\\
\midrule
21921 & block 9 & block 17 & block 50\\
21922 & block 8 & block 16 & block 55\\
21923 & block 8 & block 16 & block 50\\
\bottomrule
\end{tabular}
\end{center}

In all three runs, the measured overlap preceded numerical dependence of
the leading set by eight block steps.  Acceptance of sixteen vectors by
their residuals occurred more than forty steps after the overlap crossing.
The overlap measurement records the return of an earlier converged
subspace, while the final rank test identifies the five independent
directions represented by the sixteen reported vectors.

These results build on the fixed-width finite-precision recurrence of
Guarracino, Perla, and Zanetti \cite{GuarracinoPerlaZanetti2006}, the
Greenbaum-type block continuation of \v{S}imonov\'a and Tich\'y
\cite{SimonovaTichy2025}, and Chen's enlarged scalar construction
\cite{Chen2026Simple}.  Meurant and Tich\'y also analyze rank-deficient
block-CG recurrences and exact deflation
\cite{MeurantTichy2026Dubrulle}.

In the subspace test, trial directions obtained
from $(I-ZZ^T)GZ$ give an eigenvalue count through the classical min--max
principle, and the physical residuals of the associated Ritz vectors and
their spectral separation bound the number of independent vectors in the
original space.  Together with
Proposition~\ref{prop:rank-selection}, this relates the observed Ritz
multiplicity to accumulated overlap and the rank selected for the next
block.  The local recurrence and enlarged symmetric model remain available
through rank changes under their stated error bounds.
\section{Formal certificates and reproducibility}
\label{sec:certificates}

The companion checks use exact rational simulation and formal proofs in
Lean 4. The simulators reproduce the specified evaluation orders and
round-to-nearest decisions for the steepest-descent two-cycle
(Theorem~\ref{thm:sd-cycle}) and the CG/Lanczos separation
(Theorem~\ref{thm:cg-lanczos-separation}). Checks at significand precisions
$11$, $24$, $53$, and $113$ cover the rounding decisions of the corresponding
IEEE binary formats when the stated exponent-range assumptions hold.
The zero-denominator construction in
Proposition~\ref{prop:cg-curvature-loss} is also checked for each precision
from $2$ through $256$. The general results follow from the proofs in the
text.

The formal proofs use Lean 4 and Mathlib \cite{Lean4,Mathlib2020}.
Each statement records its hypotheses. In \texttt{Problem215.lean}, we
prove the periodic two-cycle and the true-residual identities under the
stated round-to-nearest assumptions. For the CG/Lanczos separation,
\texttt{Problem216.lean} takes the stored coefficients and fourth iterate
as inputs. We prove positive definiteness of the original matrix, singularity
and inconsistency of the projected problem, and exactness of the fourth
iterate, together with the condition-number identity.
The algebraic and local rounding results for the zero-denominator example
are checked in \texttt{Problem217.lean}.

For the precision analysis, we formalize scalar perturbation estimates and
the induction for accumulated error in \texttt{ConditionedNstepBound.lean},
then establish the transfer from forward error to normalized residual and
backward error. The complementary steepest-descent estimates are checked
in \texttt{ScalarConvergence.lean}. Starting from a weighted Kantorovich
inequality for a finite spectrum, we derive the exact contraction in
eigenvector coordinates and combine the assumed local error bounds used in
the proof of Lemma~\ref{lem:sd-one-step}. Induction gives the geometric
estimate used in Theorem~\ref{thm:sd-positive}. The local-error constant for
the diagonal case is also verified. For the iteration-budget analysis, the same module
contains the Lagrange interpolation inequality used in Lemma~\ref{lem:intervals}.
We also prove that acceptance of the independently evaluated residual
test implies the matrix-only backward-error bound under the stated error
estimate for that test.

The enlarged symmetric problem is formalized in three modules.
In \texttt{BlockLanczosCore.lean}, we prove the Gram identities and the
algebra of the block-column construction. The indexing of variable-width
blocks and successive powers of the finite auxiliary operator are treated
in \texttt{BlockLanczosGlobal.lean}. We prove the spectral-norm estimates
for the enlarged symmetric construction in \texttt{BlockLanczosNorms.lean}.
For the exact-arithmetic projection in
Proposition~\ref{prop:exact-clusters-deflation},
\texttt{BlockProjection.lean} starts from the polynomial recurrence and
verifies the projected relation. Orthogonal decomposition of the residual
then gives the Schur-complement formulas for the Gram matrix and rank of
the next block.

In \texttt{BlockLanczosPaige.lean}, we formalize the identities and
spectral-norm bounds in Theorem~\ref{thm:block-paige}.
The same module includes the inter-block
overlap recurrence \eqref{eq:block-overlap-recurrence} for variable block
widths and the positive-semidefinite Gram increment obtained when a later
block is appended.
In \texttt{BlockPaigeFreedom.lean}, we verify the dimension of the affine
solution set of the Paige equations in Proposition~\ref{prop:dof}.
Once the overlaps are fixed, distinct Ritz values determine the
off-diagonal Gram entries uniquely.

The spectral conclusions are checked in \texttt{RitzMultiplicity.lean}.
Using the orthonormal trial directions and residual identities proved in
\texttt{BlockLanczosPaige.lean}, we establish the eigenvalue count in
Theorem~\ref{thm:additional-ritz-values}, including the stated radius of the
enclosing interval. Eigenvalues are counted with multiplicity.
We also prove the singular-value estimate in
Corollary~\ref{cor:physical-rank-ghosts}. Under its physical-residual and
spectral-separation assumptions, we obtain an upper bound on the numerical
rank of the lifted Ritz vectors and a lower bound on the number of
redundant directions in the original space.

The Lean sources, pinned dependencies, and build instructions are available
at \url{https://github.com/math-experiments/finite-precision-krylov-lean}.
All twelve modules build together with the pinned Lean and Mathlib versions.
The numerical source accompanying this revision includes the plotted data,
fixed matrix inputs, and the scripts used to regenerate the vector figures.
Exact rounding tests and floating-point experiments are recorded separately
so that their arithmetic models can be reproduced.

OpenAI Codex with GPT-5.6-Sol and Claude Fable assisted with literature
retrieval, proof exploration, numerical and formal checks, and manuscript
editing. Earlier exploratory runs of OpenEvolve \cite{OpenEvolve} used a
fixed, curated sequence of candidate responses. The author assumes full
responsibility for the content.
\section{Conclusions}
\label{sec:conclusions}

The examples and bounds in this paper describe how rounding changes a
short-recurrence method at several distinct stages. In the steepest-descent
two-cycle, the stored residual remains equal to the true residual, yet both
repeat indefinitely. Under the sufficient condition of
Theorem~\ref{thm:sd-positive}, the stored residual instead contracts
geometrically at every step for which the arithmetic hypotheses hold.
The sharper diagonal estimate makes the dependence on conditioning more
transparent. Together, these results distinguish a genuine periodic orbit
from the more familiar stagnation caused by a residual gap.

The CG/Lanczos example shows that algebraic equivalence permits different
computed outcomes. The first loss of local orthogonality is identified
exactly, and a residual estimate relates singularity of the projected matrix
to the resolution of the smallest eigenvalue. At a different level, a common
positive spectral enclosure gives comparison bounds for enlarged exact
models and, with the stated transfer estimates, for the computed quantities.
The numerical examples complement these bounds with similar convergence
histories even after the two sets of tridiagonal coefficients have separated.

The precision theorem gives a sufficient bit count and a computable
backward-error stopping test for a fixed dense evaluation order within
$n$ updates. Its exponent-range assumptions include the intermediate scalar
calculations. The interpolation estimate describes sensitivity to small
spectral intervals, while the fixed-input experiments show how the precision
found by the search decreases when the iteration budget is enlarged.
Working precision and iteration count must therefore be considered together
when interpreting these examples.

For block Lanczos, local orthogonalization and rank truncation lead to a
variable-block recurrence with controlled errors. The symmetric enlarged
construction retains its computed coefficients. The block Paige identity
then bounds overlap with a Ritz vector along the coordinate direction of
its residual. The inter-block recurrence describes how the remaining
overlap components evolve.
These relations hold for a collected local perturbation, whether its source
is rounding or an inexact operator application.

The Gram identities connect this overlap to spectral information. In exact
arithmetic, a residual component orthogonal to the existing space determines
how many new directions can be added. In a computed recurrence, accumulated
overlap can support additional nearby Ritz values. The physical-residual and
rank estimate bounds the number of corresponding vectors that are independent
at the stated numerical-rank tolerance.
This separates repeated eigenvalue approximations from repeated information.
In the Hessian calculation, sixteen residual-qualified vectors span only
five dimensions, while the rank-aware recurrence, ARPACK, PRIMME, and LOBPCG
return sixteen independent directions. Residual accuracy and subspace
dimension describe different parts of that outcome.

The resulting interpretation joins local error analysis to the evolution of
overlap and then to the information contained in a Ritz cluster. It gives a
mathematical basis for examining orthogonalization and deflation decisions
through both convergence and independence. Sharper long-run estimates for
the inter-block recurrence would further quantify when the overlap reaches
the level required by the eigenvalue-count theorem.

\begingroup
\small
\bibliographystyle{plainnat}
\bibliography{references}
\endgroup

\end{document}